\documentclass[11pt]{amsart}
\usepackage[margin=1.5in]{geometry}

\usepackage{adjustbox}
\date{\today}

\usepackage{amsmath}
\usepackage{amsfonts}
\usepackage{amssymb}
\usepackage{amsthm}
\usepackage{enumerate}
\usepackage{enumitem} 
\usepackage{color}
\usepackage{stmaryrd}
\usepackage{slashed}
\usepackage{relsize}
\usepackage{geometry}
\usepackage{subfigure}
\usepackage{caption}
\usepackage{tikz-cd}
\usepackage{mathtools}
\usepackage{adjustbox}
\usepackage{pdflscape}
\usepackage{tabularx}

\usepackage{xcolor}
\definecolor{darkgreen}{rgb}{0,0.50,0}
\definecolor{darkred}{rgb}{0.55,0,0}
\definecolor{darkblue}{rgb}{0,0,0.6}
\definecolor{darkteal}{rgb}{0.0, 0.25, 0.5}

\usepackage{caption}
\usepackage{framed}

\makeatletter
\newcommand{\labeltext}[2]{%
  \@bsphack
  \csname phantomsection\endcsname 
  \def\@currentlabel{#1}{\label{#2}}%
  \@esphack
}

\definecolor{color1}{RGB}{12, 12, 149}
\definecolor{color2}{RGB}{12, 32, 161}
\definecolor{color3}{RGB}{9, 146, 212}
\definecolor{color4}{RGB}{8, 184, 226}
\definecolor{color5}{RGB}{28, 253, 175}
\definecolor{color6}{RGB}{39, 255, 151}

\usetikzlibrary{automata,arrows,calc,positioning}

\usepackage{array}

\definecolor{darkgreen}{rgb}{0,0.50,0} 
\definecolor{darkred}{rgb}{0.55,0,0}
\definecolor{darkblue}{rgb}{0,0,0.6}

\usepackage[pdfborder={0 0 0},pagebackref,colorlinks,citecolor=darkgreen,linkcolor=darkred,urlcolor=darkblue]{hyperref}
\usepackage[capitalise]{cleveref}

\newtheorem{theorem}{Theorem}[section]
\newtheorem{corollary}{Corollary}[section]

\newtheorem{proposition}{Proposition}[section]

\newtheorem{lemma}{Lemma}[section]

\newtheorem{thmx}{Theorem}

\newtheorem{conjx}[thmx]{Conjecture}

\theoremstyle{definition}
\newtheorem{definition}{Definition}[section]

\newtheorem{convention}{Convention}[section]
\newtheorem{openQuestion}{Open Question}[section]

\newtheorem{example}{Example}[section]

\newtheorem{remark}{Remark}[section]

\counterwithin{figure}{section}
\counterwithin{table}{section}
\counterwithin{equation}{section}

\newtheorem{expectation}{Expectation}

\makeatletter
\let\c@corollary=\c@thm
\let\c@prop=\c@thm
\let\c@proposition=\c@thm
\let\c@theorem=\c@thm
\let\c@lem=\c@thm
\let\c@definition=\c@thm
\let\c@conj=\c@thm
\let\c@defn=\c@thm
\let\c@df=\c@thm
\let\c@example=\c@thm
\let\c@remark=\c@thm
\let\c@lemma=\c@thm
\let\c@sch=\c@thm
\let\c@convention=\c@thm
\let\c@equation=\c@thm
\let\c@observation=\c@thm
\let\c@openQuestion=\c@thm
\makeatother

\usepackage{verbatim}

\usepackage[normalem]{ulem}

\newcommand{\Z}{\mathbb{Z}}
\newcommand{\Q}{\mathbb{Q}}
\newcommand{\C}{\mathbb{C}}
\newcommand{\R}{\mathbb{R}}

\newcommand{\CC}{\mathcal{C}}

\newcommand{\Hom}{\operatorname{Hom}}
\newcommand{\Map}{\operatorname{Map}}

\newcommand{\End}{\operatorname{End}}

\newcommand{\Spin}{\operatorname{Spin}}

\newcommand{\Pin}{\operatorname{Pin}}
\newcommand{\rank}{\operatorname{rank}}
\newcommand{\Sq}{\operatorname{Sq}}
\newcommand{\id}{\operatorname{id}}

\newcommand{\kerv}{\operatorname{kerv}}

\newcommand{\String}{\operatorname{String}}
\newcommand{\Fivebrane}{\operatorname{Fivebrane}}
\newcommand{\BOrk}{B\operatorname{Or}_k}

\newcommand{\obj}{\operatorname{obj}}
\newcommand{\Gr}{\operatorname{Gr}}
\newcommand{\Emb}{\operatorname{Emb}}

\newcommand{\sVect}{\mathrm{sVect}}
\newcommand{\Vect}{\mathrm{Vect}}
\newcommand{\SK}{\mathrm{SK}}

\newcommand{\SKK}{\mathrm{SKK}}
\newcommand{\BOr}{B\!\operatorname{Or}}
\newcommand{\Or}{\operatorname{Or}}

\newcommand{\SKKxi}{\mathrm{SKK}^\xi}
\newcommand{\RP}{\mathbb{RP}}
\newcommand{\CP}{\mathbb{CP}}
\newcommand{\HP}{\mathbb{HP}}
\newcommand{\Bord}{\mathrm{Bord}}
\newcommand{\Cob}{\mathrm{Cob}}
\newcommand{\ol}{\overline}
\newcommand{\Cat}{\mathrm{Cat}}
\newcommand{\Gpd}{\mathrm{Gpd}}

\newcommand{\sline}{\mathrm{sline}}

\newcommand\rout{\bgroup\markoverwith
	{\textcolor{magenta}{\rule[.5ex]{2pt}{0.4pt}}}\ULon}

\usepackage{booktabs}

\usepackage{bm}
\usepackage{tablefootnote}

\usepackage[all,cmtip]{xy}

\author[R. S. Hoekzema]{Renee  S. Hoekzema}
\address{Department of Mathematics, Vrije Universiteit Amsterdam}
\email{r.s.hoekzema@vu.nl}
\author[L. Stehouwer]{Luuk Stehouwer}
\address{Department of Mathematics and Statistics, Dalhousie University}
\email{luuk@dal.ca}
\author[S. Vesel\'a]{Simona Vesel\'a}
\address{Universität Bonn and Max Planck Institute for Mathematics}
\email{vesela@math.uni-bonn.de}

\date{\today}

\title{SKK groups of manifolds and non-unitary invertible TQFTs}

\begin{document}

\begin{abstract}  
This work considers the computation of controllable cut-and-paste groups $\SKK^{\xi}_n$ of manifolds with tangential structure $\xi:B_n\to BO_n$.
To this end, we apply the work of Galatius-Madsen-Tillman-Weiss, Genauer and Schommer-Pries, who showed that for a wide range of structures $\xi$ these groups fit into a short exact sequence that relates them to bordism groups of $\xi$-manifolds with kernel generated by the disc-bounding $\xi$-sphere.
The order of this sphere can be computed by knowing the possible values of the Euler characteristic of $\xi$-manifolds. We are thus led to address two key questions: the existence of $\xi$-manifolds with odd Euler characteristic of a given dimension and conditions for the exact sequence to admit a splitting. 
We resolve these questions in a wide range of cases.

$\SKK$ groups are of interest in physics as they play a role in the classification of non-unitary invertible topological quantum field theories, which classify anomalies and symmetry protected topological (SPT) phases of matter. Applying our topological results, we give a complete classification of non-unitary invertible topological quantum field theories in the tenfold way in dimensions 1-5. 
\end{abstract}

\maketitle

\tableofcontents

\section{Introduction}

The study of cut-and-paste invariants of manifolds was initiated in \cite{skbook} by Karras, Kreck, Neumann and Ossa. 
Given a closed smooth manifold, one can cut it along a separating codimension 1 submanifold with trivial normal bundle and paste back the two pieces along a diffeomorphism of the boundary to obtain a new closed manifold, which we say is \emph{cut-and-paste equivalent} to the original.
Cut-and-paste groups, also known as $\SK$ groups of manifolds, are formed by quotienting the monoid of manifolds under disjoint union by this cut-and-paste relation. 

A more refined notion of cut-and-paste equivalence, called $\SKK$ for ``schneiden und kleben kontrollierbar'', or, ``controllable cutting and pasting'', remembers the diffeomorphisms that were used to glue the boundaries. More precisely, we obtain the $\SKK$ groups by quotienting the monoid of manifolds by the four-term \ref{SKKrelations}:\labeltext{SKK relation}{SKKrelations}
\[M_1\cup_\phi M'_1+ M_2\cup_\psi M'_2\sim_{\SKK} M_1\cup_\psi M'_1+ M_2\cup_\phi M'_2,\] where $M_1, M'_1, M_2, M'_2$ are compact manifolds with the same boundary. Rearranging terms, we see that this corresponds to requiring that the difference between two cut and paste equivalent manifolds depends only on the gluing maps $\phi$ and $\psi$, as illustrated in \cref{fig:SKK}.
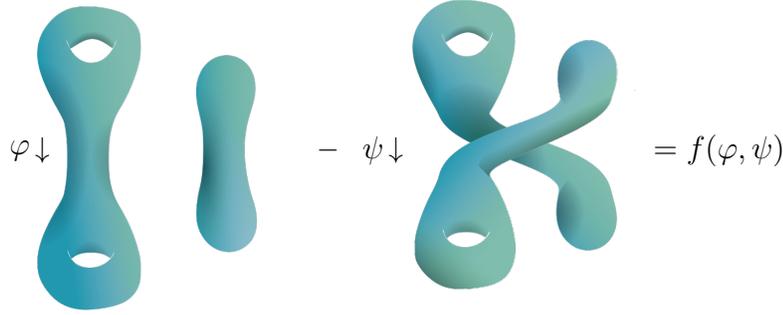
\begin{figure}[ht!]
	\centering
	    \begin{tikzpicture}[scale=1]
        \node at (0,0) {\includegraphics[width=13cm]{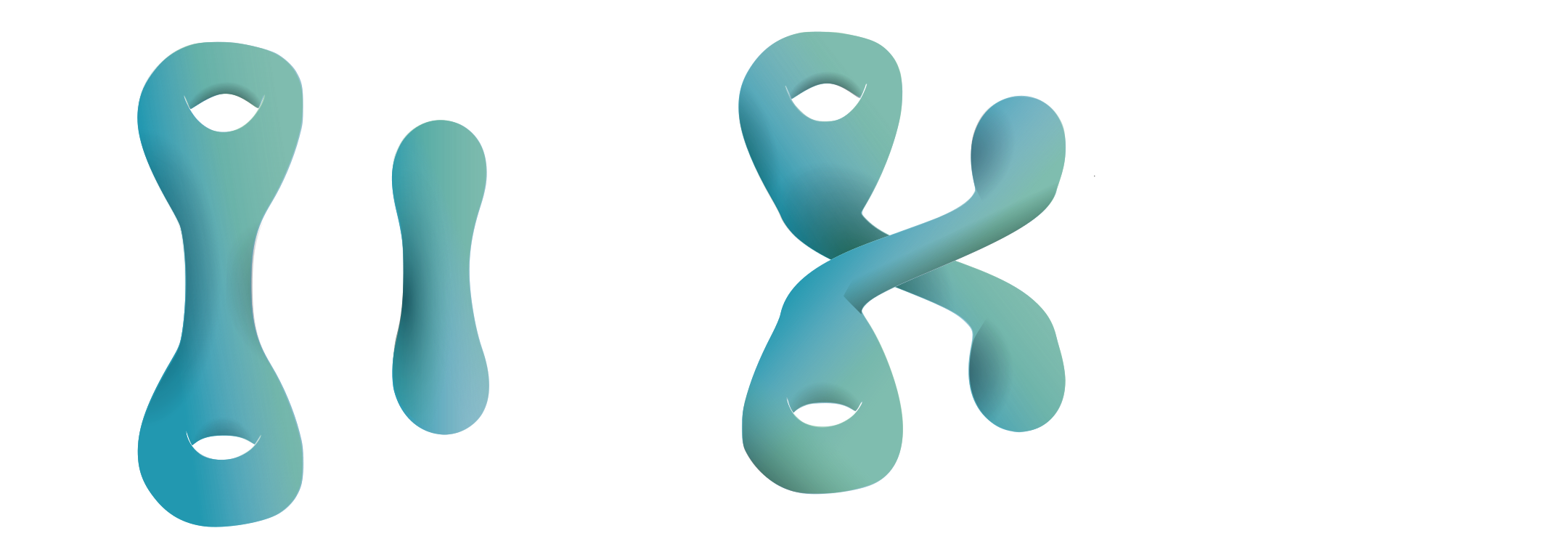}};
        \node at (-5.6,0) {\large{$\varphi$}};
        \node at (-5.3,0) {$\downarrow$};
        \node at (-1.5,0) {$-$};
        \node at (-0.9,0) {\large{$\psi$}   };
        \node at (-0.6,0) {$\downarrow$};
        \node at (3.7,0) {\large{$=f(\varphi,\psi)$}};
\end{tikzpicture}
	\caption{A pair of manifolds glued along a diffeomorphism $\phi$ differs in $\SKK$ from the pair glued along a different diffeomorphism $\psi$ by a manifold $f(\varphi,\psi)$ which depends only on $\phi$ and $\psi$.}
	\label{fig:SKK}
\end{figure}

The $\SKK$ groups of unoriented and oriented manifolds, where for the latter we require that the gluing maps are orientation-preserving, were shown in \cite{skbook} to correspond to Reinhart vector field bordism groups, and moreover they arise as fundamental groups of the unoriented resp. oriented cobordism categories \cite{ebert}. 
Interest in the computation of $\SKK$ groups further grew when it was shown that they play a role in the classification of invertible topological quantum field theories (TQFTs) \cite{KST,schommerpriesinvertible, Carmenitft}, which are important in physics for the classification of anomalies and topological phases of matter \cite{freed2014anomalies, monnieranomalies, freedhopkins}. 
Mathematically, TQFTs can be thought of as functors from a cobordism category into a linear category, as explained further in \cref{sec:physics}.

The systems studied in physics often come with intrinsic \textit{symmetries}, for example in the classification of condensed matter systems known as the tenfold way \cite{kitaev2009periodic}. These symmetries can be interpreted mathematically as \textit{tangential structure} on the manifolds. 
In this work, we treat tangential structures in a very general (not necessarily stabilised) framework, as lifts up to homotopy of the tangent bundle of an $n$-manifold along a map $\xi_n$ from a space $B_n$ to $BO_n$.
Our goal in this paper is to calculate the $\SKK^{\xi}$ groups of $\xi$-manifolds up to $\SKK^{\xi}$-equivalence, where now the manifolds have $\xi$-structures and the gluing maps are $\xi$-preserving, for many interesting and frequently arising $\xi$-structures.

The main tool for computation will be a short exact sequence relating $\SKK$ groups to bordism groups $\Omega_n^\xi$, with kernel given by the subgroup of $\SKK$ generated by the (disc-bounding) sphere:
\begin{equation}\label{eq:introskk}
\begin{tikzcd} 0 \ar[r] &\langle S_b^n \rangle_{\SKKxi_n} \arrow[r]&\SKKxi_n\arrow[r]&\Omega_n^{\xi} \ar[r] & 0.\end{tikzcd}
        \end{equation}
We will refer to this short exact sequence as \ref{SKKseq}. It was established for (un-)oriented manifolds in \cite{skbook}, and is reproven, by different methods for a general setting of twice stabilised $\xi$-structures in \cref{subsec:Genauer}, see also \cite{reutterschommerpries}. 

The structure $\xi$ needs to be once stabilised (defined in at least dimension $n+1$) in order for both the left and the right-hand term of \cref{eq:introskk} to be well defined.
In even dimensions $n$, the kernel $\langle S_b^n \rangle$ is $\Z$ because the Euler characteristic is an $\SKKxi$ invariant. In odd dimensions, given the slightly more restrictive condition of $\xi$ being twice stabilised, the kernel is either 0 or $\Z/2$ depending on whether an odd Euler characteristic $\xi$-manifold does or does not exist in dimension $n+1$ respectively. The first main question of this paper is therefore 
\begin{leftbar}
    \textit{For which tangential structures $\xi$ and dimensions $n$ does there exist a $\xi$-manifold with odd Euler characteristic?}
\end{leftbar}
This question was partially answered in work by the first author for $k$-orientable tangential structures \cite{ReneeEulerChar}, the results of which we extend and use here.
The question will be (at least partially) resolved for $\Pin^\pm$ manifolds by the current authors in \cite{PinPaper}, and we resolve other cases, of interest to physics, in Sections \ref{sec:otherStructures} and \ref{sec:physics}. 

The \ref{SKKseq} was shown to admit a splitting $\SKK_n \rightarrow \langle S^n \rangle$ for orientable manifolds in any dimension
\cite{skbook,ebert}. 
The second and most important question posed in this paper is to investigate when a splitting can be defined for more general tangential structures.
\begin{leftbar}
    \textit{For which tangential structures $\xi$ and dimensions $n$ does the \ref{SKKseq} admit a splitting?}

\end{leftbar}

We have divided our results in this direction into separate sections dealing with the odd-dimensional versus even-dimensional case (Sections \ref{sec:SKKOdd} and \ref{sec:SkkEven} respectively) because of their different nature.

\subsection{Splitting results for odd-dimensional \texorpdfstring{$\SKKxi$}{SKK} groups} 
For odd dimensions, we restrict ourselves to twice stabilised $\xi$ structures, see \cref{subsec:twiceStabXi}. This is automatically satisfied if our tangential structure is stable, i.e.\ arises from a map $\xi \colon B\to BO$, which is the case for most well-known tangential structures. 
If an $(n+1)$-dimensional $\xi$-manifold $M$ with odd Euler characteristic $\chi(M)$ exists, then $\SKKxi_n \cong \Omega^\xi_n$. 
However, for many $\xi$-structures, there are even dimensions in which odd Euler characteristic $\xi$-manifolds do not exist.
For example, oriented manifolds can only have odd Euler characteristic in dimensions $4k$, where there exists, for example $\C \mathbb{P}^{2k}$.
By restricting to more highly connected tangential structures, the dimension where such manifold exists gets restricted to $8k$ for $\Spin$ (quaternionic projective planes) and $16k$ for $\String$ manifolds (octonionic projective plane). 

If an $(n+1)$-dimensional $\xi$-manifold with odd Euler characteristic does not exist, then we have a splitting problem for the sequence 
\begin{equation}\label{seq:introZ2}
\begin{tikzcd} 0 \ar[r] &\Z/2 \arrow[r]&\SKK^{\xi}_n\arrow[r,"p_\xi"]&\Omega_n^{\xi} \ar[r] & 0.\end{tikzcd}
        \end{equation}  
Firstly we obtain the following if and only if statement for a map to give rise to a splitting of \cref{seq:introZ2}.

\begin{thmx}\label{thmIntro:iffInvarian}[\cref{{thm:iffForMfldInvariant}}]
Let be $\kappa$ a homomorphism $\mathcal{M}^\xi_n \to \Z/2,$ for $\mathcal{M}^\xi_n$ the monoid of closed $n$-dimensional $\xi$-manifolds under disjoint union.
Then $\kappa$ induces a splitting of \cref{seq:introZ2} if and only if for all $(n+1)$-dimensional $\xi$-manifolds $W$ with boundary $Y$ we have $$\kappa(Y)=  \chi(W)\mod 2.$$
\end{thmx}

Our main candidate for a splitting is the Kervaire semi-characteristic over $\Z/2$, defined for a $(2k+1)$-dimensional manifold $M$ as
\[\kerv_{\Z/2}(M)=\sum_{i=0}^k \dim_{\Z/2} H_i(M;\Z/2) \pmod 2.\]
The following is one of our main results.
\begin{thmx}[\cref{thm:OddSplittingTopWuclass}]
If for every closed $(n+1)$-dimensional $\xi$-manifold $W$  the top Wu class $v_{\frac{n+1}{2}}(W)$ vanishes, then there is a split short exact sequence \begin{equation*}\begin{tikzcd} 0 \ar[r] &\Z/2\arrow[r]&\SKK^{\xi}_n\arrow[r,"p_{\xi}"]\arrow[l,bend right,pos=0.4,,"\kerv_{\Z/2}"']&\Omega_n^{\xi} \ar[r] & 0.\end{tikzcd}
        \end{equation*}

More generally, if $\xi$-manifolds are orientable in $F$ homology for some field $F$, then the Kervaire semi-characteristic over $F$ is a splitting if and only if for every $(n+1)$-dimensional $\xi$-manifold $W$, possibly with boundary, the image of the map
$H_{\frac{n+1}{2}}(W;F) \xrightarrow{j_*} H_{\frac{n+1}{2}}(W,\partial W; F)$ has even dimension.
\end{thmx}

Below, we give a summary of splitting results for well-known tangential structures.

\begin{thmx}[Results in odd dimensions]\label{thmx:odd_dim}
We obtain the splitting results for $\SKK$ groups of manifolds with $\xi$-structures and odd dimensions listed in \cref{tab:SKKResults}.

    \begin{table}[h!]\centering
\adjustbox{scale=0.75}{
\begin{tabular}{|l|l|l|l|}
	\hline
	 odd dimensions $n$ with:& $\SKKxi_n \cong \Omega^\xi_n$ &  $\Z/2 \rightarrow \SKKxi_n \rightarrow \Omega^\xi_n $ 
  & unknown  \\ \cline{1-1}
	$\xi$-structure &  &  split by $\kerv_{\Z/2}$ &  \\
 \hline
 $(BO)_{> 0}\simeq BO\simeq \BOr_0$&$2\mid (n+1)$&-&-\\
  $(BO)_{> 1}\simeq BSO\simeq \BOr_1$&$4\mid (n+1)$& other odd $n$&-\\
  $(BO)_{> 2}\simeq B\Spin\simeq \BOr_2$&$8\mid (n+1)$&other odd $n$&-\\
$(BO)_{> 4}\simeq B\String$&$16\mid(n+1)$&other odd $n$&-\\
$\BOr_3$&$16\mid(n+1)$&other odd $n$&-\\
&&&\\
 $\BOr_k$ ($k$-oriented), $k\geq 4$&?& odd $n$ such that $2^{k+1}\nmid n+1$&$2^{k+1}\mid n+1$\\
 $(BO)_{> b}, b\geq 8$&?& For $k = \phi(b)$ ($k \approx \frac{b}{2}$, see Cor. \ref{cor:connectedCoverAndkOrientable}):&
  
  \\
  && odd $n$ such that $2^{k+1}\nmid n+1$
  
  & $2^{k+1}\mid n+1$\\
  &&&\\
$s_{n+k} \colon *\to BO_{n+k}$, $2\leq k\leq \infty$&-&all odd $n$&-\\
$\Pin^+$&$n\equiv 3,7 \pmod{8}$&$n\equiv 5\pmod{8}$&$n\equiv 1\pmod{8}$\\
&&& split for $n=1$,
\\
&&& not by $\kerv_{\Z/2}$\\
$\Pin^-$&$n\equiv 1,5,7 \pmod{8}$&$n\equiv 4\pmod{8} $&- \\
 \hline
\end{tabular}
}\caption{Results about odd-dimensional $\SKK$ groups summarising \cref{Table:k-orientable}, \cref{lm:Inheritance}, \cref{thm:mainthmforOddDimensionKOrient}, \cref{cor:SKKOfBOb}, \cref{SKKStablyFramed}, and \cref{prop:PinpmSKK}.
Here $(BO)_{> b}$ refers to the $b$-parallellisable tangential structure (\cref{subsec:parallelisable}), $B \Or_k$ refers to the $k$-orientable structure (\cref{def:k-orientable}) and $s_{n+k}$ is the $k$th-stabilised framing (\cref{subsec:FramingStr}).
}\label{tab:SKKResults}
\vspace{-0.8cm}
\end{table}
\end{thmx}

\noindent For example, in the case of $\Pin^+$ manifolds, \cref{tab:SKKResults} tells us that we have the following:
\begin{enumerate}[label=(\roman*)]
    \item For $n\equiv 3,7\pmod{8}$ we have an isomorphism
$$p_{\Pin^+} \colon \SKK_n^{\Pin^+}\cong  \Omega_n^{\Pin^+}.$$ 
\item  For $n\equiv 5\pmod{8}$ we have an isomorphism
$$(\kerv_{\Z/2},p_{\Pin^+}) \colon \SKK_n^{\Pin^+}\xrightarrow{\cong} \Z/2\times \Omega_n^{\Pin^+}.$$
\item For $n\equiv 1\pmod 8$ we do not know $\SKK_n^{\Pin^+}$, because we do not know the kernel of the map $\SKK_{8k+1}^{\Pin^+}\to  \Omega_{8k+1}^{\Pin^+}$, since it is unknown whether there exists an odd Euler characteristic $\Pin^+$ manifold of dimension $8k+2$.
\end{enumerate}
We know, by computation, that there is no odd Euler characteristic $\Pin^+$ manifold in dimensions 2 and 10 \cite{PinPaper}. Furthermore, in the case that such a manifold does not exist in dimension $8k+2$, we show that the Kervaire semi-characteristic, or any other invariant which depends only on the manifold and not on the $\Pin^+$ structure, cannot give a splitting of the SKK sequence for $\Pin^+$ in dimension $8k+1$ (\cref{prop:PinPlus8kplus1}). 
In dimension 1, the extension problem is trivial since $\Omega_1^{\Pin^+}=0$ and hence $\SKK_1^{\Pin^+}\cong\Z/2$ (see \cref{Ex:pindim1}).

For $k$-orientable manifolds with $k\geq 4$, in particular 8-parallelisable manifolds, it is unknown whether any odd Euler characteristic manifolds exist (if they do they would live in dimensions multiples of $2^{k+1}$, see \cref{section:k-orientability}) or whether the sequence splits if this is not the case.

In \cref{tab:SKKResults} we list the splittings by $\kerv_{\Z/2}$, which is the most general case.
Other splittings are possible, for example for oriented manifolds $\kerv_\Q$ is also a (potentially different) splitting in dimensions $4k+1$ whenever $\kerv_{\Z/2}$ is, but may not be a splitting in dimensions $4k+3$, see
\cref{rm:differentKervaire}.

Inspired by \cref{thmx:odd_dim}, we conjecture the following:

\begin{conjx}
\label{conjecture}
    For every twice stabilised structure $\xi$ and every odd dimension $n$, the \ref{SKKseq} is split.
\end{conjx}

    Without the assumption that $\xi$ is twice stabilised, the kernel of the map $\SKK^{\xi}_n\to \Omega_n^{\xi}$ in odd dimensions does not have to be $0$ or $\Z/2$, and there exists a tangential structure for which this is the case and the SKK sequence is known to not split. This is discussed in \cite{KST}.

\subsection{Splitting results for even-dimensional \texorpdfstring{$\SKKxi$}{SKK} groups} 
For $n$ even, the \ref{SKKseq} looks like 
\begin{equation}\label{seq:introZ}
\begin{tikzcd} 0 \ar[r] &\Z\arrow[r]&\SKK^{\xi}_n\arrow[r]&\Omega_n^{\xi} \ar[r] & 0,\end{tikzcd}
        \end{equation}
for any once stabilised structure $\xi \colon B_{n+1}\to BO_{n+1}$. We give a complete criterion for splitting for $\xi$ structures satisfying a mild finiteness condition:
\begin{thmx}[\cref{thm:slittingSKKevenEvenChi} and \cref{thm:torsionWithOddEuler}]
    Let $n$ be even and $\xi_{n+1}\colon B_{n+1} \rightarrow BO_{n+1}$ a (once stabilised) tangential structure.
    If there exists a torsion class $[M] \in \Omega^{\xi}_{n}$ with $\chi(M)$ odd, then \cref{seq:introZ} does not split. 
    
    Moreover, if $B_{n+1}$ has finitely generated homology in all degrees, then the converse holds: if all manifolds $M^n$ with odd Euler characteristic have infinite order in $\Omega^\xi_n$, then the same sequence splits non-canonically (i.e. depending on a choice of generating manifolds for $\Omega^{\xi}_{n}$). 
   
    In the special case where every $n$-dimensional closed $\xi$-manifold has even Euler characteristic, we get an explicit splitting
    \[\SKK^{\xi}_n\xrightarrow{\cong}\Z\times \Omega_n^{\xi}\quad [M] \mapsto (\chi(M)/2, [M]).\]
\end{thmx}

\begin{thmx}[Results in even dimensions]

In even dimensions the \ref{SKKseq} \cref{seq:introZ} has the following splitting status:

\begin{enumerate}[label=(\roman*)]
    \item For $\xi=BO$ the SKK sequence does {\bf not} split for any even $n$ by \cref{cor:unoriented_even_notsplit}.
    \item $\xi=BSO$  the SKK sequence  \begin{enumerate}
        \item splits by $\frac{\chi-\sigma}{2}$ for $n = 0 \pmod{4}$~\cite{ebert}.
        \item splits by $\frac{\chi}{2}$ for $n = 2 \pmod{4}$~\cite{ebert}.
    \end{enumerate}  
    \item For any orientable $\xi$ e.g. $\BOr_k$ for $k>0$, $(BO)_{> b}$ for $b\geq 1$ or $s_{n+k}$ for $k\geq 1$ the SKK sequence 
    \begin{enumerate}
        \item splits by $\frac{\chi-\sigma}{2}$ for $n = 0 \pmod{4}$ by \cref{cor:orientableeven}.
        \item splits by $\frac{\chi}{2}$ for $n = 2 \pmod{4}$ by \cref{cor:orientableeven}.
    \end{enumerate} 
    \item For $\Pin^+$  the SKK sequence
    \begin{enumerate}
        \item does {\bf not} split for $n= 0,4 \pmod{8}$ by \cref{cor:pin_even_dim}.
        \item splits by $\frac{\chi}{2}$ for $n=6 \pmod{8}$ and $2, 10$ by \cref{thm:slittingSKKevenEvenChi}.
    \end{enumerate}
    (The general case of $n= 2 \pmod{8}$ is currently open).
    \item For $\Pin^-$  the SKK sequence
    \begin{enumerate}
        \item does {\bf not} split for $n= 0,2,6 \pmod{8}$ by  \cref{cor:pin_even_dim}.
        \item splits by  $\frac{\chi}{2}$ for  $n=4 \pmod{8}$ by \cref{thm:slittingSKKevenEvenChi}.
    \end{enumerate}
\end{enumerate}
\end{thmx}

\subsection{Implications for physics}

It is standard lore in the physics literature that unitary invertible field theories are classified by bordism invariants \cite{freedhopkins, yonekura}.
Mathematically, unitarity of TQFTs can be defined using dagger categories, see \cref{def:unitaryTQFT}.

Even though unitarity is a core principle in QFT, non-unitary QFTs also play an important role in the physics literature. 
In condensed matter physics, for instance, non-Hermitian systems have attracted attention for their unique physical properties~\cite{nonhermitian}.
In mathematical physics, important examples of non-unitary QFTs include topological twists of supersymmetric theories, which violate the spin-statistics theorem.
Moreover, dualities in the non-invertible symmetries (such as the classic Kramers-Wannier symmetry) naturally give rise to non-unitary operators~\cite{shao2023s, Li_2023}. Non-unitary invertible TQFTs also appear in the analysis of global anomalies in non-unitary quantum field theories~\cite{chang2021exotic,TachikawaYonekura}.
In \cref{sec:physics}, we explain that invertible TQFTs that are not necessarily unitary are classified by $\SKK$ groups, a perspective we learned from \cite{KST}.

To take symmetries of the TQFT into account, we take our bordisms to come equipped with certain tangential structures.
A specific collection of ten symmetry groups called the tenfold way \cite{altlandzirnbauer,kitaev2009periodic} is of special interest in condensed matter physics. 
Our theorems above together with some additional calculations in \cref{sec:otherStructures} and \cref{sec:physics} provide a complete list of $\SKKxi$ groups classifying not necessarily unitary invertible TQFTs in spacetime dimensions up to $5$ for the $10$ induced tangential structures and several others in \cref{tab:physics} and \cref{subsection:classification_physics}. 
A key role is played by what we call Kervaire TQFTs, which are certain invertible TQFTs with partition function equal to $Z(M) = (-1)^{\kerv_{\Z/2}(M)}$.
Kervaire TQFTs for the Kervaire semi-characteristic over $\Q$ have appeared in the literature \cite[Example 6.15]{freed2019lectures}, but we are unaware of previous work which emphasises the fact that the Kervaire semi-characteristic over $\Z/2$ gives a source of TQFTs for a much more general class of structures and dimensions.
In conclusion, our results generalise the computations in \cite[Section 9.3]{freedhopkins} from unitary to non-unitary invertible TQFTs.

Our computations of $\SKK$ groups classify `discrete' invertible TQFTs.
For some applications in physics, continuous invertible TQFTs (\cref{def:continuousITQFT}) are more relevant, see \cite[Sections 5.3 and 5.4]{freedhopkins} for further discussion.
In \cref{th:ctsITFT}, we prove that even though torsion elements of continuous and discrete theories agree, other elements are related to bordism groups with two vector fields instead.

\subsection*{Acknowledgements}
The authors would like to thank Matthias Kreck, Stephan Stolz and Peter Teichner for many discussions about their unpublished work \cite{KST}.
Additionally, we are grateful to Arun Debray, Daniel Galvin, Andrea Grigoletto, Achim Krause, Lukas M\"uller, Thomas Rot, Carmen Rovi, Julia Semikina, Alba Send\'on Blanco, Yuji Tachikawa and Rolf Vlierhuis for useful discussions. RH was supported by the Dutch Research Council (NWO) through the grant VI.Veni.212.170.
LS is supported by the Atlantic Association for Research in the Mathematical Sciences and the Simons Collaboration on Global Categorical Symmetries. 
LS is grateful to Dalhousie University for providing the facilities to carry out his research.
SV is supported by the Deutsche Forschungsgemeinschaft (DFG, German Research Foundation) under Germany’s Excellence Strategy – EXC-2047/1 – 390685813 and the Max Planck Institute for Mathematics. The authors thank the Vrije Universiteit Amsterdam and the Max Planck Institute for Mathematics in Bonn for their hospitality. 

\section{Tangential structures and the parity of the Euler characteristic}\label{section:tangential}

\subsection{Background on tangential structures on manifolds and cobordisms}\label{subsec:TangStruc}

The theory of smooth manifolds becomes more interesting if we consider an extra geometric structure on them, such as orientations, spin structures et cetera.
From a homotopy-theoretic point of view, such structures can be defined as lifts of the classifying map of the tangent bundle to some space $B_n$.
Let $BO_n$ be the classifying space of the orthogonal group of $\R^n$.
 Let $BO = BO_\infty$ be the colimit over $n$ induced by group homomorphisms $O_k \to O_{n}$ given by
\[
A \mapsto 
\begin{pmatrix}
A & 0\\
0 & 1
\end{pmatrix}.
\]

\begin{definition}
    An $n$\emph{-dimensional tangential structure} consists of a pointed topological space $B_{n}$ and a pointed map $\xi_{n}\colon B_{n} \to BO_{n}$. 
 A \emph{stable tangential structure} consists of a pointed topological space $B$ and a pointed map $\xi\colon B \to BO$.
    \end{definition}

\begin{definition}\label{def:xiStructure}
Let $\xi_n\colon B_n \to BO_n$ be an $n$-dimensional tangential structure.
    If $X \to BO_{n}$ classifies an $n$-dimensional vector bundle $E \to X$, a homotopy filling the triangle
\[
\begin{tikzcd}
     & B_n \ar[d, "\xi_n"]
    \\
    X \ar[r]\arrow[r, ""{name=U, below}]{} \ar[ru] & BO_n  \arrow[Rightarrow,from=1-2,to=U,shorten=4.5mm]
\end{tikzcd}
\]
is called a $\xi_{n}$-structure on $E$. By a $\xi_n$-structure on an $n$-dimensional manifold $M$ we mean $\xi_n$-structure on its tangent bundle $TM$ represented by some fixed classifying map $M\xrightarrow{TM}BO_n$
\[
\begin{tikzcd}
     & B_n \ar[d, "\xi_n"]
    \\
    M \ar[r,"TM"]\arrow[r, ""{name=U, below}]{} \ar[ru] & BO_n  \arrow[Rightarrow,from=1-2,to=U,shorten=4.5mm]
\end{tikzcd}.
\]
By a $\xi_n$-manifold we mean a manifold with a manifold with a chosen $\xi_n$-structure. 
\end{definition}

\begin{remark}\label{rm:fibrationOrNot}
    It is customary to assume that the map $\xi_n\colon B_n \to BO_n$ is a fibration. 
    In this case, it can be assumed without loss of generality that a $\xi$-structure is a map $M\to B_n$ which lifts the map to $BO_n$ on the nose.
\end{remark}

The following way of lowering the dimension of tangential structures will be employed throughout the paper.

\begin{definition}\label{def:homotopyPullbackB_k}
For any $n\leq \infty$, let $\xi_n \colon B_n\to BO_n$ be a tangential structure. For any $k \leq n$, we define $B_k$ to be the homotopy pullback of $\xi_n\colon B_{n} \rightarrow BO_{n}$ along the stabilisation map $BO_k \to BO_{n}$:
\[
\begin{tikzcd}
B_k \arrow[rd,phantom, "\lrcorner", very near start]\arrow[d, "\xi_k"] \arrow[r] & B_n \arrow[d, "\xi_n"] \\
BO_k \arrow[r]                   & BO_n .                 
\end{tikzcd}
\]
In particular, if we have a map $\xi \colon B \to BO$, we can define $\xi_k$ for every $k$ by homotopy pullback.
\end{definition}

If $\xi_n\colon B_{n} \to BO_{n}$ is a tangential structure, a $\xi_k$-structure on a $k$-dimensional vector bundle $E$ for some $k \leq n$ is equivalent to a $\xi_n$-structure on $E \oplus \underline{\R}^{n-k}$, see \cref{lm:stableEquiv}.
This justifies the following abuse of notation: we sometimes refer to $\xi_k$-structures on a $k$-dimensional vector bundle for some $k 
\leq n$ simply as a $\xi_n$-structure (or even $\xi$-structure).
In practice, a $k$-dimensional tangential structure $\xi_k \colon B_k\to BO_k$ often comes from a stable tangential structure:

\begin{definition}\label{def:stabilOFTangentialStrucutre}
     For $k\leq n\leq \infty$ a \emph{stabilisation} of $\xi_k\colon B_k\to BO_k$ is a tangential structure $\xi_n \colon B_n \to BO_n$ such that $\xi_{k}$ is the homotopy pullback of $\xi_n$ along $BO_{k} \to BO_n$ as in \cref{def:homotopyPullbackB_k}.
Similarly we say a structure $\xi_k\colon B_k\to BO_k$ is $i$ times stabilised {(\it once stabilised, twice stabilised,...)} if there exists a $(k+i)$-dimensional stabilisation  $\xi_{k+i}\colon B_{k+i}\to BO_{k+i}$.
\end{definition}

Strictly speaking, a stabilisation of a tangential structure is data, see \cref{ex:twoDiffStabilisationsOfStructures}.

\begin{example}
Some of the most commonly considered $\xi$-structures on a manifold stem from the Whitehead tower $\{(BO)_{\geq k}\}_k$ of connective covers of $BO$, where $(BO)_{\geq k}$ has the property that $\pi_i (BO_{\geq k}) = \pi_i (BO)$ for $i\geq k$ and 0 below. The first few connective covers are known under special names
\[BO\leftarrow BSO\leftarrow B\Spin \leftarrow B\String\leftarrow B\Fivebrane\leftarrow\cdots\]
The following well-known obstruction theoretic properties tell us whether manifolds admit a tangential structure lifting to the first stages of the Whitehead tower:
\begin{itemize}
    \item Every manifold $M$ has a canonical $BO$ structure.
    \item A manifold $M$ admits an $SO$-structure (orientation) if and only if the first Stiefel-Whitney class $w_1(M)$ vanishes.
    \item A manifold $M$ admits a $\Spin$ structure if and only if it admits an $SO$ structure and $w_2(M)$ vanishes.
    \item A manifold $M$ admits a $\String$ structure if and only if it admits a $\Spin$ structure and $\frac{1}{2}p_1(M)=0$.
    \item A manifold $M$ admits a $\Fivebrane$ structure if and only if it admits a $\String$ structure and $\frac{1}{6}p_2(M)=0$ \cite{fivebrane09}.
\end{itemize}
\end{example}

Next, we want to define $(n+1)$-dimensional bordism with $\xi_{n+1}$-structure between closed $n$-manifolds with $\xi_n$-structures.
For this, we will need to require compatibility between the $\xi_n$-structure on a boundary with the $\xi_{n+1}$-structure of the bordism.
The main subtlety that comes up is that we have to make a choice of in- versus outgoing normal vector field:

\begin{remark}\label{rem:ConventionMfldWithBoundary}
Given an $n$-dimensional manifold $M$ with boundary, a choice of normal vector field at the boundary induces a vector bundle isomorphism $TM|_{\partial M} \cong \underline{\R} \oplus T \partial M$. 
This normal vector field always exists and there are two homotopy classes of such isomorphisms, corresponding to the choice of in- versus outward normal.
Since a $\xi$-structure on $\underline{\R} \oplus T \partial M$ is equivalent to a $\xi$-structure on $T\partial M$, this in particular implies that a $\xi$-structure on $M$ and a choice of normal direction induces a $\xi$-structure on its boundary.
\end{remark}

\begin{convention}
\label{conv:xiStructureOnABoundary}
    If $M$ is an $n$-dimensional $\xi$-manifold with boundary, we take the $\xi$-structure on $\partial M$ corresponding to the outward-pointing normal. 
\end{convention}

The normal vector should get reversed for in- versus outgoing parts of the bordism, which we can express using orientation reversal:

\begin{definition}
\label{def:orientationreversal}
    Let $\xi_{n+1} \colon B_{n+1}\to BO_{n+1}$ be a tangential structure and let $M$ be a closed $n$-dimensional $\xi$-manifold. 
    Then $TM \oplus \underline{\R}$ has a canonical $\xi_{n+1}$-structure.
    We obtain a new $\xi_{n+1}$-structure by composing the previous homotopy with the self-homotopy of $TM \oplus \R$ induced by reflection in the $(n+1)^{st}$ coordinate $\id_{TM} \oplus - \id_{\underline{\R}}$.
    Using the fact that $\xi_{n+1}$-structures on $TM \oplus \R$ are equivalent to $\xi_n$-structures on $TM$, we obtain a $\xi_n$-structure $\overline{M}$ on $M$ that we will call the \emph{orientation reversal of $M$}.
\end{definition}

    We refer to \cref{rem:orientationreversaldetails} for details on the above definition.

\begin{definition}\label{def:xiBordism}
  Let $\xi_n \colon B_n \to BO_{n+1}$ be a tangential structure and let $M_0,M_1$ be (possibly empty) closed $n$-dimensional manifolds with $\xi$-structure.
    An \emph{$(n+1)$-dimensional $\xi$-bordism from $M_0$ to $M_1$} consists of an $(n+1)$-dimensional compact manifold $W$ with $\xi$-structure together with 
    \begin{enumerate}[label=(\roman*)]
        \item a splitting of $\partial W$ into two components $\partial W = \partial_{in} W \sqcup \partial_{out} W$;
        \item  $\xi$-diffeomorphisms $\overline{M_0} \cong \partial_{in} W$ and $M_1 \cong \partial_{out} W$.
    \end{enumerate} 
    We say that $M_0$ is \emph{$\xi$-bordant} to $M_1$ if there exists a $\xi$-bordism from $M_0$ to $M_1$.
\end{definition}

Here, a $\xi$-diffeomorphism is a diffeomorphism $f$ equipped with a datum specifying how the $\xi$-structures get transported under $df$, see \cref{def:xidiffeoTetrahedron}. 
Note that a manifold $W$ with boundary $M$ according to \cref{conv:xiStructureOnABoundary} is a bordism from $\varnothing$ to its boundary, or equivalently a bordism from $\ol{M}$ to $\varnothing$.

Next, we will define bordism groups for manifolds with $\xi$-structure.
In the generality we are working in, there is a subtlety: for a general (nonstable) tangential structure $\xi\colon B_{n+1} \to BO_{n+1}$, it might not be the case that being bordant is a symmetric relation:

\begin{remark}\label{rm:BordantIsNotSymmetric}
    Suppose $n = 1$ and $B_{n+1} = *$.
    Then a $\xi$-structure on a circle is a trivialisation of $T S^1 \oplus \underline{\R}$.
    There are $\Z$-many such framings and the boundary of the two-dimensional disc induces the framings corresponding to $\pm 1 \in \Z$ depending on whether we take it to be induced by the in- or outpointing normal vector on the boundary.
    In other words, the disc defines bordisms $\varnothing \to S^1$ and $S^1 \to \varnothing$, but these circles have different framings.\footnote{Another way to see that they must have different framings is to note that if they were the same their composition would give a framing of $S^2$, which does not exist.}
    Therefore, in defining the bordism groups we need to quotient out by the equivalence relation generated by bordism, i.e. we impose that $Y \sim Y'$ in the framed bordism group if and only if there exists a string of bordisms
\[\begin{tikzcd}
        Y\ar[r,"X_1"]& Y_1& Y_2 \ar[l,"X_2"']\ar[r,"X_3"]&\dots &Y'\ar[l,"X_n"'].
    \end{tikzcd}
  \]
\end{remark}

This motivates the following definition:

\begin{definition}
Let $\xi_{n+1}\colon B_{n+1} \to BO_{n+1}$ be a tangential structure.
    The \emph{$\xi$-bordism group} $\Omega^\xi_n$ of dimension $n$ is defined to be the set of closed $n$-dimensional $\xi$-manifolds modulo the equivalence relation generated by $(n+1)$-dimensional $\xi$-bordism. 
    The group operation is given by disjoint union.
\end{definition}

\begin{remark}
Given an $(n+1)$-dimensional tangential structure, we can only define the bordism group up to dimension $n$.
\end{remark}

The bordism groups are indeed groups, with the empty $n$-manifold being the neutral element.
Indeed, by definition of bordism and orientation reversal, the cylinder $M \times [0,1]$ on a $n$-dimensional $\xi$-manifold $M$ defines a bordism from $M \sqcup \overline{M}$ to $\varnothing$, therefore $[\overline{M}]$ is the inverse of $[M]$.

Given $\xi_n\colon B_n \rightarrow BO_n$, a $\xi_n$-structure on $S^n$ may or may not exist.
    For example, if $B_n$ is contractible, a $\xi_n$-structure is a framing of $TM$, and any sphere that is not also the underlying manifold of a Lie group (which happens only in dimensions 0, 1, 3 and 7) does not admit a framing.

\begin{definition}\label{def:boundingSphere}
    Let $\xi_{n+1}\colon B_{n+1} \to BO_{n+1}$ be a tangential structure. 
    Then the \emph{disc-bounding $n$-sphere} $S^n_b$ (or shortened to \emph{bounding sphere}) is the $\xi_{n+1}$-structure on $S^n$ defined uniquely by restricting the canonical $\xi_{n+1}$-structure on $D^{n+1}$ to the boundary (in accordance to the convention of the boundary being outward-pointing as specified in the \cref{conv:xiStructureOnABoundary}).
\end{definition}

Here, the canonical $\xi$-structure on $D^{n+1}$ consists of the map to the basepoint of $B_{n+1}$ and a (contractible) choice of nullhomotopy of the tangent bundle $D^{n+1} \to BO_{n+1}$.
The notion of bounding sphere in dimension $n$ only makes sense if we can talk about $(n+1)$-dimensional $\xi$-structures.

The following example of a bounding sphere additionally demonstrates the non-uniqueness of stabilisations of tangential structures, as discussed in \cref{def:stabilOFTangentialStrucutre}. 

\begin{example}\label{ex:twoDiffStabilisationsOfStructures}
    Consider $B_n = B\Spin_n$ and $B_n' = BSO_n \times B\Z/2$ with $\xi_n'$ its projection map to the first factor.
    Then we have $B_1 \cong B_1' \cong B\Z/2$ as tangential structures and so a $\xi_1$-structure on a one-dimensional manifold is the same as a $\xi'_1$-structure, which is a double cover.
    However, the $\xi$-bounding circle is the anti-periodic (connected) double cover of $S^1$, while the $\xi'$-bounding circle is the periodic (disconnected) double cover of $S^1$.
\end{example}

\begin{remark}\label{rm:boundingCircle}
    The term bounding sphere is slightly misleading because there can be other $\xi$-structures on $S^n$ that are not $\xi$-diffeomorphic to $S^n_b$ but are still trivial in the bordism group. 
    
    For a concrete example of this phenomenon, consider $B = B\Pin^+$ in dimension $n = 1$ (see \cref{def:Pin}).
    We have that $B\Pin^+_1 \cong B\Z/2 \times B\Z/2$, with the map to $BO_1$ projection onto one of the factors. 
    We see that a $B\Pin^+_1$ structure on a circle corresponds to a $\Z/2$ bundle, of which there are two, the period and the anti-periodic (M\"obius) bundle.
    Both of these are induced by $\Spin$ structures on the circle, and the anti-periodic $\Spin$ circle bounds the disc (\cref{ex:twoDiffStabilisationsOfStructures}).
    
    However, contrary to the $\Spin$ case, we have $\Omega^{\Pin^+}_1 = 0$ \cite{kirbytaylor} and so the periodic circle also bounds a two-dimensional $\Pin^+$-manifold, which happens to be the M\"obius strip.
\end{remark}

\begin{definition}\label{def:cob_cat}
\label{def:cob}
Let $\xi\colon B_n \to BO_{n}$ be a tangential structure.
    The \emph{$n$-dimensional $\xi$-bordism category} $\Cob^\xi_{n-1,n}$ is the category in which
    \begin{itemize}
        \item objects $(n-1)$-dimensional closed manifolds with $\xi$-structure;
        \item morphisms from $Y_1$ to $Y_2$ are $\xi$-bordisms up to $\xi$-diffeomorphism relative boundary.
    \end{itemize}
\end{definition}

    In order to make \cref{def:cob} rigorous, one needs to provide the $\xi$-structure on the composition of two bordisms.
Note that $\Cob^\xi_{n-1,n}$ is a symmetric monoidal category under disjoint union.
We refer to \cite{milnorhcobordism, tillmann1996classifying, kock2004frobenius} for details.

\begin{definition}[See also \cref{def:reversible}]\label{def:reversible_text}
    A category $\mathcal{C}$ is called \textit{reversible} (at every object) if whenever there is a morphism $f\colon X \to Y$, there also exists a morphism $f'\colon Y \to X$.
\end{definition}

The category $\Cob^\xi_{n-1,n}$ is in most cases reversible, but not always. A counterexample is given by the $\Cob^\xi_{1,2}$ for $\xi$ the 2-dimensional framing (see \cref{rm:BordantIsNotSymmetric}).
If $\xi$ is twice stabilised with respect to $n-1$, then $\Cob^\xi_{n-1,n}$ is guaranteed to be reversible because the bordisms admit orientation-reversal, see \cref{prop:turningBordismsAround}.

\subsection{Twice stabilised tangential structures}\label{subsec:twiceStabXi}

Let $\xi_n\colon B_n\to BO_n$ be a tangential structure.  In \cref{subsec:shortExactSeq}, we derive a short exact sequence (\ref{SKKseq}) relating the $\SKK_n^{\xi_n}$ group and the bordism group $\Omega_n^{\xi_n}$ of $n$-dimensional manifolds with a $\xi$-structure. For the bordism group $\Omega_n^{\xi_n}$ to be defined, our $\xi$-structure needs to be {\it once stabilised} with respect to $n$ (see \cref{def:stabilOFTangentialStrucutre}). In the case that $n$ is even this is a sufficient assumption to prove that the \ref{SKKseq} holds and that the kernel of the obvious map $\SKK_n^{\xi}\to \Omega_n^{\xi}$ is $\Z$. In the case that $n$ is odd, we assume that $\xi_n$ is {\it twice stabilised} in order to prove the \ref{SKKseq}. Under this assumption, we will see that the kernel of the map $\SKK_n^{\xi}\to \Omega_n^{\xi}$ is either $0$ or $\Z/2$ for $n$ odd (\cref{thm:TheExactSequence}.) From \cite{KST} we know that the latter is not true if we omit a higher stabilisation requirement in odd dimensions. However, their work shows that the weaker assumption that the sphere $S^{n+1}$ admits a $\xi$-structure and the category $\Cob^\xi_{n,n+1}$ is reversible (see \cref{def:reversible_text}) suffice to prove the \ref{SKKseq} and the surgery lemma
(\cref{lm:BordismEulerCharInSKK}), and with that our main result \cref{thm:iffForMfldInvariant} is still valid in this case. 
In the current work, we work with the twice stabilised assumption since \cite{KST} is currently unpublished.

\subsection{The parity of the Euler characteristic of a manifold}\label{subsec:EulerChar}

In \cref{sec:SKK} we will need to determine possible parity of Euler characteristic of manifolds admitting a given structure $\xi\colon B_n \to BO_{n}$. This will help us with calculation of the group $\SKK^{\xi}_n$.

This subsection proves the following basic fact.

\begin{lemma}
    Let $M$ be a closed manifold of dimension $2k$. Assume that the top Stiefel-Whitney class $w_{2k}(M)$ or the top Wu class $v_k(M)$ vanishes. Then the Euler characteristic $\chi(M)$ is even.
\end{lemma}

Recall that for an $n$-dimensional closed manifold $M$, the top Stiefel-Whitney class $w_n(M)$ relates to the parity of the Euler characteristic of the manifold in the sense that
\[\langle w_n(M),[M]\rangle\equiv\chi(M) \pmod{2},\]
where $[M]\in H_n(M;\Z/2)$ is the mod 2 fundamental class of the manifold.

For $n=2k$ even and $M$ a connected manifold, the cup product gives rise to a non-degenerate intersection form on the middle-dimensional cohomology 
\begin{align*}
   \lambda \colon H^k(M;\Z/2)&\times H^k(M;\Z/2) \rightarrow \Z/2\\
   (x, y) &\mapsto \langle x y,[M]\rangle.
\end{align*}

Since the cup square on $H^k(M;\Z/2)$ is represented by cupping with the top Wu class $v_k$, i.e. $x^2\equiv v_k x \in H^{2k}(M;\Z/2)\cong \Z/2$ for any $x \in H^k(M;\Z/2)$, we have that the form $\lambda(x,x)=0$ precisely if $v_k=0$.

\begin{definition}
    A $\Z$ or $\Z/2$ valued bilinear form $\lambda$ is called even if $\lambda(x,x)$ is even for every $x$.
\end{definition}
A non-degenerate even intersection form over $\Z/2$ is necessarily even ranked, i.e. the underlying vector space is even-dimensional.
By Poincar\'e duality we have that $rk_{\Z/2}H^k(M;\Z/2)\equiv \chi(M) \pmod{2}$. So we get that if $v_k=0$ then $\chi(M)$ is even. 
Indeed, the Wu formula dictates $w_{2k}=v_k^2$ in the cohomology of a manifold, so $v_k=0$ implies $w_{2k}$ vanishing, although the first condition is strictly stronger than the second, see the following example. 

\begin{example}\label{example_klein}
The Klein bottle has non-vanishing $v_1=w_1$, but has $w_2=0$.
\end{example}

\subsection{Relative Wu and Stiefel-Whitney classes}\label{sec:RelativeCharClasses}
The purpose of this section is to summarise some results about Wu and Stiefel-Whitney classes for manifolds with boundary.

Let $M$ be a manifold, possibly with boundary. Recall that the $\Z/2$ cohomology of the infinite Grassmannians $BO$ is generated as a ring by the Stiefel-Whitney classes. The total Stiefel-Whitney class is $w=1+w_1+w_2+\cdots$. Let $TM \colon M\to BO$ be the classifying map of the tangent bundle of $M$; for manifolds with boundary, this can be defined by restricting the tangent bundle of a double $DM$.
Then the total Stiefel-Whitney class of $M$ is $w(M)=TM^*(w)$. 

Recall that for $M$ a closed manifold, the total Wu class of the manifold $M$ is defined as $v(M)=1+v_1(M)+v_2(M)+\cdots$, where $v_{k}(M)$ is defined through the requirement that 
\[\langle x\smile v_k,[M]\rangle=\langle\Sq^{k}(x),[M]\rangle\quad \forall\; x\in H^{n-k}(M;\Z/2).\]
Stiefel-Whitney and Wu classes for closed manifolds come together in the formula
\[\Sq(v)=w,\]
where $\Sq=\Sq^0+\Sq^1+\Sq^2+\cdots$ is the total Steenrod square.

In a similar sense, we can now define Wu classes for manifolds with boundary. 

\begin{definition}[{\cite[§7, pg 532]{Kervaire57}}]\label{def:Wuclassboundary}
Let $M$ be a manifold with boundary.  
Given an integer $k$, let $$f\colon H^{n-k}(M,\partial M;\Z/2)\to \Z/2$$ be the homomorphism given by $x\mapsto \langle\Sq^{k}(x),[M,\partial M]\rangle$. Under the following composition of isomorphisms 
\[\Hom(H^{n-k}(M,\partial M;\Z/2),\Z/2)\cong H_{n-k}(M,\partial M;\Z/2)\cong H^{k}(M;\Z/2),\]
define the \textit{(absolute) Wu class} $v_{k}\in H^{k}(M;\Z/2) $ to be the image of the homomorphism $[f]\in\Hom(H^{n-k}(M,\partial M;\Z/2),\Z/2)$ in $H^{k}(M;\Z/2)$.
\end{definition}
We then have for $x \in H^{n-k}(M,\partial M;\Z/2)$, \[\langle\Sq^{k}x,[M,\partial M]\rangle=\langle x,PD(v_{k})\rangle=\langle x ,v_{k}\frown [M,\partial M]\rangle=\langle x\smile v_{k},[M,\partial M] \rangle.\] The total Wu class of a manifold is again given as the sum $v(M)=1+v_1+v_2+v_3+\cdots$. 
\begin{remark}
In \cite{Kervaire57} there is also a notion of \textit{relative} Wu class that lives in the relative cohomology $H^{k}(M,\partial M;\Z/2)$. This definition requires the vector bundle to be trivial on the boundary. We will not use this definition.
\end{remark}

The following is a generalisation of \cite[Lemma 7.3]{Kervaire57}.

\begin{lemma}\label{lem:KervaireInducedMap}
Let $M_1$ and $M_2$ be manifolds with boundary along with an identification $\partial M_1=\partial M_2$. Denote $W=M_1\cup M_2$. Then the inclusion
\[\iota\colon M_1 \hookrightarrow W\]
induces a map of the Wu classes $\iota^* v(W)=v(M_1)$.
\end{lemma}
\begin{proof}
    Let $e\colon  (M_1,\partial M_1)\to (W,M_2)$ denote the map of pairs.
    Note that $e$ satisfies excision.  
    Consider the following diagram
\[\adjustbox{scale=0.77}{
\begin{tikzcd}
H^*(M_1;\Z/2) \arrow[d,"\cong"]  \arrow[d,"-\cap {[M_1,\partial M_1]}"']                              & H^*(W;\Z/2) \arrow[l,"\iota"] \arrow[d]\arrow[d,"-\cap {e_*[M_1,\partial M_1]}"'] & H^*(W;\Z/2) \arrow[d,"\cong"]\arrow[d,"-\cap {[W]}"'] \arrow[l, equal] \\
{H_{n-*}(M_1,\partial M_1;\Z/2)} \arrow[r,"\cong"] \arrow[r,"e"'] \arrow[d,"\cong"] & {H_{n-*}(W,M_2;\Z/2)} \arrow[d,"\cong"] & H_{n-*}(W; \Z/2) \arrow[d,"\cong"] \arrow[l]                       \\
{\Hom(H^{n-*}(M_1,\partial M_1;\Z/2),\Z/2)} \arrow[r,"\cong"] \arrow[r,"e"'] & {\Hom(H^{n-*}(W,M_2;\Z/2),\Z/2)} & {\Hom(H^{n-*}(W;\Z/2),\Z/2).} \arrow[l]               
\end{tikzcd}
}
\]
The top squares commute by naturality of the cap product and the bottom squares commute by naturality of the universal coefficient sequence.

The Wu classes $v(M_1)\in H^*(M_1;\Z/2)$, $v(W)\in H^*(W;\Z/2)$ are preimages of the following classes
\begin{align*}
    \left<\Sq^k(-),[M_1,\partial M_1] \right>&\in \Hom(H^{n-*}(M_1,\partial M_1;\Z/2),\Z/2)\\
    \left<\Sq^k(-),[W] \right>&\in \Hom(H^{n-*}(W;\Z/2),\Z/2)
\end{align*}
respectively.
From the naturality of the Steenrod squares, we can then deduce that $\iota(v(W))=v(M_1)$.
\end{proof}

As a corollary we have:
\begin{corollary}
    For a manifold $M$ possibly with boundary we have
\[\Sq(v(M))=w(M).\]
\end{corollary}
\begin{proof}
    The closed case is classical. The case with boundary follows from \cref{lem:KervaireInducedMap}:
    \[Sq(v(M))=Sq(\iota^*v(DM))=\iota^*Sq(v(DM))=\iota^*(w(DM))=w(M).\qedhere \]
\end{proof}

For the future, we record the following:
\begin{corollary}\label{cor:closedWuboundaryWu}
Let $\xi\colon B_{n+2}\to BO_{n+2}$ be a tangential structure. Then if all closed $(n+1)$-dimensional $\xi$-manifolds for $n$ odd have vanishing top Wu class $v_{\frac{n+1}{2}}$, then for all $(n+1)$-dimensional $\xi$-manifolds with boundary, the top relative Wu class (see \cref{def:Wuclassboundary}) vanishes as well.
\end{corollary}
\begin{proof}
     Let $M$ be an $(n+1)$-dimensional $\xi$-manifold with boundary. Then there exists a manifold $M'$ with boundary $\ol{\partial M}$ (see \cref{prop:turningBordismsAround}). Then by assumption $v_{\frac{n+1}{2}}(M\cup M')$ vanishes. By \cref{lem:KervaireInducedMap} we get $v_{\frac{n+1}{2}}(M)=0$.
\end{proof}

\subsection{\texorpdfstring{$k$}{k}-orientability}\label{section:k-orientability}
The following sequence of tangential structures plays a central role in this paper.

\begin{definition}[$k$-orientability]\label{def:k-orientable}
Let $\BOr_k$ be the homotopy fibre of the map $$BO\xrightarrow{(w_{2^0},w_{2^1},\cdots ,w_{2^{k-1}})} \prod_{i=0}^{k-1}K(\Z/2,2^i).$$ Manifolds with a $\BOr_k$-structure are called {\it k-orientable}. 

\end{definition}
We sometimes use $\Or_k$ instead of $\BOr_k$ in the superscript like in $\Omega^{\Or_k}_n$ to make it consistent with the classical notation $\Omega^{O}_n,\Omega^{SO}_n$ etc.

A manifold is $k$-orientable if and only if $w_i(M)=0$ for $1\leq i<2^k$, since the vanishing of Stiefel-Whitney classes degrees $2^0,2^1, \cdots 2^{k-1}$  ensures the vanishing of all Stiefel-Whitney classes up to degree $2^k-1$. The concept of $k$-orientability was introduced by the first author in \cite{ReneeEulerChar}. 
Every manifold $M$ is $0$-orientable, i.e. $BO = \BOr_0$.
A manifold is 1-orientable if and only if it is orientable, i.e. $BSO = \BOr_1$.
A manifold is 2-orientable if it has vanishing $w_1,w_2$ and this is equivalent to having a spin structure, so $\BOr_2=B\Spin$.

A 3-orientable manifold is {\it not} the same as a manifold with a $\String$ structure. Every $\String$ manifold has vanishing $w_1,\cdots ,w_4$, which implies that there is a map $B\String\to \BOr_3$ over $BO$, hence every $\String$ manifold is 3-orientable (this will be used in \cref{ex:BstrigSplitting}). The converse however is not true. An example of a manifold that is 3-orientable but not $\String$ is given by $\mathbb{CP}^3$ (see \cite{douglas2011homological}).
Generally, a lift of the tangent bundle to the $k^{th}$ {\it non-trivial} connective cover of $BO$ (occurring at dimensions 0,1,2,4 mod 8) implies that a manifold is $k$-orientable, as will be discussed below. 
It was shown in \cite{ReneeEulerChar} that $k$-orientable manifolds have the property that many Wu classes vanish.

\begin{theorem}[\cite{ReneeEulerChar} Theorem 5.2]\label{morewus}
Let $M^n$ be an $n$-dimensional manifold that is $k$-orientable, then Wu classes $v_\ell$ vanish for all $\ell$ such that $2^k \nmid \ell$.
\end{theorem}

Following the reasoning in \cref{subsec:EulerChar} we obtain the following Corollary of \cref{morewus}.
\begin{corollary}[\cite{ReneeEulerChar} Corollary 5.3]\label{thm:k-orientEulerChar}
A $k$-orientable manifold $M$ has an even Euler characteristic unless its dimension is a multiple of $2^{k+1}$. 
\end{corollary}

\begin{table}[h!]
\begin{tabular}{|clcc|}
	\hline
	$\BOr_k$ & corresponds to & \bf{dimensions with} 	& \bf{example $k$-orientable}\\
	& & \bf{odd $\chi$ possible} & \bfseries{manifolds with odd $\chi$} \\
	\hline
	0 & any manifold & $2m$ & $\mathbb{RP}^{2m}$\\
	1 & orientable manifolds & $4m$ & $\mathbb{CP}^{2m}$\\
	2 & spinnable manifolds & $8m$ & $\mathbb{HP}^{2m}$\\
	3 & implied by stringable & $16m$ & $(\mathbb{OP}^2)^m$\\
	4 & implied by fivebraneable & $32m$ & (unknown) $\mathcal{X}^{32m}$ \\	\hline
\end{tabular}
\caption{Possible dimensions with odd Euler characteristic $k$-orientable manifolds, with example manifolds if known.}
\label{Table:k-orientable}
\end{table}
Implications of this result are summarised in \cref{Table:k-orientable}.
In particular, whether there exists a manifold $\mathcal{X}^{32m}$ with odd Euler characteristic that is $k$-orientable for $k\geq 4$, is an open question. If it does, its dimension is a multiple of 32. Note that in particular, any 8-connected manifold will be 4-orientable.

\begin{openQuestion}\label{OpQues:DoesSuchOddEulerCharExist}
Does there exist a 4-orientable manifold $\mathcal{X}^{32m}$ with odd Euler characteristic? 
More generally, does there exist a $k$-orientable $2^{k+1}m$-dimensional manifold for $k\geq 4$ with odd Euler characteristic?
\end{openQuestion}
This question was posed in \cite{ReneeEulerChar}, and discussed further in \cite{ReneeRosenfeld}.
At the moment of writing it remains open, and as a consequence, some questions in this paper will remain unresolved.

We can state an analogous theorem to \cref{morewus} for manifolds with boundary, whose proof is a direct application of \cref{cor:closedWuboundaryWu}.
\begin{theorem}\label{lm:ReneeKOrientabilityAndWuClassInRelativeCase}
Let $M^n$ be an $n$-dimensional manifold, possibly with boundary, that is $k$-orientable, then the Wu classes $v_l$ vanish for all $l$ such that $2^k \nmid l$.
\end{theorem}

\begin{remark}
It is important to stress that, while we can generalise the results about vanishing Wu classes to the relative setting, this does not imply that the parity of the Euler characteristic is constrained for $k$-orientable manifolds with boundary, i.e. we do not have a relative version of \cref{thm:k-orientEulerChar}. This is because the top Stiefel-Whitney class does not correspond to the parity of the Euler characteristic for a manifold with boundary. Indeed, any even-dimensional disc $D^n$ is contractible and therefore admits a $\Or_k$-structure for any $k$, but it has Euler characteristic 1.
\end{remark}

\subsubsection{Relationship between \texorpdfstring{$k$}{k}-orientability and \texorpdfstring{$b$}{b}-parallelisability}\label{subsec:parallelisable}

Define $(BO)_{>b} = (BO)_{\geq b+1}$ to be a space over $BO$ with vanishing homotopy groups below $b+1$ and the given map inducing an isomorphism on homotopy groups of degree $\geq b+1$. It is called a $(b+1)$-parallelisable structure or a $b$-connective cover. One can ask for which $k,b$ a map $(BO)_{\geq b+1}\to \BOr_k$ over $BO$ exists, in particular what is the largest such $k$ for a given $b$. Stong computed the persistence of Stiefel-Whitney classes in the stages in the Whitehead tower of $BO$ \cite{Stong63}, see also \cite[pg. 9]{ReneeEulerChar}.

\begin{proposition}[\cite{Stong63}]
Fix $b$ an integer. Define $$\phi(b)= \left|\left\{\; s\; | \;1\leq s\leq b, \; s\equiv 0,1,2,4 \pmod{8}\right\}\right|.$$
Then the reduced cohomology ring $\widetilde{H}^*((BO)_{\geq b+1};\Z/2)$ is trivial for $*<2^{\phi(b)}$. 
\end{proposition}

\begin{corollary}\label{cor:CoversOfBOvskOrientability}
    Let $b,k$ be integers such that $k\leq \phi(b)$ (note that $\phi(b)\leq \frac{b}{2}$ and $\phi$ is close to this bound). Then there is a map  $(BO)_{\geq b+1}\to \BOr_k$ over $BO$.
\end{corollary}
In particular, we have that
\begin{corollary}\label{cor:connectedCoverAndkOrientable}
    If a manifold $M$ has a lift of its stable tangent bundle to the $b^{th}$-connected cover $(BO)_{\geq b+1}$ of $BO$ and for an integer $k\leq \phi(b)$ ($k\leq \frac{b}{2}$ suffices). Then $M$ is $k$-orientable.
\end{corollary}

This estimate is the best possible in a sense that the class $w_{2^k}$ is non-zero in  $H^{2^k}((BO)_{>b};\Z/2)$, see \cite{Stong63, ReneeEulerChar}. In words, if a manifold has a lift to $k^{th}$ {\it non-trivial} connective cover of $BO$ 
then it is $k$-orientable, and the $k^{th}$ non-trivial connective cover is more or less $(BO)_{> 2k}$, i.e. if $M$ is $2k$-parallelisable then it is $k$-orientable.

\begin{example}\label{ex:Bstring3Orientable}
    There is a map $B\String\to \BOr_3$ over $BO$. In particular, for any integer $m$ such that $16\nmid m$ we have that every $m$-dimensional String manifold has even Euler characteristic. 
\end{example}

\subsection{Other tangential structures}\label{sec:otherStructures}

\subsubsection{Unstable and stable framings}\label{subsec:FramingStr}

Consider the structure given by the inclusion of the basepoint $s_n\colon *\to BO_n$.
Observe that if $\zeta$ is an $n$-dimensional vector bundle, then an $s_n$-structure is a trivialisation of $\zeta$, i.e. an ordered $n$-tuple of pointwise linearly independent non-vanishing sections. 
Similarly, the stable structure $s\colon *\to BO$ gives a trivialisation of the stable vector bundle $[\zeta]$. Somewhere in between, we can consider a $k$-dimensional vector bundle $\zeta'$, $k\leq n$. Then an $s_n$ structure on $\zeta'$ is a trivialisation of $\zeta'\oplus \uline{\R}^{n-k}$.
Note that stably framed manifolds (and hence unstably framed manifolds) have vanishing top Stiefel-Whitney class and hence even Euler characteristic.

The structure $s_n$ for $n\neq 1,3,7$ is an example of a non-stabilisable structure:

\begin{lemma}\label{lem:UnstableFramingStructureCannotBeStabilised}
    There does not exist an $(n+1)$-dimensional structure whose pullback to $BO_n$ is $s_n\colon *\to BO_n$ for $n\neq 1,3,7$.
\end{lemma}
\begin{proof}
     Assume there was such a structure $s'\colon B_{n+1}\to BO_{n+1}$ whose pullback is $s_n$. For every $n\neq 1,3,7$, there exists a non-trivial $n$-dimensional vector bundle which is trivial under one stabilisation. Such vector bundles would have an $s'$ structure, but not an $s_n$ structure, which is a contradiction. An example of such a bundle is $TS^n$.
\end{proof}

    Note that for $n=1$ we have $* \simeq BSO_1$ and so $s_1$ can be stabilised to $BSO$.     

\subsubsection{\texorpdfstring{$\Pin^\pm$}{Pin}-manifolds}\label{subsec:Pin}

\begin{definition}\label{def:Pin}
    For $0 \leq n < \infty$ let the tangential structures $B\Pin^+_n\to BO_n$ and $B\Pin^-_n\to BO_n$ be the homotopy fibres of the following fibration sequences.
    \begin{align*}
    &B\Pin^+_n\to BO_n\xrightarrow{w_2}K(\Z/2,2)\\
    &B\Pin^-_n\to BO_n\xrightarrow{w_2+w_1^2}K(\Z/2,2)
    \end{align*}
\end{definition}

For a more geometric definition, see \cite{lawsonmichelsohn}.
In upcoming work \cite{PinPaper}, we hope to completely resolve the question of the parity of the Euler characteristic of $\Pin^{\pm}$-manifolds.
We state the results known so far in \cref{Table:PinEulerChar}. 
For convenience, we also include the parity of $\chi$ of manifolds with structure group $BO,BSO$ and $B\Spin$. In cases where an odd $\chi$ manifold exists, we demonstrate the claim with an example manifold in brackets.

\begin{table}[h!]
\begin{tabular}{|l|lllll|}\hline
dim\textbackslash{}structure & $BO$ & $BSO$ & $B\Spin$ & $B\Pin^{-}$ & $B\Pin^{+}$ \\
\hline odd           & 0  & 0   & 0     & 0        & 0        \\
$8k$            & $\Z(\RP^{8k})$ &$ \Z(\CP^{4k}) $ & $\Z(\HP^{2k})$& $\Z(\HP^{2k}) $      & $\Z(\RP^{8k})$       \\
$8k+2 $         & $\Z(\RP^{8k+2})$ & $2\Z $& $2\Z $  & $\Z(\RP^{8k+2})$       & $2\Z \bf{\; for\; } k=0,1$ \\
& & & & &
$\bf{?\; for\; } k\geq 2$       \\
$8k+4 $         & $\Z(\RP^{8k+4})$ & $\Z(\CP^{4k+2})  $& $2\Z $  & $2\Z$      & $\Z(\RP^{8k+4})$       \\
$8k+6$          & $\Z(\RP^{8k+6})$ & $2\Z $& $2\Z $  & $\Z(\RP^{8k+6}) $     & $2\Z$   \\
\hline 
\end{tabular}
\caption{Possible Euler characteristic of manifolds with $O,SO, \Spin,\Pin^{-}$ and $\Pin^{+}$ structures, see \cite{PinPaper}.}
\label{Table:PinEulerChar}
\vspace{-0.8cm}
\end{table}

In \cite{PinPaper} we actually obtain the following result about Wu classes vanishing for $\Pin^{\pm}$ manifolds in certain dimensions.
\begin{theorem}[\cite{PinPaper}]\label{thm:pinWuClasses}
    Let $k$ be an integer. Then any $(8k+4)$-dimensional $\Pin ^-$ manifold has $v_{4k+2}=0$ and any $(8k+6)$-dimensional $\Pin^+$ manifold has $v_{4k+3}=0$.

    Furthermore, the claim holds for manifolds with boundary and their Wu classes.
\end{theorem}
\begin{proof}
    The first statement is shown in \cite{PinPaper}. The second statement follows from \cref{cor:closedWuboundaryWu}.
\end{proof}
 We conjecture that the conclusion of the above Theorem does \emph{not} hold in the $\Pin^+$ case in the dimension $8k+2$, but that nonetheless all such manifolds still have even Euler characteristic. 
 
Results about $\SKK$ groups of $\Pin^\pm$-manifolds in odd dimensions can be found in \cref{subsec:SKKPin}. For even dimensions, see \cref{cor:pin_even_dim}. 

\subsubsection{Tangential structures relevant for physics}

We will recall and calculate the $\SKK$ groups for
\begin{enumerate}[label=(\roman*)]
    \item $\Spin^c$ in \cref{ex:spinc};
    \item $\Spin^h$ in \cref{ex:spinh};
    \item $G_{\pm} = \Pin^{\pm} \times_{\Z/2} SU_2$ in \cref{subsection:classification_physics};
    \item $\Pin^{\tilde{c}-}$ in \cref{subsection:classification_physics};
    \item $\Pin^c$ in \cref{subsection:classification_physics}.
\end{enumerate}

We also calculate the SKK group for the following structure in \cref{prop:PinCPlus}.

\begin{definition}\label{def:PinCPlus}
    The $\Pin^{\tilde{c}+}$ tangential structure \cite[Proposition 9.4]{freedhopkins} \cite[Proposition 14]{stehouwermorita} is the structure given by the homotopy pullback
    \[
    \begin{tikzcd}
        B\Pin^{\tilde{c}+}\ar[d] \ar[r] & BO_2 \ar[d]
        \\
        BO \ar[r] & \pi_{\leq 2} BO
    \end{tikzcd}
    \]
    where $\pi_{\leq 2} BO \cong K(\Z/2,1) \times K(\Z/2,2)$ is the Postnikov truncation of $BO$ and the right vertical map is the composition of the stabilisation $BO_2 \to BO$ and the truncation.
\end{definition}

The following is claimed in \cite[Lemma D.8]{shiozaki2018many}, and we include a proof here.
\begin{lemma}
    A manifold $M$ has $\Pin^{\tilde{c}+}$ structure if the Bockstein 
    \[\beta_M \colon H^2(M; \Z/2) \to H^3(M; \Z^{w_1})\]
    for the sequence of $\pi_1(M)$-modules $\Z^{w_1}\to \Z^{w_1}\to \Z/2$ 
    gives $\beta (w_2(M))=0$. 
\end{lemma}
\begin{proof}
Note that the maps from $BO_2$ and $BO$ to $\pi_{\leq 2} BO \cong K(\Z/2,1) \times K(\Z/2,2)$ are given by $(w_1,w_2)$.
Therefore, by the homotopy pullback property, a $\Pin^{\tilde{c}+}$-structure on $TM$ exists if and only if there exists a rank two vector bundle $V$ such that $w_1(TM) = w_1(V)$ and $w_2(TM) = w_2(V)$.
Rank two vector bundles $V$ with given $w_1(V) \in H^1(M;\Z/2)$ are classified by their twisted Euler class $e(V) \in H^2(M;\Z^{w_1(V)})$.
Moreover, the twisted Euler class gets mapped to the class $w_2(V) \in H^2(M;\Z^{w_1(V)})$ under the map $\Z^{w_1(V)} \to \Z/2$ of $\pi_1(M)$-modules.
We see that $M$ admits a $\Pin^{\tilde{c}+}$-structure if and only if there exists a class $e(V) \in H^2(M;\Z_{w_1(M)})$ such that 
\[
e(V) \pmod 2 = w_2(M).
\]
The result follows from the fact that $\beta$ is the next map in the long exact sequence on cohomology induced by 
$\Z^{w_1}\to \Z^{w_1}\to \Z/2$.
\end{proof}

\begin{proposition}
\label{pinc+structure}
    For every even $n$, there is an $n$-dimensional $\Pin^{\tilde{c}+}$-manifold with odd Euler characteristic.
\end{proposition}
\begin{proof}
    Note that for $n = 4k$, the manifold $\RP^{4k}$ is $\Pin^{\tilde{c}+}$ since $w_2(\RP^{4k})=0$.

    For $n\equiv 2\pmod{4}$, we can take $X\coloneqq(\CP^{2})^{\frac{n-2}{4}}\times \RP^2$. 
    It suffices to show that the Bockstein on $X$ is zero. 
    All the projections to $\CP^2$ and the projection to $\RP^2$ induce a map on $\pi_1$-modules and so produce diagrams for $Y=\RP^2$ or $Y=\CP^2$ of the form:
    $$\begin{tikzcd}
        H^2(Y;\Z/2)\ar[r,"\beta_Y"]\ar[d]&H^3(Y;\Z^{w_1(Y)})\ar[d]\\
        H^2(X;\Z/2)\ar[r,"\beta_X"]&H^3(X;\Z^{w_1(X)}).
    \end{tikzcd}$$
    For either $Y$ the group $H^3(Y;\Z^{w_1(Y)})$ vanishes. By the K\"unneth formula for the field $\Z/2$, we have \[H^2((\CP^{2})^{\frac{n-2}{4}}\times \RP^2;\Z/2)\cong H^2(\CP^{2};\Z/2)^{\frac{n-2}{4}}\oplus H^2(\RP^2;\Z/2).\] 
    It follows by naturality that $\beta_X=0$.
\end{proof}

\section{\texorpdfstring{$\SKKxi$}{SKK} groups}\label{sec:SKK}

We will now define the main object of study in this paper: the $\SKKxi_n$ groups.
Let $n>0$ be a positive integer, let $\xi_{n}\colon B_{n} \to BO_{n}$ be a tangential structure and let $Y, Y'$ be closed $n-1$-dimensional manifolds with $\xi$-structure.
Let $N_1,N_2,N_1',N_2'$ be $\xi$-manifolds with identifications as $\xi$-manifolds $\partial N_1=\ol{\partial N_2}=Y$ and $\partial N_1'=\ol{\partial N_2'}=Y'$. 
Let $f,g\colon Y\to Y'$ be $\xi$-diffeomorphisms (\cref{def:xidiffeoTetrahedron}). 
Then the $\SKKxi_n$ relation dictates
\begin{equation}\label{eqn:SKK}
N_1\cup_f N'_1+ N_2\cup_g N'_2\sim N_1\cup_g N'_1+ N_2\cup_f N'_2.
\end{equation}

By rearranging all the $N_1,N_1'$ on one side and all $N_2,N_2'$ on the other side we obtain a useful slogan: ``the SKK relation asserts that gluing together two halves of a manifold in two different ways $f,g$ should only depend on $f,g$ and not on the halves being glued", see \cref{fig:SKK}. We refer to \cref{eqn:SKK} as the \ref{SKKrelations}.

\begin{definition}[The $\SKKxi$ group]
     We define $\SKKxi_n$ to be the Grothendieck completion of the monoid of $n$-dimensional closed $\xi$-manifolds $\mathcal{M}_n^\xi$ under disjoint union quotiented by the \ref{SKKrelations}
    \cref{eqn:SKK}.
\end{definition}

\begin{remark}
\label{rem:maponSKK}
A map of tangential structures 
    \[
\begin{tikzcd}
     & B_n' \ar[d, "\xi'"]
    \\
   B_n \ar[r,"\xi'"]\arrow[r, ""{name=U, below}]{} \ar[ru] & BO_n  \arrow[Rightarrow,from=1-2,to=U,shorten=4.5mm]
\end{tikzcd}
\]
    induces a homomorphism $\SKK^\xi_n \to \SKK^{\xi'}_n$. 
\end{remark}

In \cite{KST}, Kreck, Stolz and Teichner provide the following ``map-free" relation, which will be useful in \cref{ap:ProofPi1ofCategories} about the SKK groups of categories. They show that the two relations are equivalent. Their proof has appeared in the literature in \cite[Proposition A.1.]{lorant}.
\begin{proposition}[Chimaera relation] \label{prop:SecondDefOfTheSKKRelation}
    The following relation on the monoid of $\xi$-manifolds is equivalent to the \ref{SKKrelations}:\labeltext{chimaera relation}{chimaera}
    \begin{equation}\label{eqn:chimaera}
    N_1\cup N_1'+N_2\cup N_2'\sim N_1\cup N_2'+N_2\cup N_1',\end{equation}
    given identifications of $\xi$-manifolds $\partial N_1=\ol{\partial N_1'}=\partial N_2=\ol{\partial N_2'}=N$ for some fixed $\xi$-manifold $N$.
\end{proposition}
We refer to \cref{eqn:chimaera} as the \ref{chimaera}.
The idea of the proof of \cref{prop:SecondDefOfTheSKKRelation} is to remove reference to the gluing diffeomorphisms by gluing appropriate mapping cylinders to our manifolds.

While $\SKKxi_n$ groups can be defined for any $\xi_n\colon B_n \to BO_n$, we will from now on assume at least the existence of a single stabilisation $\xi_{n+1}\colon B_{n+1} \to BO_{n+1}$. The reason for this is that we want to relate $\SKKxi_n$ to the $n$-dimensional bordism group $\Omega^{\xi}_n$, which is only defined if we assume $\xi_{n+1}\colon B_{n+1} \to BO_{n+1}$ exists.
SKK groups without this stability condition are studied in \cite{KST}.

Recall that in general, the Euler characteristic is not a bordism invariant. However, it is an $\SKKxi$ invariant.
\begin{lemma}
    The Euler characteristic gives a homomorphism $\chi\colon \SKKxi_n \to \Z$ for any $\xi$. 
\end{lemma}
\begin{proof}
By the inclusion-exclusion principle for the Euler characteristic, we have that $\chi(N_1\cup_f N_1)+ \chi(N_2\cup_g N_2')= \chi(N_1\cup_g N_1)+ \chi(N_2\cup_f N_2)$ for any manifolds $N_1,N_2$ and any $\xi$-diffeomorphisms $f,g\colon \partial N_1\to \partial N_2$.
\end{proof}

\begin{remark}
Let $n$ be an even integer and let $\xi\colon B_{n+1} \to BO_{n+1}$ be an $(n+1)$-dimensional tangential structure so that the bounding sphere $S^n_b$ is defined. It has Euler characteristic 2 and therefore generates a free subgroup of $\langle S^n_b\rangle\subset \SKKxi_n$. 
\end{remark}

The following lemma follows by Novikov additivity of the signature.
\begin{lemma} 
    In dimension $n\equiv 0\pmod{4}$ the signature is an $\SKK^{\xi}_n$-invariant for any orientable $\xi$. 
\end{lemma}

\begin{remark}[Inverses in $\SKK^{\xi}_n$]
    Let $\xi\colon B_{n+1} \to BO_{n+1}$ be a tangential structure. In contrast to the bordism group $\Omega^{\xi}_n$, orientation reversal does not give an inverse in $\SKKxi_n$ in general.
        However if $n$ is odd, then it does hold that $[M]=-[\ol{M}]$ in $\SKK^{\xi}_n$, since we will see later in \cref{lm:BordismEulerCharInSKK} that we have $[M] + [\ol{M}]=\chi(M\times I)[S^n]=0$.
\end{remark}

\subsection{A short exact sequence comparing SKK with bordism groups}\label{subsec:shortExactSeq}

The main goal of this section will be to derive a short exact sequence involving the SKK group which will be our main computational tool in further sections.
It reduces the computation of $\SKKxi_n$ to computations of the usual bordism group $\Omega^\xi_n$ up to splitting questions that we will resolve in many cases.

Even though general statements about the computation of $\Omega^\xi_n$ are difficult to obtain, this is a well-studied problem for which there are many techniques such as the Adams spectral sequence, the Atiyah-Hirzebruch spectral sequence and many of its generalisations such as the James spectral sequence \cite{petethesis}.
Therefore we will focus on understanding $\SKKxi_n$ in terms of $\Omega^\xi_n$.

\begin{theorem}
\label{thm:TheExactSequence}
    Let $\xi_{n+1}\colon B_{n+1} \to BO_{n+1}$ be a once stabilised tangential structure.
    Then the canonical map $\SKKxi_n \to \Omega^{\xi}_n$ is well-defined and yields an exact sequence
    \labeltext{SKK sequence}{SKKseq}
    \begin{equation} \label{seq:theSesSKK}
    \begin{tikzcd}
0\ar[r]&\langle S_b^n \rangle_{\SKKxi_n} \ar[r]&\SKK_n^{\xi}\ar[r]& \Omega^\xi_n \ar[r]& 0
\end{tikzcd}
\end{equation}
    Moreover if $n$ is even, then $\langle S_b^n \rangle_{\SKKxi_n} \cong \Z$.
    For $n$ odd, if $\xi$ is twice stabilised (see \cref{subsec:twiceStabXi}) we have
    \begin{enumerate}[label=(\roman*)]
        \item \label{it:ses1}$[S^n_b] = 0 \in \SKKxi_n$ if there exists a closed $(n+1)$-dimensional $\xi$-manifold with odd Euler characteristic;
        \item\label{it:ses2}$\langle S_b^n \rangle_{\SKKxi_n} \cong \Z/2$ if all closed $(n+1)$-dimensional $\xi$-manifolds have even Euler characteristic.
    \end{enumerate}
\end{theorem}
We refer to \cref{seq:theSesSKK} as the \ref{SKKseq}.

In \cite{KST}, as of now unpublished, Kreck, Stolz and Teichner use geometric arguments to reprove \cref{thm:TheExactSequence}, whereas we use a homotopy theoretic approach.

 In Sections \ref{sec:SKKOdd} and \ref{sec:SkkEven}, we explore whether the \ref{SKKseq} is split, separating the odd- and even-dimensional case because of their distinct character.

\begin{remark}
    Note that in the \ref{SKKseq}, even though the middle term is defined for a tangential structure $\xi_n\colon B_n \to BO_n$, the first and third term crucially use the assumption of the existence of $\xi_{n+1}\colon B_{n+1} \to BO_{n+1}$. The first because we require a stabilisation for the sphere to admit the bounding $\xi$-structure, and the third because we need bordisms to admit a $\xi$-structure.
\end{remark}

We will need the following {\it surgery lemma}, which is proved in \cite{skbook} for $B=BSO$ or $B=BO$. We prove it using homotopy theoretic methods in \cref{subsec:Genauer}.

\begin{lemma}[{The orientable case {\cite[Lemma 4.3]{skbook}}}]\label{lm:BordismEulerCharInSKK}
 Let $B_{n+2} \xrightarrow{\xi_{n+2}} BO_{n+2}$ be a tangential structure.
 Let $W^{n+1}$ be a $\xi_{n+1}$-bordism between two $n$-dimensional $\xi_{n}$-manifolds $M$ and $N$. 
 Then in $\SKK^{\xi}_n$ we have
 \[[M]-[N]=(\chi(M)-\chi(W))[S^n_b].\]
 \end{lemma}

 \begin{remark}
    If $n$ is even, there is a simpler proof of \cref{lm:BordismEulerCharInSKK} from the \ref{SKKseq}, only requiring $\xi$ to be once stabilised:

    if $M$ and $N$ are $\xi$-cobordant the \ref{SKKseq} shows that there is an integer $k$ such that $[M]-[N]=k[S^n_b]$. 
    
    We then have $\chi(k[S^n_b])=\chi([M]-[N])$ and so $2k=\chi(M)-\chi(N)$. If $W$ is a manifold with boundary $M\sqcup \ol{N}$ then $2\chi(W)=\chi(M)+\chi(\ol{N})=\chi(M)+\chi(N)$ and so $k=\chi(W)-\chi(N)$.
\end{remark}

\begin{lemma}[Inheritance of splittings]\label{lm:Inheritance}

Let $\xi_{n+1}\colon B_{n+1}\to BO_{n+1}$, $\xi'_{n+1}\colon B'_{n+1}\to BO_{n+1}$ be two $(n+1)$-dimensional tangential structures with a map $\varphi\colon B_{n+1}\to B'_{n+1}$ over $BO_{n+1}$(see  \cref{rem:maponSKK} ). 
Assume that the induced map $\langle S^n \rangle_{\SKKxi_n}\xrightarrow{\varphi_*}  \langle S^n \rangle_{\SKK^{\xi'}_n}$ is an isomorphism.
Suppose furthermore that the \ref{SKKseq} for $\xi'$ has a section as below
\[
\begin{tikzcd}
    0 \arrow[r] & \langle S^n \rangle_{\SKKxi_n} \ar[d, "\varphi_*"]\ar[d,"\cong"'] \ar[r] & \SKKxi_n \ar[d,"\varphi_*"] \ar[r] & \Omega^\xi_n \ar[d,"\varphi_*"] \ar[r] & 0
    \\
    0 \arrow[r] & \langle S^n \rangle_{\SKK^{\xi'}_n} \ar[r] & \SKK^{\xi'}_n \ar[r]\ar[l,bend left,pos=0.4, "s"'] & \Omega^{\xi'}_n \ar[r] & 0.
\end{tikzcd}
\]
Then the upper row also splits by the induced section.
\end{lemma}

\subsection{Genauer's perspective on the SKK short exact sequence}\label{subsec:Genauer}

This section summarises the homotopy theoretic proof of \ref{SKKseq} in the literature \cite{GMTW,genauer, Steimle21,reutterschommerpries} for $\xi_{n+2}\colon B_{n+2} \rightarrow BO_{n+2}$ a twice stabilised tangential structure.

We define the \emph{topological bordism category of $\xi$-manifolds} $\Bord^{\xi}_{n-1,n}$ as either a topological category (e.g.~\cite{GMTW,Steimle21}) or an $(\infty,1)$-category (e.g.~\cite{calaquescheimbauer,schommerpriesinvertible}) with objects $(n-1)$-dimensional $\xi$-manifolds and morphisms $n$-dimensional $\xi$-bordisms between them.
This is a topological version of \cref{def:cob_cat}, and the two are related by taking the homotopy category or $\pi_0$ on morphism spaces, see \cref{sec:highercats}.

\begin{theorem}[\cite{GMTW}]
We have a weak homotopy equivalence
    \[\Omega \|\Bord^{\xi}_{n-1,n}\|\simeq \Omega^\infty MT\xi_{n},\]
where $MT\xi_{n}$ is the Madsen-Tillmann spectrum of $\xi_n$. 
\end{theorem}
One can fit bordism categories of subsequent dimensions into a homotopy fibre sequence of topological (or $(\infty,1)$-) categories known as the Genauer sequence \cite{genauer,Steimle21}:
\[\Bord^{\xi}_{n,n+1}\to \Bord^{\xi, \partial}_{n,n+1}\to \Bord^{\xi}_{n-1,n}\]
where $\Bord^{\xi, \partial}_{n,n+1}$ is the bordism category of $\xi$-manifolds in which both objects and morphisms are allowed to have a free boundary (thought of as sinking through a fixed hyperplane in $\R^\infty$), and the second map takes the boundary (the intersection with the hyperplane).

The homotopy type of the bordism category with boundary was established in \cite{genauer}, see also \cite[Section 3.4]{reutterschommerpries}
\begin{equation}
\label{eq:bordboundary}
\Omega \|\Bord^{\xi, \partial}_{n,n+1}\| \simeq \Omega^\infty \Sigma_+^\infty B_{n+1}.
\end{equation}
This sequence of categories gives rise to a homotopy fibre sequence of their nerves \cite[Theorem 4.8]{Steimle21} and hence spectra
\[
MT \xi_{n+1} \to \Sigma_+^\infty B_{n+1} \to MT \xi_{n},
\]
constructed in \cite[Section 5]{GMTW} and \cite{genauer}, which yields a long exact sequence of homotopy groups.

We have that \[\pi_{-1}MT\xi_{n} = \pi_0 \|\Bord^{\xi}_{n-1,n} \|= \Omega^\xi_n.\]
If moreover $\xi$ is twice stabilised with respect to $n$, then since both bordism categories $\Bord^{\xi}_{n,n+1}$ and $\Bord^{\xi}_{n-1,n}$ are reversible, it follows from \cref{ap:ProofPi1ofCategories}, \cref{th:SKKispi1ofcob}, that 
\[\pi_{0}MT\xi_{n} = \pi_1 \|\Bord^{\xi}_{n-1,n}\| = \SKK^\xi_n,\]
\[\pi_{0}MT\xi_{n+1} = \pi_1 \|\Bord^{\xi}_{n,n+1}\| = \SKK^\xi_{n+1}.\]

\noindent By an argument analogous to \cref{prop:turningBordismsAround} $\Bord^{\xi, \partial}_{n,n+1}$ is also reversible and 
\begin{equation}\label{eq:SKKRelativeBordCat}\pi_0 \Sigma^\infty_+ B_{n+1}\cong \pi_1(\|\Bord^{\xi, \partial}_{n,n+1}\|)
\cong\SKK(\Bord^{\xi, \partial}_{n,n+1}),\end{equation}
\noindent in other words, the monoid  $\mathcal{M}_{n+1}^{\xi, \partial}$ of $(n+1)$-dimensional $\xi$-manifolds with boundary, viewed as cobordisms with boundary $\varnothing\to \varnothing$ modulo the relative version of the \ref{chimaera}, is isomorphic to $\pi_1(\|\Bord^{\xi, \partial}_{n,n+1}\|)$. Note that since $B_{n+1}$ is connected, we have $\pi_{-1} \Sigma^\infty_+ B_{n+1} = 0$ and $\pi_0 \Sigma^\infty_+ B_{n+1} \cong \Z$.

 The long exact sequence in homotopy groups therefore comes down to:
\begin{center}
    \begin{tikzcd}
        \dots \ar[r] & \SKK^{\xi}_{n+1}\ar[r, "\chi"]&\Z\ar[r, "S_b^n"] & \SKK^{\xi}_n \ar[r] & \Omega^{\xi}_n \ar[r] & 0,
    \end{tikzcd}
\end{center}
where the maps are given by the Euler characteristic $\chi$ and sending the generator to the bounding sphere $S^n_b$ as we now show. 

For the non-orientable case see \cite{GMTW} and \cite{bokstedtcomputations}.

In particular they show that the isomorphism
\[
\SKK(\Bord^{O, \partial}_{n,n+1}) \cong \pi_1(\|\Bord^{O, \partial}_{n,n+1}\|) \cong \pi_0 \Sigma^\infty_+ BO_{n+1} \cong \Z,
\]
 is given by the Euler characteristic.

Note that the Genauer sequence is natural in the tangential structure $\xi$.
So using the comparison maps $MT\xi_{*}\to MTO_{*}$ for $* = n,n+1$ and $\Sigma_+^\infty B_{n+1}\to \Sigma_+^{\infty}BO_{n+1}$ we obtain the following commutative diagram comparing the tails of the two sequences:
\begin{equation}\label{eq:lengthenedSKK}
    \begin{tikzcd}        \SKK^{\xi}_{n+1}\ar[d]\ar[r]&\Z\ar[equal]{d} \ar[r, "\alpha"] & \SKK^{\xi}_n \ar[r]\ar[d] & \Omega^{\xi}_n\ar[d] \ar[r] & 0\ar[equal]{d} 
        \\
        \SKK^{O}_{n+1}\ar[r,"\chi"]&\Z \ar[r, "S^n"]& \SKK^{O}_n \ar[r] & \Omega^{O}_n \ar[r] & 0.
    \end{tikzcd}
\end{equation}
Commutativity of the left square implies that the top map is also $\chi$ and so the isomorphism
\[
\SKK(\Bord^{\xi, \partial}_{n,n+1}) \cong \pi_1(\|\Bord^{\xi, \partial}_{n,n+1}\|) \cong \pi_0 \Sigma^\infty_+ B_{n+1} \cong \Z,
\]

\noindent is also the Euler characteristic.  
 
 Note that the morphism $D^{n+1} \colon \varnothing\to \varnothing$ in $\Bord_{n,n+1}^{\xi, \partial}$ gets mapped to the morphism $S^n_b\colon \varnothing \to \varnothing$ in $\Bord_{n-1,n}^{\xi}$.
 Since $\chi(D^{n+1}) = 1$, we see that this morphism is a generator of $\pi_1(\|\Bord_{n,n+1}^{\xi, \partial}\|)$.
 Therefore we get $\alpha(1) = S^n_b$.

 \begin{remark}
    Recall that the $\xi$-structure on a boundary depends on the choice between an in- versus outgoing vector field normal to the boundary, see \cref{conv:xiStructureOnABoundary}. Choosing the other convention here will lead to the boundary $n$-sphere having the potentially different $\xi$-structure $\ol{S^n_b}$, which makes $D^{n+1}$ into a bordism $\ol{S^n_b} \to \varnothing$ instead of the desired $\varnothing \to \ol{S^n_b}$.

    Choosing the other convention for the normal vector will change some formulas, for example in \cref{lm:BordismEulerCharInSKK}.
\end{remark}

    We note that \cref{eq:SKKRelativeBordCat} follows from abstract homotopy-theoretic arguments.
    We now include a geometric proof in the case $n+1=2$ and $\xi$-structure $O$ or $SO$. We expect a similar proof to be possible for all $n$.

\begin{proposition}
    Let $\xi$ be either the identity or the stable orientation tangential structure $BSO\to BO$. Then any $\xi$-surface $\Sigma$, possibly with boundary is equivalent to $\chi(\Sigma)$ copies of $(D^2,S^1)$ under the \ref{chimaera}.

    In particular we have that the map 
\[\mathcal{M}^{\xi,\partial}_2/\{\text{chimaera relations}\}\to\Z\]
is given by the Euler characteristics and is an isomorphism
\end{proposition}
\begin{proof}

    \begin{figure}
     \subfigure[]{\begin{minipage}{.20\textwidth}\includegraphics[scale=0.5]
{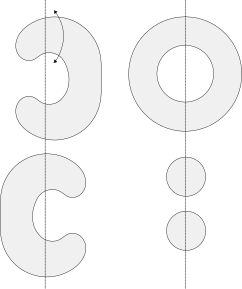}\end{minipage}}\label{fg:GenauerProof4}
        \hspace{1cm}
        \subfigure[]{\begin{minipage}{.35\textwidth}\includegraphics[scale=0.6]
{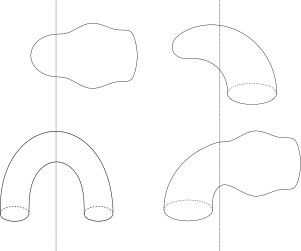}\end{minipage}}\label{fg:GenauerProof1}
        \hfill
        \subfigure[]{\begin{minipage}{.35\textwidth}\includegraphics[scale=0.55]{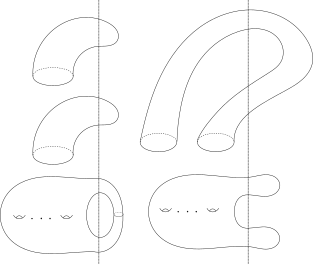}\end{minipage}}\label{fg:GenauerProof2}
\label{fg:GenauerProof3}
  \caption{Sequence of chimaera moves proving the following relations:\\
        (a) $2[D^2]=[S_b^1\times I]+2[D^2]$ \; or\; $2[D^2]=[M\ddot{o}b]+2[D^2]$\\
        (b) $[\Sigma]+[S_b^1\times I]=[\Sigma\setminus D^2]+[D^2]$ \; in particular \;  $[S^2]+[S_b^1\times I]=2[D^2]$ \; or \;$[\RP^2]+[S_b^1\times I]=[M\ddot{o}b]+[D^2]$.;\\ 
        (c) $[\Sigma_{g+1}]+2[D^2]=[\Sigma_g]+[S_b^1\times I]$}
        \label{fig:Genauer2dProof} 
    \end{figure}

      \cref{fig:Genauer2dProof} depicts three relations in $\SKK(\Bord^{\xi, \partial}_{1,2})$. Each should be interpreted as considering the disjoint union of manifolds on the left, cutting them according to the vertical line and swapping the components to obtain the disjoint union of manifolds on the right. Note that all boundaries in \cref{fig:Genauer2dProof} are free boundaries. The illustrations prove the equations in the caption of \cref{fig:Genauer2dProof} using the \ref{chimaera}, which is equivalent to the \ref{SKKrelations} (\cref{prop:SecondDefOfTheSKKRelation}). \cref{fig:Genauer2dProof}(a) represents two possible relations, in the second relation  $M\ddot{o}b$ is the M\"obius strip, and the relation is obtained by introducing a single twist on one of the left discs. In all pictures, the surfaces are allowed to be non-orientable.

    Assume we start with a class in $\SKK(\Bord^{\xi, \partial}_{1,2})$, represented by a possibly disconnected, possibly non-orientable surface $\Sigma$, possibly with boundary, where we have added formal inverses, i.e. allowing components to come with a minus sign. 
     Firstly note that $S^1\times I$ and $M\ddot{o}b$ are zero in $\SKK(\Bord^{\xi, \partial}_{1,2})$ (\cref{fig:Genauer2dProof}(a)). Secondly, relation (b) in \cref{fig:Genauer2dProof} allows us to reduce the number of components which have a boundary but are not discs. So we can assume our manifold consists of components all of which are either discs or closed. Let $\Sigma_0$ be a closed, possibly non-orientable surface of orientable genus $g>0$. Using relation (c) in \cref{fig:Genauer2dProof} we can reduce the orientable genus of $\Sigma_0$ by one, introducing an extra $-2[D^2]$. Finally \cref{fig:Genauer2dProof}(b) shows that $[S^2]$ is SKK-equivalent to $2[D^2]$ and that $[\RP^2]$ is SKK-equivalent to a M\"obius strip plus a $[D^2]$, where we note that the M\"obius strip was zero in $\SKK(\Bord^{\xi, \partial}_{1,2})$ by \cref{fig:Genauer2dProof}(a).  This finishes the proof.
\end{proof}

Now we prove the surgery lemma.

\begin{proof}[Proof of \cref{lm:BordismEulerCharInSKK}]
   We use the commutative square

   \[\begin{tikzcd}
       \Z\ar[r,"{[S^n_b]
    }"]&\SKK_n^{\xi}\\
       \pi_1(\|\Bord^{\xi, \partial}_{n,n+1}\|)\ar[u,"\chi"']\ar[u,"\cong"]\ar[r,"\partial"]&\pi_1 \|\Bord^{\xi}_{n-1,n}\|\ar[u,"\cong"].
   \end{tikzcd}\]

    Let $V$ be any $(n+1)$-dimensional $\xi$-manifold with boundary $X$. It can be viewed as an element $[V]\in \pi_1(\|\Bord^{\xi, \partial}_{n,n+1}\|)$ as a bordism from $\varnothing$ to $\varnothing$. Since $\partial(V)=[X]$, we get that $[X]=\chi(V)[S^n_b]$ in $\SKK_n^{\xi}$.
    
   Let $W$ be a $\xi$-cobordism from $M$ to $N$. It is also a $\xi$-nullbordism of $N \sqcup \ol{M}$ (see \cref{rem:ConventionMfldWithBoundary}  and \cref{def:xiBordism} for our conventions). We thus get $\chi(W)[S^n_b]=[\ol{M}]+[N]$. Applying this argument to $W = M\times I$, we find that $[M]+[\ol{M}]=\chi(M)[S^n_b]$. Putting it together we get
   \[[M]-[N]=(\chi(M)-\chi(W))[S^n_b].\qedhere\]
\end{proof}

We conclude the present section with the proof of the SKK short exact sequence.

\begin{proof}[Proof of \cref{thm:TheExactSequence}]
 The \ref{SKKseq} follows from \cref{eq:lengthenedSKK}. The claims \ref{it:ses1} and \ref{it:ses2} follows from \cref{eq:lengthenedSKK}.   
\end{proof}

\section{SKK groups in odd dimensions}\label{sec:SKKOdd}

\subsection{The if and only if criterion for splitting of the \ref{SKKseq} }
Let $n$ be odd and $\xi$ be a twice stabilised tangential structure with respect to $n$.

If there exists an $(n+1)$-dimensional closed $\xi$-manifold with odd Euler characteristic, by \cref{seq:theSesSKK} the \ref{SKKseq} simplifies to
\[
\SKKxi_n \cong \Omega^\xi_n.
\]
Otherwise, there is a short exact sequence
\[
\begin{tikzcd}
0\ar[r]&\Z/2\ar[r]&\SKK_n^{\xi}\ar[r]& \Omega^\xi_n \ar[r]& 0,
\end{tikzcd}\]
where the first map sends the generator to the bounding sphere $S^n_b$. The goal of this section is to provide an abstract criterion for when a candidate $\Z/2$-valued invariant of $\xi$-manifolds provides a section of the inclusion of the sphere and therefore a splitting $\SKK_n^{\xi} \cong \Omega^\xi_n \times \Z/2$ of the \ref{SKKseq}.
In other words, we are looking for an $\SKK$ invariant 
\[\begin{tikzcd}
\SKKxi_n \ar[r]&\Z/2,\end{tikzcd}
\]
that is non-trivial on the sphere.
When such a splitting exists, there can of course be many; splittings form a torsor over the group of homomorphisms $\Omega^\xi_n \to \Z/2$.
The main result we will use to obtain such a splitting is:

\begin{theorem}
\label{thm:iffForMfldInvariant}
Let $\xi_{n+2}\colon B_{n+2} \to BO_{n+2}$ be a tangential structure and $n$ an odd integer.
Let be $\kappa$ be a $\Z/2$-valued invariant of $n$-dimensional closed $\xi$-manifolds that is additive with respect to disjoint union, 
i.e. a homomorphism
$$\kappa\colon \mathcal{M}^\xi_n \to \Z/2$$
for $\mathcal{M}^\xi_n$ the monoid of $\xi$-manifolds.

Then $\kappa$ factors through $\SKK^{\xi}_n$ and is a splitting of the sequence 
\[
\begin{tikzcd}
0\ar[r]&\Z/2 \ar[r]&\SKK_n^{\xi}\ar[r]& \Omega^\xi_n \ar[r]& 0
\end{tikzcd}
\]
if and only if for all $(n+1)$-dimensional $\xi$-manifolds $W$ with boundary $Y$ we have $$\kappa(Y)=  \chi(W)\mod 2.$$
\end{theorem}
\begin{proof}
Suppose $\kappa$ is a splitting.
If $\partial W = Y$, then $[Y] \in \SKK^{\xi}_n$ is in the kernel of the map to $\Omega^{\xi}_n$.
Hence $[Y] = \chi(W) [S_b^n] \in \SKK^{\xi}_n$ by \cref{lm:BordismEulerCharInSKK}.
Since $\kappa$ is a well-defined $\SKK$ invariant, we have

\[
\kappa(Y) = \kappa(\underbrace{S_b^n \sqcup \dots \sqcup S_b^n}_{\chi(W)}) \equiv \chi(W) \kappa(S_b^n) \pmod{2} \; = \;  \chi(W) \pmod{2},
\]
where the last equality holds because $\kappa(S_b^n) = 1 \in \Z /2$ since $\kappa$ defines a splitting. When $\chi(W)$ is negative, we instead evaluate $\kappa(Y \sqcup S_b^n \sqcup \dots \sqcup S_b^n)$, where there are $- \chi(W)$ copies of the sphere.

Assume conversely that $\kappa$ satisfies $\kappa(Y) = \chi(W) \mod 2$ for all $(n+1)$-$\xi$-manifolds $W$ with boundary $Y$.
Note first that this condition implies that every closed $(n+1)$-dimensional $\xi$-manifold has even Euler characteristic, so the \ref{SKKseq} has kernel $\Z/2$.
It also implies in particular that $\kappa(S_b^n)=1$ since we can take $W$ to be the disc. 
So if we show that $\kappa$ is an $\SKK$ invariant, then we have a well-defined splitting.

Let $M_1,M_2,M_3,M_4$ be $n$-dimensional $\xi$-manifolds with the boundary identification $\partial M_1=\ol{\partial M_2}=\partial M_3=\ol{\partial M_4}=X $  and $f,g\colon X \to X$ be $\xi$-diffeomorphisms.
It suffices to show that
\[
\kappa(M_1 \cup_f M_2) - \kappa(M_1 \cup_g M_2) = \kappa(M_3 \cup_f M_4) - \kappa(M_3 \cup_g M_4).
\]

Let $V_{1,2}$ be the $\xi$-manifold with boundary considered in \cite[Lemma 1.9]{skbook}, defined by gluing $(M_1\times I)$ and $( M_2\times I)$ together by identifying parts of the boundary by $f\times \id\colon \partial M_1\times [0,\frac{1}{3}]\to \ol{\partial M_2\times [0,\frac{1}{3}]}$ and $g\times \id\colon \partial M_1\times [\frac{2}{3},1]\to \ol{\partial M_2\times [\frac{2}{3},1]}$, see \cref{fg:MatthiasFavouriteMfld}(a). Let $V_{3,4}$ be the analogous manifold for $M_3,M_4$.

\begin{figure}[h]
\centering
\subfigure[]{\begin{minipage}[b][6cm][b]{.35\textwidth}\begin{tikzpicture}[scale=0.75]
\draw (0,0)--(9,0) --(9,6)--(0,6)--(0,0);
\draw (0,3)--(3,3) node[midway,above] {$\partial M_1\times [0,\frac{1}{3}]$}node[midway,below] {$\ol{\partial M_2\times [0,\frac{1}{3}]}$}node[pos=-0.1]{$f$};
\draw (6,3)--(9,3)node[midway,above] {$\partial M_1\times [\frac{2}{3},1]$}node[midway,below] {$\ol{\partial M_2\times [\frac{2}{3},1]}$}node[pos=1.1]{$g$};
\draw (3,3)..controls (4.5,2) ..(6,3);
\draw (3,3)..controls (4.5,4) ..(6,3);
\node at (4.5,1.3) {$M_2\times I$};
\node at (4.5,4.7) {$M_1\times I$};
\end{tikzpicture}
\vspace{.1cm}
\end{minipage}}
\hspace{3cm}
\subfigure[] {\begin{minipage}[b][6cm][b]{.4\textwidth}\begin{tikzpicture}
        \node at (-6,0) {\includegraphics[width=.75\textwidth]{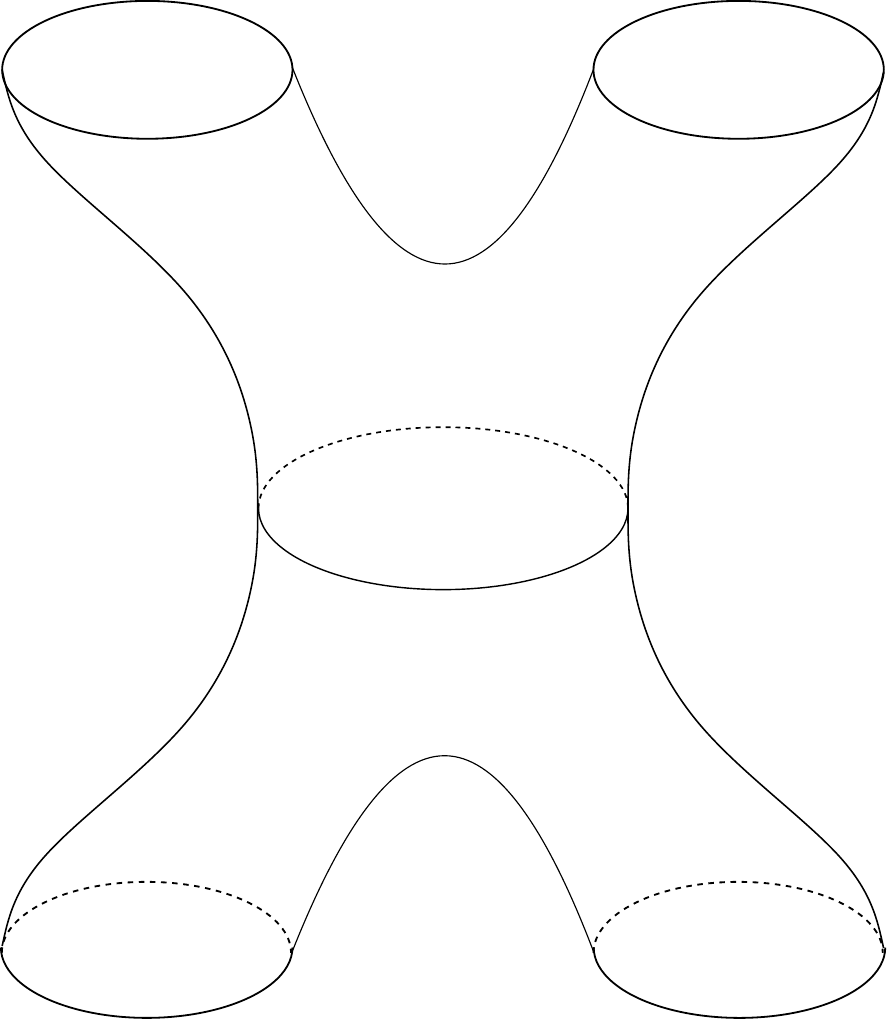}};
        \node at (-6,0.7) {$V_{1,2}$};
        \node at (-6,-0.9) {$V_{3,4}$};
        \node at (-4.5,0) {$T_{fg^{-1}}$};
        \node at (-7.5,2.7) {$M_1\cup_f M_2$};
        \node at (-4.5,2.7) {$\ol{M_1\cup_g M_2}$};
        \node at (-7.5,-2.7) {$M_3\cup_f M_4$};
        \node at (-4.5,-2.7) {$\ol{M_3\cup_g M_4}$};
    \end{tikzpicture}\end{minipage}}\label{fig:1b}
\caption{(a) The manifold $V_{1,2}$. \\\hspace{\textwidth} 
(b) A manifold with boundary $W_{1,2,3,4}=V_{1,2}\cup_{T_{fg^{-1}}}V_{3,4}$.} \label{fg:MatthiasFavouriteMfld}
\end{figure}

After smoothing the corners the manifolds $V_{1,2}$ and $V_{3,4}$ inherit a $\xi$-structure and have boundary 
\begin{align*}
    \partial V_{1,2}&\cong (M_1 \cup_f M_2) \sqcup \overline{(M_1 \cup_g M_2)} \sqcup  T_{fg^{-1}}\\
    \partial V_{3,4}&\cong (M_3 \cup_f M_4) \sqcup \ol{(M_3 \cup_g M_4)} \sqcup  T_{fg^{-1}}
\end{align*}
where $T_{fg^{-1}}$ denotes the mapping torus. Form $W_{1,2,3,4}=V_{1,2}\cup_{T_{fg^{-1}}} V_{3,4}$. 

Note that $\kappa$ is not sensitive to orientation reversal since for any $n$-dimensional $\xi$-manifold $N$, the cylinder is a nullbordism of $N\sqcup \ol{N}$ and so \[\kappa(N)+\kappa(\ol{N})=\chi(N\times I)=0\]. 

Omitting orientation reversals we have
\begin{align*}
    \chi(W_{1,2,3,4}) &\pmod 2 =  \kappa(\partial W_{1,2,3,4})=\\
    &\kappa(M_1 \cup_f M_2) +\kappa(M_1 \cup_g M_2)+\kappa(M_3 \cup_f M_4)+ \kappa(M_3 \cup_g M_4).
\end{align*}

To finish the proof, we show that $\chi(W_{1,2,3,4})$ is even. It suffices to show that $\chi(V_{1,2})$ and $\chi(V_{3,4})$ are even.

We compute
 \[
\chi(V_{1,2}) = \chi(M_1\times I)+\chi(M_2\times I)-2\chi(\partial M_1\times I)\equiv \chi(M_1)+\chi(M_2) \mod 2.
\]
But $\chi(M_1)=\chi(M_2)=0$, because the manifolds are odd dimensional. The computation is analogous for $V_{3,4}$.
\end{proof}

\subsection{Kervaire semi-characteristics}

In this section, we introduce the Kervaire semi-characteristic for a field $F$ and state our main technical result \cref{thm:OddSplittingTopWuclass}.

Previously the Kervaire semi-characteristic over $\Q$ was shown to give a splitting of the \ref{SKKseq} for oriented manifolds of dimension $1 \mod{4}$ 
\cite[Remark on page 47]{skbook}, \cite[page 11-12]{ebert}.
We find that the $\Z/2$-Kervaire semi-characteristic provides a splitting in a wider range of cases, and it is our main candidate for a splitting of the \ref{SKKseq} in odd dimensions.

\begin{definition}\label{def:kervaire}
Let $F$ be a field and $M$ a $(2k+1)$-dimensional manifold, and assume that $M$ is $HF$-orientable. The $F$-{\it Kervaire semi-characteristic} of $M$ is the following element of $\Z/2$:
\[\kerv_F(M)=\sum_{i=0}^k \dim_F H_i(M;F) \pmod 2.\]

\end{definition}
By Poincar\'e duality, since the manifold $M$ is $HF$-orientable, we can equivalently define $\kerv_F(M)$ as the sum of dimensions of all even-dimensional (co-)homology groups of $M$ modulo 2.

\begin{example}
    The odd-dimensional sphere $S^{2k+1}$ has Kervaire semi-characteristic 1 over any field. 
\end{example}

\begin{example}
    The Kervaire semi-characteristic over any field of a one-dimensional closed manifold is the number of components modulo two.
\end{example}

The Kervaire semi-characteristic over $F$ only depends on the characteristic of $F$.
Indeed, if $F \subseteq F'$ is a field extension, then $F'$ is a free $F$-module and so
\[
H^*(M; F) \otimes_F F' \cong H^*(M; F').
\]

\begin{remark}\label{rem:KervCharDifferentCharacteristics}
    The Kervaire semi-characteristic over $\Z/p$ for varying primes $p$ and over $\Q$ can in general all be different topological invariants as evidenced by the Lens spaces $L(p,q)$. Given $p'\neq p$ a different prime, we have
\[
    1=\kerv_{\Q}(L(p,q))= \kerv_{\Z/p'}(L(p,q))
    \neq \kerv_{\Z/p}(L(p,q))=0.
\]
\end{remark}

In the rest of the section, we will prove the following.
\begin{theorem}\label{thm:OddSplittingTopWuclass}
    Let $n$ be odd. Let $\xi_{n+2}\colon B_{n+2}\to BO_{n+2}$ be a tangential structure, where for every closed $(n+1)$-dimensional manifold $M$ the top Wu class $v_{\frac{n+1}{2}}(M)$ vanishes. Then we have $\langle S^n_b\rangle \cong\Z/2$ and there is a split short exact sequence \begin{equation*}\begin{tikzcd} 0 \ar[r] &\Z/2\arrow[r]&\SKK^{\xi}_n\arrow[r]\arrow[l,bend right,pos=0.4,,"\kerv_{\Z/2}"']&\Omega_n^{\xi} \ar[r] & 0\end{tikzcd}
        \end{equation*} where $\kerv_{\Z/2}$ is the Kervaire semi-characteristic over $\Z/2$.
\end{theorem}

First, we prove the following general condition for a Kervaire semi-characteristic to be a splitting of the \ref{SKKseq} in odd dimensions.

\begin{proposition}
\label{prop:iffkervaire}
     Let $F$ be a field, $n$ an odd integer and let $\xi_{n+2}\colon B_{n+2} \to BO_{n+2}$ be a tangential structure, such that every $(n+1)$-dimensional $\xi$-manifold has even Euler characteristic. Furthermore assume that every $n$- dimensional $\xi$-manifold is $HF$ oriented.
    Then $\kerv_F$ gives a splitting $\SKKxi_n \to \Z/2$ of the sequence
\begin{equation*}\begin{tikzcd} 
0 \ar[r] &\Z/2\arrow[r]&\SKK^{\xi}_n\arrow[r]&\Omega_n^{\xi} \ar[r] & 0
\end{tikzcd}
        \end{equation*}
    if and only if for every $(n+1)$-dimensional $\xi$-manifold $W$, possibly with boundary, the image of the map
    \[
    H_{\frac{n+1}{2}}(W;F) \xrightarrow{j_*} H_{\frac{n+1}{2}}(W,\partial W; F)
    \]
    has even dimension.
\end{proposition}

Before we get to the proof, we review some facts about relative homology and Poincar\'e-Lefschetz duality.
Recall that Poincar\'e-Lefschetz duality for $HR$-oriented manifolds $M^n$ possibly with boundary, e.g. see \cite[Theorem 3.43]{hatcher}, says that the cap product defines isomorphisms 
\[
PD\colon  H^{n-k}(M,R)\xrightarrow{\cong} H_k(M,\partial M;R) \text{ and } PD\colon  H^{n-k}(M,\partial M;R)\xrightarrow{\cong} H_k(M;R),
\]
for any integer $k$.
This leads to the following definition.

\begin{definition}[The intersection form of even-dimensional manifolds]\label{defIntersectionForm}
Let $R$ be a ring and $M$ a compact manifold of dimension $2k$, possibly with boundary, which is orientable in homology with coefficients in $R$. Then there is an intersection form in cohomology
\[\lambda(a',b')= \left<j^*(a'),PD(b')\right>\] for $a', b' \in  H^k(M, \partial M;R)$.

The adjoint of this form can be defined as the following composition
\[
\adjustbox{scale=0.87}{
\begin{tikzcd}
{H^{n-k}(M,\partial M,R))} \arrow[r, "j^*"]   & {H^{n-k}(M,R)}  \arrow[r, "\cong","PD"']                           & {H_{k}(M,\partial M;R)} \arrow[r, "coev"]        & {\Hom(H^{k}(M,\partial M;R),R).}   
\end{tikzcd}}
\]
\end{definition}

\begin{remark}\label{rm:RankOfIntersectionFormIsTheRankOfThatInclusionMap}
    If $R = F$ is a field, then the coevaluation map
\[
coev\colon H_{k}(M,\partial M;R) \to \Hom(H^{k}(M,\partial M;R),R)
\]
is an isomorphism.
Since Poincar\'e duality is also an isomorphism, we obtain in this case that $\rank(\lambda)=\rank(j^*)$.
Similarly we get in homology that $\rank(\lambda)=\rank(j_*)$ for $j_*\colon H_k(M;F)\to H_k(M,\partial M;F).$
\end{remark}

The following lemma is also proven in \cite[pg. 991]{Stong76}.

\begin{lemma}\label{lm:Bounding}
Let $F$ be a field and let $W^{2k}$, $\partial W=Y$ be $HF$-oriented manifolds. Then \[\kerv_F(Y)=\dim(H_k(W;F) \xrightarrow{j_*} H_k(W,Y;F))+\chi(W)\pmod{2}.\]
\end{lemma}

\begin{proof}
Consider the long exact sequence in homology of a pair $(W,Y)$ suppressing the coefficients $F$. We can truncate it on the left as follows:
\begin{align*}
    0\to\ker q\to H_k(W)&\xrightarrow{q} H_k(W,Y)\to H_{k-1}(Y)\to\cdots\\
    \cdots &\to H_1(W,Y)\rightarrow H_0(Y)\rightarrow H_0(W)\rightarrow H_0(W,Y)\to 0.
\end{align*}
The sum of the dimensions of all terms in an exact sequence is zero modulo $2$.
\[
0\equiv\dim(\ker(q))+ \dim H_k(W)+\sum_{i=0}^k H_i(W,Y)+\sum_{i=0}^{k-1}H_i(W)+\sum_{i=0}^{k-1} H_i(Y) \pmod 2
\]
\[0\equiv\dim(H_k(W)\xrightarrow {j_*}H_k(W,Y))+\sum_{i=k}^{2k} H_i(W)+\sum_{i=0}^{k-1}H_i(W)+\kerv_F(Y) \pmod 2\]
In the last step we used that $H_i(W,Y)\cong H^{2k-i}(W)\cong H_{2k-i}(W)$ using that our coefficients are in a field. We conclude
\[0\equiv \dim(H_k(W;F)\rightarrow H_k(W,Y;F))+\chi(W)+\kerv_F(Y) \pmod 2 .\qedhere\]
\end{proof}
We now prove \cref{prop:iffkervaire}:
\begin{proof}[Proof of \cref{prop:iffkervaire}]
It is clear that $\kerv_{F}$ is a $\Z/2$-valued invariant of $\xi$-manifolds that is additive with respect to disjoint union. 
By 
\cref{thm:iffForMfldInvariant}, it suffices to show that for every $\xi$-manifold $W$ with boundary, we have that
\[
\kerv_F(\partial W) \equiv \chi(W) \pmod 2.
\]
By \cref{lm:Bounding} however, 
\[
\kerv_F(\partial W) \equiv \chi(W) + \dim(H_k(W;F) \xrightarrow{j_*} H_k(W,Y;F)) \pmod 2,
\]
so that $\kerv_F$ is a splitting if and only if $\dim(j_*)$ is even for all $W$. 
\end{proof}

Next, we want to study for which $\xi$ the Kervaire semi-characteristic gives a splitting, using the conditions of \cref{prop:iffkervaire}. 
Therefore, we turn to the question of when the obstruction term given by 
\[
\dim(H_k(W;F) \xrightarrow{j_*} H_k(W,Y;F)) \pmod 2
\]
vanishes for specific $\xi$-structures.

\begin{proof}[Proof of \cref{thm:OddSplittingTopWuclass}]
        We aim to show that the Kervaire semi-characteristic $\kerv_{\Z/2}$ gives a splitting of the \ref{SKKseq} if the top Wu class vanishes for every closed $\xi$-manifold of dimension $(n+1)$. 

        By \cref{prop:iffkervaire}, it suffices to show that $\dim (H_{k}(W;\Z/2)\xrightarrow{j_*}H_{k}(W,Y;\Z/2))$ is even. This dimension is, by \cref{rm:RankOfIntersectionFormIsTheRankOfThatInclusionMap}, equal to the dimension of the non-degenerate part of the intersection form on $W$. 

Take $x \in H^{k}(W, Y;\Z/2)$. Then the Wu class $v_k \in H^{k}(W;\Z/2)$ (see \cref{sec:RelativeCharClasses}) has the property that 
$x^2 = \Sq^k x = v_k x$. We also have that $v_k=0$ for closed $2k$-dimensional $\xi$-manifolds by assumption. By \cref{cor:closedWuboundaryWu}, this then holds for manifolds with boundary as well.

We have \[\lambda(x,x)= \left<j^*(x), PD(x)\right>  = \left<j^*(x)x,[W,Y]\right>
= \left<x^2,[W,Y]\right>,\] where the last equality follows from naturality of the cup product with respect to $(W,\varnothing)\times (W,Y)\to (W,Y)\times (W,Y)$. 
Hence $\lambda(x,x)=0$, i.e. $\lambda$ is an even form. By the classification of $\Z/2$-valued even forms we get that the non-degenerate part of $\lambda$ has even dimension, which completes the proof.
\end{proof}

Next, we remark on the splittings given by Kervaire semi-characteristic over fields other than $\mathbb{Z}/2$, as well as non-uniqueness of the splitting of the SKK sequence in general.

\begin{remark}[Kervaire semi-characteristics over different fields]\label{rm:differentKervaire}
    Suppose $\xi$ is a twice stabilised tangential structure and $n$ an odd integer such that $\langle S^n_b\rangle_{\SKK^{\xi}}\cong\Z/2$.
    Assume that $\kerv_{\Z/2}$ gives a splitting of the \ref{SKKseq}. Then for a field $F$, the semi-characteristic $\kerv_F$ splits the same sequence if and only if the difference $\kerv_F-\kerv_{\Z/2}$ is a bordism invariant. 

 We study the case $F = \Q$. 
 For $(4k+1)$-dimensional orientable manifolds $M$
\[\kerv_{\Q}(M)-\kerv_{\Z/2}(M)= \langle w_2w_{4k-1},[M]\rangle,\] 
where the right hand side is the \emph{de Rham invariant} \cite{Lusztig69}. This is a bordism invariant which detects the isomorphism $\Omega^{SO}_5 \cong \Z/2$ and is in particular non-trivial on $5$-dimensional manifolds.
It also detects the symmetric signature $\Omega^{SO}_n \to  L_n(\Z) \cong \Z/2$, where $L_n(\Z)$ is the symmetric $L$-theory group for $n = 1 \pmod 4$. It follows that for  $n=4k+1$, $\kerv_\Q$ also gives a splitting of the SKK sequence provided $\kerv_{\Z/2}$ does. This splitting is different at least in dimension $4k+1=5$. This recovers the classical result of \cite{skbook}. 

It can also happen that $\kerv_\Q$ does not give a splitting while $\kerv_{\Z/2}$ does. This happens for example for $\Spin$ manifolds in dimension 3: $\RP^3$ has $\kerv_{\Z/2}(X)=0$ while $\kerv_{\Q}(X)=1$, but it is zero in bordism as $\Omega_3^{\Spin}=0$.
\end{remark}

\subsection{SKK groups of \texorpdfstring{$k$}{k}-orientable manifolds}

One of our main applications of 
\cref{thm:OddSplittingTopWuclass} is for $k$-orientable tangential structures. From this we will additionally deduce some splitting results for the connective covers of $BO$ (e.g. $BSO$, $B\Spin$, $B\String$).

Let  $\xi\colon \BOrk \to BO$ be a $k$-orientable structure as defined in \cref{def:k-orientable}. The question of deciding which $(n+1)$-dimensional $\BOr_k$-manifolds necessarily have an even Euler characteristic, was discussed in \cref{section:k-orientability}.
The general answer is not known, see \cref{OpQues:DoesSuchOddEulerCharExist}.

Let us denote the subgroup generated by spheres $\left<S^n_b\right>\leq \SKK_n^{\Or_k}$ by $I^k_n$. 
Then \cref{thm:TheExactSequence} gives us the exact sequence
\begin{equation}\label{eq:kOrientSequence}
\begin{tikzcd} 0 \ar[r] &I^k_n\arrow[r]&\SKK^{\Or_k}_n\arrow[r]&\Omega_n^{\Or_k} \ar[r] & 0.\end{tikzcd}
\end{equation}

Now we prove one of our main Theorems:

\begin{theorem}\label{thm:mainthmforOddDimensionKOrient}
    For any $k\geq 0$ , and any $n$ odd we have
    \begin{enumerate}[label=(\roman*)]
        \item\label{it:TheCaseWhereWeFoundTheSection} if $2^{k+1}\nmid n+1$ then $I^k_n\cong\Z/2$ and there is a split short exact sequence \begin{equation*}\begin{tikzcd} 0 \ar[r] &\Z/2\arrow[r]&\SKK^{\Or_k}_n\arrow[r]\arrow[l,bend right,pos=0.4,,"\kerv_{\Z/2}"']&\Omega_n^{\Or_k} \ar[r] & 0\end{tikzcd}
        \end{equation*} where $\kerv_{\Z/2}$ is the Kervaire semi-characteristic over $\Z/2$.
        \item\label{it:oddDimMflExists} if $2^{k+1}\mid n+1$ and there exists an $(n+1)$-dimensional $k$-orientable manifold with odd Euler characteristic then $I^k_n=0$ and the obvious map is an isomorphism
        \[\SKK_{n}^{\Or_k}\cong \Omega^{\Or_k}_n.\]
        \item\label{it:oddDimMflDoesntExists} if $2^{k+1}\mid n+1$ and such manifold from  \ref{it:oddDimMflExists} does not exist then $I^k_n\cong\Z/2$ and there is a short exact sequence
        \begin{equation*}\begin{tikzcd}0 \ar[r] &\Z/2\arrow[r]&\SKK^{\Or_k}_n\arrow[r]&\Omega_n^{\Or_k} \ar[r] & 0.\end{tikzcd}
        \end{equation*}
    \end{enumerate}
\end{theorem} 
\begin{proof}\label{proof:ourMainOddDimensionalTheoremKOrient}

Parts \ref{it:oddDimMflExists} and \ref{it:oddDimMflDoesntExists} are immediate consequences of \cref{thm:TheExactSequence}.
For \ref{it:TheCaseWhereWeFoundTheSection}, assume we have $k,n$ such that $2^{k+1}\nmid n+1$. Then by \cref{thm:k-orientEulerChar} and \cref{thm:TheExactSequence}, we have a short exact sequence exact sequence
\begin{equation*}\begin{tikzcd}0 \ar[r] &\Z/2\arrow[r]&\SKK^{\Or_k}_n\arrow[r]&\Omega_n^{\Or_k} \ar[r] & 0.\end{tikzcd}\end{equation*}

By \cref{lm:ReneeKOrientabilityAndWuClassInRelativeCase}, an $(n+1)$-dimensional $k$-orientable manifold, possibly with boundary, has $v_{\frac{n+1}{2}}=0$. The result follows by \cref{thm:OddSplittingTopWuclass}.
 \end{proof}

 We will use the results about splitting of the \ref{SKKseq} for $k$-orientable manifolds to deduce the splittings in dimensions specified below of the same sequence for various Whitehead truncations of $BO$, e.g. $BSO,B\Spin,B\String,B\Fivebrane,\cdots$. 
 
 Recall \cref{cor:CoversOfBOvskOrientability} where, for a given integer $b$ we determined the maximum $k$, such that there is a map $(BO)_{\geq b}\to \BOrk$ over $BO$.

\begin{example}\label{ex:BstrigSplitting}
    Recall from \cref{ex:Bstring3Orientable} that there is a map $B\String\to \BOr_3$ over $BO$. Take an odd integer $n$ such that $16$ does \emph{not} divide $n+1$. This gives us:
    \[
\begin{tikzcd}
    0 \arrow[r] & \langle S^n \rangle_{\String} \ar[d, "\cong"] \ar[r] & \SKK_n^{\String} \ar[d] \ar[r] & \Omega^{\String}_n \ar[d] \ar[r] & 0
    \\
    0 \arrow[r] & \langle S^n \rangle_{\Or_3} \ar[r] & \SKK^{\Or_3}_d \ar[r]\ar[l,bend left,pos=0.5,"\kerv_{\Z/2}"] & \Omega^{\Or_3}_n \ar[r] & 0.
\end{tikzcd}
\]
The leftmost vertical map is an isomorphism of the groups, both $\Z/2$, so the top row also splits by the $\Z/2$-valued Kervaire semi-characteristic.
\end{example}

This of course generalises.
\begin{corollary}[of \cref{thm:mainthmforOddDimensionKOrient}, see also \cref{cor:CoversOfBOvskOrientability}]\label{cor:SKKOfBOb}
Let $b,k$ be integers as in \cref{cor:CoversOfBOvskOrientability}. Let $n$ be an odd integer, such that $2^{k+1}\nmid n+1$. Then the following \ref{SKKseq} for $(BO)_{> b}$
\begin{center}    \begin{tikzcd}0\ar[r]&\Z/2\arrow[r]&\SKK^{(BO)_{> b}}_n\arrow[r]\arrow[l,bend right,pos=0.4,,"\kerv_{\Z/2}"']&\Omega_d^{(BO)_{> b}}\ar[r]&0\end{tikzcd}
\end{center}
splits by the Kervaire semi-characteristic over $\Z/2$.
\end{corollary}
\begin{proof}
Recall by \cref{cor:CoversOfBOvskOrientability} we have a map $(BO)_{> b}\to \BOrk$ over $BO$, in particular every $(BO)_{> b}$-manifold $M$ is $k$-orientable. By \cref{thm:TheExactSequence} we get $\langle S^n_b\rangle \cong \Z/2$ in $\SKK^{(BO)_{> b}}_n$. Also by \cref{thm:mainthmforOddDimensionKOrient} the SKK sequence for $n$ splits by the Kervaire $\Z/2$ semi-characteristic and so the inheritance of splittings (\cref{lm:Inheritance}) establishes the result.  
\end{proof}

\cref{cor:SKKOfBOb} reflects what happens in (certain) odd dimensions such that there are no odd $\chi$ manifolds in dimension $n+1$.
On the other hand, the existence of the odd $\chi$ manifolds listed in \cref{Table:k-orientable} guarantees the following.

\begin{proposition}\label{prop:SKKForSmallConnectiveCovers}
We have the following isomorphisms:
    \begin{enumerate}[label=(\roman*)]
        \item \cite{skbook} For $(BO)_{>0}\simeq BO$ and for $2\mid (n+1)$ we have $$\SKK^{O}_n\cong \Omega_n^{O}.$$ 
        \item\cite{skbook} For $(BO)_{> 1}\simeq BSO$ over $BO$ and for $4\mid (n+1)$ we have $$\SKK^{SO}_n\cong \Omega_n^{SO}.$$
        \item For $(BO)_{> 2}\simeq B\Spin$ over $BO$ and for $8\mid (n+1)$ we have $$\SKK^{\Spin}_n\cong \Omega_n^{\Spin}.$$
        \item For $(BO)_{> 4}\simeq B\String$ over $BO$ and for $16\mid (n+1)$ we have $$\SKK^{\String}_n\cong \Omega_n^{\String}.$$
    \end{enumerate}
\end{proposition}

\subsection{SKK groups for other tangential structures}

\subsubsection{Unstable and stable framings}

Fix $n$ an odd integer. Consider various framing tangential structures of \cref{subsec:FramingStr}. For a framing structure to be twice stabilised we need to either consider the stable framing structure $s\colon *\to BO$ or a framing on at least twice stabilised vector bundles with respect to $n$, $s_{n+k}\colon *\to BO_{n+k}$, $k\geq 2$. 
For any of these structures, closed manifolds in any dimension have even Euler characteristic.

\begin{theorem}\label{SKKStablyFramed}
    For $2\leq k\leq \infty$, the \ref{SKKseq} 
\[\begin{tikzcd}
0\ar[r]&\Z/2\ar[r]&\SKK_n^{s_{n+k}}\ar[r]& \Omega^{s_{n+k}}_n \ar[r]& 0
\end{tikzcd}\]
    splits by the $\Z/2$-Kervaire semi-characteristic.
\end{theorem}
\begin{proof}
    The statement follows from \cref{thm:OddSplittingTopWuclass}. 
    The appropriate Wu classes vanish because the tangent bundles of $s_{n+k}$-manifolds are stably trivial.
\end{proof}

\subsubsection{\texorpdfstring{$\Pin^\pm$}{Pin} -manifolds}
\label{subsec:SKKPin}
\begin{proposition}\label{prop:PinpmSKK}
    Let $n$ be an odd integer. Then we have 

    \[\SKK_n^{\Pin^+}=\begin{cases}
        \Omega_n^{\Pin^+}& \text{for } n\equiv 3,7\pmod{8}\\
        \Z/2\times  \Omega_n^{\Pin^+}& \text{for } n\equiv 5\pmod{8}\\
        ?&n\equiv 1\pmod{8}
    \end{cases}\]

    \[\SKK_n^{\Pin^-}=\begin{cases}
        \Omega_n^{\Pin^-}& \text{for } n\equiv 1,5,7\pmod{8}\\
        \Z/2\times  \Omega_n^{\Pin^-}& \text{for } n\equiv 3\pmod{8}.
    \end{cases}\]
    
Here the maps $\SKK_n^{\Pin^+}\to \Z/2$ resp. $\SKK_n^{\Pin^-}\to \Z/2$ are given by $\kerv_{\Z/2}$. 
\end{proposition}
\begin{proof}
    The statement follows from \cref{thm:OddSplittingTopWuclass}  and  \cref{thm:pinWuClasses}.
\end{proof}
For the calculation of $\Pin^+$ and $\Pin^-$ bordism  groups we refer the reader to \cite{kirbytaylor, kirbyTaylorPinPlus}.
To our knowledge $\SKK_n^{\Pin^+}$ is unknown in general for $n\equiv 1\pmod{8}$, both because it remains unresolved whether $8k+2$-dimensional $\Pin^+$ manifolds have even Euler characteristic for $k\geq 2$, and because, if they do, it remains unclear whether the sequence is split for $k \geq 1$.

The following example shows that $\SKK^{\Pin^+}_1\cong \Z/2\times  \Omega^{\Pin^+}_1$, but the splitting is \emph{not} given by the Kervaire semi-characteristic over any field.  
In fact, a splitting necessarily depends on the $\xi$-structure, and not just on the underlying manifold.

\begin{example}(On the group $\SKK^{\Pin^+}_1$)\label{Ex:pindim1} 
Consider the structure $B\Pin^+ \rightarrow BO$. To calculate $\SKK^{\Pin^+}_1$, we first need to understand the Euler characteristic of $\Pin^+$-surfaces. For a surface $\Sigma$, the parity of the Euler characteristic is measured by $w_2$, which is also the obstruction for the existence of a $\Pin^+$ structure. 
Therefore every $\Pin^+$-surface has even Euler characteristic. There are two connected 1-dimensional $\Pin^+$-manifolds, the periodic circle $S^1_{per}$ and the anti-periodic circle, which in our context we call bounding $S^1_b$ (see \cref{rm:boundingCircle}).

We conclude that the bounding circle $S^1_b$ generates a $\Z/2$ inside $\SKK^{\Pin^+}_1$. 

We have $\Omega^{\Pin^+}_1 = 0$ and so $\SKK^{\Pin^+}_1 \cong \Z/2$.
However, the latter isomorphism is not given by the Kervaire semi-characteristic over any field, which here is simply the number of connected components modulo two.
Indeed, the periodic circle $S^1_{per}$ bounds a M\"obius strip and therefore is trivial in $\SKK_1^{\Pin^+}$ by \cref{lm:BordismEulerCharInSKK}, hence no map that does not take into account the $\Pin^+$ structure can give a splitting.
\end{example}

\begin{proposition}\label{prop:PinPlus8kplus1}
    If $k \geq 0$ is so that all $8k+2$-dimensional $\Pin^+$-manifolds have even Euler characteristic, then the \ref{SKKseq}
\[\begin{tikzcd}
0\ar[r]&\Z/2 \ar[r]&\SKK^{\Pin^+}_{8k+1}\ar[r]& \Omega_{8k+1}^{\Pin^+}\ar[r]& 0
\end{tikzcd}\]
     can never be split by an invariant that only depends on the underlying manifold (in particular, it cannot be split by a Kervaire semi-characteristic).
    \end{proposition}
    \begin{proof}
    Pick any $\Spin$ structure on the quaternionic projective space $\HP^{2k}$ and form the spin manifolds $X_1=S^1_{b}\times \HP^{2k}$ and $X_2=S^1_{per}\times \HP^{2k}$.
Considering these as $\Pin^{+}$-manifolds, $D^2\times \HP^{2k}$ and $M\ddot{o}b\times \HP^{2k}$ are $\Pin^{+}$-nullbordisms of $X_1$ and $X_2$ respectively.
        Applying \cref{lm:BordismEulerCharInSKK} to the nullbordisms we find $[X_1]=[S^{8k+1}_b]\in \SKK^{\Pin^+}_{8k+1}$ and $[X_2]=0\in \SKK^{\Pin^+}_{8k+1}$ which finishes the proof. 
    \end{proof}

\subsubsection{SKK groups with tangential structures relevant for physics}

\begin{example}
\label{ex:spinc}
Consider the groups $\Spin^c_n \coloneqq \frac{\Spin_n \times U(1)}{\Z/2}$ with their corresponding stable tangential structure $B = B\Spin^c$. 
Since every $\Spin^c$ manifold is orientable, there are no $\Spin^c$-manifolds with odd Euler characteristic of dimension $4k+2$.
We claim $\C \mathbb{P}^{2k}$ is a $\Spin^c$ manifold with odd Euler characteristic of dimension $4k$.
The only obstruction for an orientable manifold $M$ to admit a $\Spin^c$ structure is the third integral Stiefel-Whitney class $W_3 \coloneqq \beta w_2 \in H^3(M;\Z)$, where $\beta$ is the Bockstein homomorphism.
Since $H^3(\C \mathbb{P}^n;\Z) = 0$, we see that $\C \mathbb{P}^n$ is $\Spin^c$ for any $n$.
Therefore by the \cref{thm:mainthmforOddDimensionKOrient} we inherit the splitting of the \ref{SKKseq} using the forgetful map $\Spin^c\to O$ (\cref{lm:Inheritance}). Hence
\[
\SKK^{\Spin^c}_{n} \cong
\begin{cases}
    \Omega^{\Spin^c}_{n} & n\equiv 3\pmod{4},
    \\
    \Omega^{\Spin^c}_{n} \times  \Z/2 & n\equiv 1\pmod{4}.
\end{cases}
\]
where the map to $\Z/2$ is given by $\kerv_{\Z/2}$.
\end{example}

\begin{example}\label{ex:spinh}
    Define $\Spin^h_n = \frac{\Spin_n \times SU_2}{\Z/2}$ where the quotient is by the diagonal $\Z/2$-subgroup \cite{spinh}.
    It can be shown that every $\Spin^c$ manifold is $\Spin^h$ and every $\Spin^h$ manifold is orientable.  Applying the last example gives
    \[
\SKK^{\Spin^h}_{n} \cong
\begin{cases}
    \Omega^{\Spin^h}_{n} & n\equiv 3\pmod{4},
    \\
    \Omega^{\Spin^h}_{n} \times  \Z/2 & n\equiv 1\pmod{4}.
\end{cases}
\]
\end{example}

\begin{proposition}\label{prop:PinCPlus}
    Recall the structure $\Pin^{\tilde{c}+}$ defined in \cref{def:PinCPlus}. Then for every odd $n$ we have $\SKK^{\Pin^{\tilde{c}+}}_n\cong \Omega^{\Pin^{\tilde{c}+}}_n$.
\end{proposition}
\begin{proof}
There is an odd Euler characteristic manifold with $\Pin^{\tilde{c}+}$ structure in every even dimension (\cref{pinc+structure}).
\end{proof}

\section{SKK groups in even dimensions}\label{sec:SkkEven}

The goal of this section is to express $\SKKxi_n$ in terms of $\Omega^\xi_n$ in the case that $n$ is even and the tangential structure $\xi$ is at least once stabilised. 
Recall that in even dimensions the \ref{SKKseq} takes on the form
\begin{equation}
\label{eq:evenexactseq}
\begin{tikzcd}
0\ar[r]&\Z\ar[r,"S^n_b"]&\SKK_n^{\xi}\ar[r]& \Omega^\xi_n \ar[r]& 0.
\end{tikzcd}
\end{equation}
Unlike in the odd-dimensional case, we do not need to assume that $\xi$ is twice stabilised, but we do need $\xi$ to be once stabilised so that the bordism group $\Omega_n^{\xi}$ is well defined.

Recall that the Euler characteristic and the signature, if applicable, are $\SKK$ invariants $\SKKxi_n \to \Z$. The following is a classical result. 
\begin{theorem}[{\cite[page 11-12]{ebert}}]
\label{thm:SKbookEvenDim}
    Let $n$ be even. Then the sequence
    \[\begin{tikzcd}
    0\ar[r]&\Z\ar[r]&\SKK^{SO}_n\ar[r]&\Omega^{SO}_n\ar[r]&0
    \end{tikzcd}\]
    splits by $\frac{\chi-\sigma}{2}$ if $n \equiv 0 \pmod 4$ and $\frac{\chi}{2}$ if $n \equiv 2 \pmod 4$.
\end{theorem}

For orientable $\xi$-structures we obtain the following corollaries.
\begin{corollary}\label{cor:orientableeven}
    Let $n$ be even and let $\xi\colon B_{n+1}\to BSO_{n+1}$ be an orientable tangential structure, i.e. $\xi$ factores through $BSO_{n+1}$. Then the SKK sequence 
        \[\begin{tikzcd}
        0\ar[r]&\Z\ar[r]& \SKK^{\xi}_n\ar[r]& \Omega^{\xi}_n\ar[r]& 0\end{tikzcd}\]      
         splits by $\frac{\chi-\sigma}{2}$ if $n \equiv 0 \pmod 4$ and $\frac{\chi}{2}$ if $n \equiv 2 \pmod 4$.
\end{corollary}
\begin{proof}
    This follows from the inheritance of splittings for the SKK sequences for $\xi$ and $BSO$, see \cref{lm:Inheritance}.
\end{proof}

Note that in particular \cref{cor:orientableeven} applies to the $k$-orientable structures $\BOr_k$ for $k>0$ as these factor through $BSO$.

Using that even-dimensional spheres have even Euler characteristic, we immediately obtain the following result.
\begin{theorem}\label{thm:slittingSKKevenEvenChi}
    Let $\xi\colon B_{n+1} \to BO_{n+1}$ be a tangential structure with the property that every $n$-dimensional closed $\xi$-manifold has even Euler characteristic. 
    Then for even $n$, half the Euler characteristic is a splitting of the \ref{SKKseq} and thus
    \[
    \SKKxi_n \to \Z \times \Omega^\xi_n \quad [M] \mapsto (\chi(M)/2, [M])
    \]
    is an isomorphism.
\end{theorem}

The following proposition extends our knowledge of SKK in even dimensions to the case where not every $\xi$-manifold has even Euler characteristic.
\begin{proposition}\label{prop:attempt}
Let $n$ be even and $\xi$ a once stabilised structure. Then there is an isomorphism
\[
\SKK^\xi_n \xrightarrow{\varphi} \Omega^\xi_n \times_{\Z/2} \Z, \quad [X] \mapsto ([X],\chi(X)),
\]
where $\Omega^\xi_n \times_{\Z/2} \Z$ denotes the pullback of groups along the maps $\Omega^\xi_n \to \Z/2$ and $\Z \to \Z/2$ given by the Euler characteristic modulo two and the mod 2 map respectively. 
The inverse is given by $([X],r)\mapsto [X]+\frac{r-\chi(X)}{2}[S_b^n]$. 
\end{proposition}
\begin{proof}
Recall that the Stiefel-Whitney number
\[
\left<w_n(M),[M]\right> \equiv \chi(M) \pmod 2
\]
is an unoriented bordism invariant and hence a $\xi$-bordism invariant.
Consider the commutative diagram
\begin{center}
\begin{tikzcd}
0 \arrow[r] & \Z \arrow[r,"\cdot 2"] & \Z \arrow[r,"\mod 2"] & \Z/2 \arrow[r] & 0
\\
0 \arrow[r] & \langle S^n_b \rangle \arrow[r] \arrow[u, "\cong"] & \SKK_n^\xi \arrow[r] \arrow[u, "\chi"] & \Omega^{\xi}_n \arrow[u, "\left<w_n {,-} \right>"] \arrow[r] & 0.
\end{tikzcd}
\end{center}
The right square is a pullback square of groups, since it induces an isomorphism on the kernels and the cokernels of its horizontal maps.
It follows that $\SKK^\xi_n $ is isomorphic to $ \Omega^\xi_n \times_{\Z/2} \Z, $ given by the map $\varphi([X] ) =([X],\chi(X))$.
The inverse is as described because \begin{align*}
\varphi\left([X]+\frac{r-\chi(X)}{2}[S_b^n]\right)&=\left([X]+\frac{r-\chi(X)}{2}[S_b^n],\chi([X]+\frac{r-\chi(X)}{2}[S_b^n])\right)
\\
&=([X],r).\qedhere
\end{align*}
\end{proof}

The above proposition is a useful tool to compute $\SKK$ groups in even dimensions concretely.
It moreover helps us to understand that even when there are closed $n$-dimensional $\xi$-manifolds with an odd Euler characteristic, it is still possible for the \ref{SKKseq} to split, as we show below. 
This corrects a mistake in \cite[Theorem 2.12.2(b)]{lorant}.

\begin{theorem}\label{thm:torsionWithOddEuler}
Let $n$ be even and $\xi$ a once stabilised structure.
    If there is a torsion class $[M] \in \Omega^{\xi}_{n}$ with $\chi(M)$ odd, then the \ref{SKKseq}  
    \[\begin{tikzcd}
0\ar[r]&\Z\ar[r]&\SKK_n^{\xi}\ar[r]& \Omega^\xi_n \ar[r]& 0
\end{tikzcd} \]    
    does not split. 
    Moreover, if $B_{n+1}$ has finitely generated homology in all degrees, then the converse holds: if all manifolds $M^n$ with odd Euler characteristic have infinite order in $\Omega^\xi_n$, then the same sequence splits non-canonically. 

    Furthermore if $\Omega^\xi_n$ is torsion free then the \ref{SKKseq} splits via the map $\frac{\chi}{2}$.
\end{theorem}
\begin{proof}
Assume the sequence splits.
    Considering \cref{prop:attempt}, a splitting is equivalent to a group homomorphism $\psi\colon \ker(\Omega^\xi_n \times \Z \to \Z/2) \to \Z$ with the property that $\psi(0, 2r) = r$.
    Now suppose $M^n$ is a closed $\xi$-manifold such that $k[M] = 0 \in \Omega^\xi_n$ for some non-zero $k \in \Z$.
    Without loss of generality, we can assume that $k$ is even.
    Then 
    \begin{align*}
    k\psi([M], \chi(M)) &= \psi(k[M], k \chi(M)) = \psi(0,k \chi(M)) = \frac{k \chi(M)}{2} 
    \\
    &\implies \psi([M], \chi(M)) = \frac{\chi(M)}{2} \in \Z. 
    \end{align*}
    We see that $\chi(M)$ has to be even.
    
    Conversely, assume that every manifold generating a torsion element of $\Omega^\xi_n$ has even Euler characteristic.
    Let $T$ denote the torsion subgroup of $\Omega^\xi_n$. 
    By a spectral sequence argument, the fact that $H_n(B;\Z)$ is finitely generated implies that $\Omega^\xi_n$ is finitely generated.
    Fix an isomorphism $\Omega^\xi_n \cong T \times \Z^m$. 
    Let $x_1,\cdots,x_m$ generators of $\Z^m$ and fix $\xi$-manifolds $M_1,\cdots, M_m$ such that $M_i$ represents $x_i$.  We can identify 
    \begin{align*}
        \Omega^\xi_n\times_{\Z/2}\Z &\cong (T\times  \Z^m)\times_{\Z/2}\Z
        \cong T \times (\Z^m \times_{\Z/2} \Z)
        \\
        & =\{t\in T,\sum_i{\alpha_i x_i}\in \Z^m,r\in \Z\mid \sum_i\alpha_i \chi(M_i)\equiv r \pmod 2\}.
    \end{align*} Here the second equality uses that every torsion element has even Euler characteristic.
    Define the map
    \begin{align*}
        &\psi\colon (T\times  \Z^m)\times_{\Z/2}\Z\to\Z\\
        &\left(t,\sum_i\alpha_ix_i, r\right)\mapsto \frac{r-\sum_i \alpha_i\chi(M_i)}{2}.
    \end{align*}
    The map $\psi$ is obviously well-defined and a homomorphism. 
    It is a splitting because the bounding sphere includes in $\Omega^\xi_n\times_{\Z/2}\Z$ as $(0,2)$ and $\psi(0,2)=1$.
\end{proof}

\begin{corollary}
    Let $n$ be even and let $\xi\colon B_{n+1} \to BO_{n+1}$ be a tangential structure.
    Suppose $\Omega^{\xi}_n$ is torsion.
    Then the \ref{SKKseq} splits if and only if every $n$-dimensional $\xi$-manifold has even Euler characteristic.
    In that case, a splitting is given by $\frac{\chi}{2}\colon \SKKxi_n \to \Z$.
\end{corollary}

\begin{corollary}\label{cor:unoriented_even_notsplit}
    For unoriented $\SKK$ groups, the \ref{SKKseq}
    \[
    \begin{tikzcd}
    0 \ar[r]& \Z \ar[r]& \SKK^O_n \ar[r]& \Omega^O_n \ar[r]& 0      
    \end{tikzcd}
    \]
    does not split.
\end{corollary}
\begin{proof}
    In every even dimension $n$, there is a manifold with odd Euler characteristic, giving us a surjective map $\Omega^O_n \to \Z/2$.
    Since $\Omega^O_n$ is purely torsion, the exact sequence
    \[
    \begin{tikzcd}
    0 \ar[r]& \Z \ar[r]& \Z \times_{\Z/2} \Omega^O_n \ar[r]& \Omega^O_n \ar[r]& 0      
    \end{tikzcd}
    \]
    never splits by \cref{thm:torsionWithOddEuler}.
\end{proof}

\begin{example}
    When $n = 2$ we have $\Omega^O_2 \cong \Z/2$, $\SKK^O_2\cong\Z$ and so the sequence becomes
    \[\begin{tikzcd}[column sep = 45]
            0 \ar[r]& \Z \ar[r,"\cdot 2"]& \Z \ar[r,"\mod 2"]& \Z/2 \ar[r]& 0.
    \end{tikzcd}
    \]
\end{example}

\begin{remark}
    We now show that our \Cref{thm:torsionWithOddEuler} shows that the \ref{SKKseq} splits for $\xi\colon BSO\to BO$, reproving part of the previously known \cref{thm:SKbookEvenDim}. For a fixed even $n$, we need to show that every orientable $n$-dimensional manifold which is torsion in $\Omega^{SO}_n$ has even Euler characteristic. Let $n\equiv 0\pmod{4}$ and let $M$ be an oriented manifold which is torsion the oriented bordism group. Then there is an orientable manifold $W$ bounding $\sqcup_kM$ for some non-zero integer $k$. But then the signature $\sigma(kM)=0$ and hence $\sigma(M)=0$. As $\chi(M)\equiv \sigma(n)\pmod{2}$, this proves the claim. For $n\equiv 2\pmod{4}$ we have shown previously that every $n$-dimensional orientable manifold has even Euler characteristic. 

    Note that our result \Cref{thm:torsionWithOddEuler}
    is weaker in the sense that it does not give a formula for a splitting of the SKK sequence in dimensions $n\equiv 0\pmod{4}$.
    
    \end{remark}

The $\Pin^{\pm}$, bordism groups are torsion \cite{anderson1969pin,giambalvo1973pin}. Therefore the splitting or otherwise of the SKK sequence for even $n$ depends only on the existence of a $\Pin^{\pm}$ manifold with odd Euler characteristic. The following corollary then follows from our discussion in \cref{subsec:Pin}.

\begin{corollary}\label{cor:pin_even_dim}
    The short exact sequence
    \[\begin{tikzcd}
0\ar[r]&\Z \ar[r]&\SKK^{\Pin^{\pm}}_n\ar[r]& \Omega^{\Pin^{\pm}}_n\ar[r]& 0
\end{tikzcd}\]
    splits for $\Pin^-$ if $n\equiv 4\pmod{8}$ and does not split for $n\equiv 0,2,6\pmod{8}$. 
    
    Furthermore, the sequence splits for $\Pin^+$ for $n\equiv 6\pmod{8}$ as well as $n=2, 10$, but it does not split for $n\equiv 0,4\pmod{8}$.
\end{corollary}
Note that we currently cannot resolve the status of the splitting of the sequence for $\Pin^+_n$ if $n\equiv 2\pmod{8}$ for $n\geq 18$, see \cref{subsec:Pin} and also \cite{PinPaper}.

\begin{example}\cite[section 5]{oscarpin} \label{ex:pin-dim2}
In dimension 2, every manifold admits a $\Pin^-$ structure as $w_1^2 + w_2 = 0$ by a Wu formula.
Since we have $\chi(\R \mathbb{P}^2) =1$, there exists a $\Pin^-$ manifold with odd Euler characteristic, and the SKK sequence does not split.
The bordism group is $\Omega^{\Pin^-}_2 \cong \Z/8$ \cite{kirbytaylor}.
We obtain an isomorphism
\[
\SKK_2^{\Pin^-} \cong \Z \times_{\Z/2} \Z/8 \cong \Z \times \Z/4, \quad (a,b) \mapsto \left(\frac{a-b}{2} \pmod 4, b \right)
\]
fitting in the non-split short exact sequence
\[\begin{tikzcd}
    0\ar[r]&\Z\ar[r]&\Z\times\Z/4\ar[r]&\Z/8\ar[r]&0.
\end{tikzcd}\]

\end{example}

\begin{remark}\label{rem:spinr}
    It would be interesting to study even-dimensional $\SKK$ groups for tangential structures that are not once stabilised. Such structures will be studied in \cite{KST}. For example, \cite{lorant} computes 
    that 
    \[
    \SKK_2^{\Spin^r_2} = 
    \begin{cases}
        \Z \times \Z/2 & r \text{ even,}
        \\
        \Z & r \text{ odd,}
    \end{cases}
    \]
    where $\Spin^r_2 \to SO_2$ is the $r$-fold cover. 
    It is known that these structures do not admit a stabilisation if $r >2$. 
\end{remark}

\section{Invertible TQFTs and \texorpdfstring{$\SKK$}{SKK}}
\label{sec:physics}

Topological quantum field theories (TQFTs) are an important object of study bridging the fields of geometry, algebraic topology and mathematical physics.
In the setting most closely related to this paper, a TQFT is defined as a symmetric monoidal functor $Z$ from the symmetric monoidal category $\Cob_{n-1,n}^\xi$ (\cref{def:cob}) to some target symmetric monoidal category $\mathcal{C}$, such as the category of vector spaces over the complex numbers with tensor product \cite{atiyahtft}. 
One particularly easy class of TQFTs are those that are invertible.

\begin{definition}
\label{def:invertible}
A TQFT $Z$ is called \emph{invertible} if for all objects $Y$ in the bordism category the object $Z(Y)$ is invertible under the tensor product in $\mathcal{C}$ and for all morphisms $X\colon Y_1 \to Y_2$, $Z(X)$ is an invertible morphism in $\mathcal{C}$.
\end{definition}

Invertible TQFTs in $n$ dimensions play an important role in physics, because they classify anomalies of $n-1$-dimensional quantum field theories \cite{freed2014anomalies,monnieranomalies} and are conjectured to classify $n$-dimensional symmetry-protected topological phases of matter \cite{kapustinturzillo, freedhopkins}.

Invertible TQFTs are closely related to $\SKK$ invariants via the restriction of the functor to the monoid of closed manifolds (called the partition function of the TQFT)\footnote{The authors learned this observation and many other considerations in this section from Kreck, Stolz and Teichner \cite{KST}.}.
For example, for $\mathcal{C}$ the category of complex vector spaces, an invertible TQFT assigns $\mathbb{C}$ to all objects and multiplication by a number to all bordisms.
We then observe that the partition function is an $\SKK$ invariant by a straightforward computation: \begin{equation*}
    \frac{|Z(M_1\cup_{\varphi} \ol{M_2})|}{|Z(M_1\cup_{\psi} \ol{M_2})|}=\frac{|Z(M_1)Z(C_{\varphi})Z(\ol{M_2})|}{|Z(M_1)Z(C_{\psi})Z(\ol{M_2})|}=\frac{|Z(C_{\varphi})|}{|Z(C_{\psi})|},
\end{equation*}
where $C$ denotes the mapping cylinders.

We now give a more abstract perspective on the appearance of $\SKK$ groups.
Note that a TQFT $Z$ is invertible if and only if lands in the maximal Picard groupoid $\mathcal{C}^\times$ contained in $\mathcal{C}$.
By the universal property of the groupoidification, an invertible TQFT factors uniquely through the groupoidification $\widehat{ \Cob}_{n-1,n}^\xi $ of $\Cob_{n-1,n}^\xi$.
We can therefore understand invertible TQFTs as maps between Picard groupoids\footnote{The groupoidification of a symmetric monoidal category with duals is automatically a Picard groupoid, where the inverse is given by the dual object.} $\widehat{\Cob}_{n-1,n}^\xi \to \mathcal{C}^\times$, which are 
well-understood by a theorem of Ho\`ang \cite{sinhthesis, johnsonosorno}.
When the target is the category of supervector spaces $\sVect_\C$\footnote{Supervector spaces are more desirable than ungraded vector spaces from the perspective of physics because they allow for the definition of the fermion parity operator $(-1)^F$. The interesting braiding of $\sVect$ corresponds to the dichotomy of Bose- versus Fermi statistics.}, equivalence classes of invertible field theories are in one-to-one correspondence with homomorphisms
\[
\pi_1 \widehat{ \Cob}_{n-1,n}^\xi \to \C^\times.
\]
This is because the Picard groupoid $\sline_\C \subseteq \sVect_\C$ of superlines has the property that it is a truncation of a particular spectrum (called the Brown-Comenetz dual of the sphere, see \cite[Section 5.3.]{freedhopkins}) that has the universal property that maps into it are in one-to-one correspondence with maps on $\pi_1$, and $\pi_1 \sVect^\times_\C \cong \C^\times$, see also \cite{schommerpriesinvertible}.

Applying the considerations in \cref{ap:ProofPi1ofCategories}, this explains the relevance of $\SKK$ groups for the study of invertible TQFTs:

\begin{theorem} \cite{KST}
\label{th:IFTsareSKK}
    Let $\xi\colon B_{n+1} \to BO_{n+1}$ be a tangential structure.
    Then equivalence classes of $n$-dimensional invertible TQFTs with target $\sVect_\C$ are in one-to-one correspondence with homomorphisms to $\C^\times$:
    \[
    ITQFT_n^\xi \cong \Hom(\SKKxi_n, \C^\times).
    \]
\end{theorem}

\begin{corollary}
\label{cor:partitionfunctiondeterminestheory}
An invertible TQFT is uniquely determined by its partition function. 
\end{corollary}

\begin{remark}
When the target category is not the category of super vector spaces, the result is slightly more complicated. 
However, there is an algebraic classification of morphisms between Picard groupoids \cite{KST}, which allows for a classification of invertible TQFTs with more general target categories, compare \cite{Carmenitft}. 
\end{remark}

From now on, let $\xi\colon B \to BO$ be a stable tangential structure.

\begin{definition} \cite{atiyahtft} \cite[Appendix G]{turaev}
\label{def:unitaryTQFT}
    A \emph{unitary TQFT}\footnote{The property that Euclidean QFTs obtain after Wick-rotating a Lorentzian 
    unitary quantum field theory is typically called reflection-positivity \cite{glimmjaffe,freedhopkins}. We will call such QFTs unitary independent of whether they are in Lorentzian or Euclidean signature and hope this will not lead to confusion.} is a symmetric monoidal dagger functor $\Cob^\xi_{n-1,n} \to \operatorname{sHilb}$ into the dagger category of super Hilbert spaces.
    Let $uITQFT_n^\xi$ be the group of unitary invertible TQFTs.
\end{definition}

We will not go into detail about dagger categories here.
In particular, we will not specify the dagger structure on the bordism category, referring to \cite{higherdagger} for a construction only requiring stability of $\xi$.
    However, \cref{def:unitaryTQFT} is equivalent to \cite[Definition 4.18]{freedhopkins} of a reflection-positive structure.
    We refer to \cite{luukthesis} for details.

\label{rem:unitaryITFTbordism}
Often in physics applications, invertible TQFTs are related to bordism groups instead of $\SKK$ groups.
The justification for this is that most important QFTs are unitary.
It has been shown that unitary invertible TQFTs correspond roughly to those homomorphisms $Z\colon \SKKxi_n \to \C^\times$ for which there exists a homomorphism $\Omega^\xi_n \to \C^\times$ such that the diagram
\[
\begin{tikzcd}
    \SKKxi_n \ar[r] \ar[d,"Z"] & \Omega^\xi_n \ar[dl, dashed]
    \\
    \C^\times &
\end{tikzcd}
\]
commutes.
More precisely, we have
\[
uITQFT^\xi_n \cong 
\begin{cases}
\Hom(\Omega^\xi_n, U(1)) & n \text{ odd,}
\\
\Hom(\Omega^\xi_n, U(1)) \times \R_{>0}  & n \text{ even}
\end{cases}
\]  
where the element of $\R_{>0}$ is the value assigned to the bounding sphere. 
Note that by the \ref{SKKseq}, the dashed line exists if and only if $Z(S^n_b) = 1$ and is unique in that case.
We will not get into these theorems here, see \cite[Theorem 8.29]{freedhopkins} for the theorem in the extended setting and \cite{yonekura} a $1$-categorical formulation without dagger categories.

However, non-unitary invertible TQFTs are also of physical interest.
For example, we expect them to offer a natural framework for describing non-Hermitian topological phases, which exhibit novel symmetry and topological structures beyond the conventional Hermitian paradigm~\cite{kawabata2019symmetry, nonhermitian}. Additionally, non-unitary operators arise intrinsically in the study of non-invertible symmetries, including generalized duality transformations such as those extending Kramers-Wannier duality~\cite{shao2023s, Li_2023}.
Non-unitary invertible TQFTs also play a role in the study of global anomalies of non-unitary quantum field theories~\cite{chang2021exotic,TachikawaYonekura}. Therefore it is interesting to study the whole group $ITQFT_n^\xi$ of invertible TQFTs and how it relates to $uITQFT^\xi_n$, which is what we do in the current work.

\hfill 

We now remark on the physical interpretation of the stable tangential structure $\xi\colon B \to BO$. 
Consider a quantum system with a certain internal symmetry group $G$ possibly containing time-reversal symmetry, such as one of the classes in the tenfold way \cite{altlandzirnbauer,kitaev2009periodic}, see \cite[Section 2.1]{stehouwermorita} for a mathematical approach to such symmetry groups.
Then, there is an associated construction of a structure group $H_n(G) \to O_n$ such that spacetimes in the QFT come equipped with a tangential $H_n(G)$-structure, see \cite[Table 9.2.1, Remark 9.36]{freedhopkins} and \cite[Section 3.3]{luukasreflection}.
This gives in the colimit our desired stable tangential structure $\xi\colon BH(G) \to BO$.
In physics language, TQFTs with this tangential structure $\xi$ should be thought of as TQFTs with internal symmetry $G$ by coupling to background $G$-gauge fields.
The computation of $\SKKxi_n$ for this $\xi$ is therefore related to the classification of (possibly non-unitary) topological phases protected by $G$ in spacetime dimension $n$.

\subsection{Odd-dimensional non-unitary invertible TQFTs and Kervaire TQFTs}

We will now explain the consequences of our work to odd-dimensional non-unitary invertible TQFTs.
Our primary example of a non-unitary invertible TQFT will be the Kervaire TQFT (\cref{def:kervaireTFT}).
We start with a motivating example:

\begin{example}
\label{ex:tachikawa1d}
    As explained in \cite[Appendix E.1]{TachikawaYonekura}, there exists a QFT in one spacetime dimension of which the low-energy effective field theory is the following invertible TQFT.
    Consider the unique symmetric monoidal functor $Z\colon \Cob_{0,1}^{SO} \to \sVect_\C$ which assigns the odd line to the point independent of the orientation. 
    Note that this theory does not use a spin structure on spacetime, so in this sense it is an `integer spin theory'.
    However, it does not factor through $\Vect_\C$ and so the theory is not bosonic; $(-1)^F = -1$ on the state space.
    In particular, the theory violates spin-statistics and therefore is not unitary \cite[Section 11]{freedhopkins}.
    Explicitly, one can compute the partition function to be $Z(S^1) = -1$. We note that this invertible field theory corresponds to the non-trivial element of $\Hom(\SKK^{SO}_1,\C^\times) \cong \Z/2$.
    The invariant is given by the Kervaire semi-characteristic over any field\footnote{In this dimension, Kervaire semi-characteristics over different fields agree.} $F$, resulting 
    in the partition function
    \[
    Z_{\kerv_F}(X) = (-1)^{\kerv_{F}(X)} = (-1)^{\dim H_0(X;F)} = (-1)^{|\pi_0(X)|}
    \]
    on a one-dimensional closed oriented manifold $X$.
\end{example}

A generalisation of the above example to an invertible TQFT violating spin-statistics in any spacetime dimensions equal to $1$ modulo $4$ has been considered before for the case of $B = BSO$ and $F = \Q$ \cite[Example 6.15]{freed2019lectures}.
However, one of the main observations of our work is that the Kervaire semi-charactistic partition function generalises best for the field $F = \Z/2$. Indeed, it generalises to spacetime dimensions equal to $3$ modulo $8$ for spin theories:
\begin{example}
\label{ex:3dspin}
Consider neutral fermions with no further symmetries in dimension $2 + 1$, corresponding to a class D topological superconductor on the condensed matter side.
On the TQFT side this corresponds to the tangential structure $\xi\colon B\Spin \to BO$ and so to classify non-unitary phases of matter we have to compute $\SKK^{\Spin}_3$.
For this, recall that $\Omega^{\Spin}_3 = 0$ and every four-dimensional spin manifold has even Euler characteristic so that $\SKK^{\Spin}_3 \cong \Z/2$.
By \cref{thm:mainthmforOddDimensionKOrient}, in this dimension $B\Spin$ satisfies the assumptions in \cref{def:kervaireTFT}, so the Kervaire TQFT exists.
We conclude that there is a single non-trivial invertible field theory with partition function
\[
Z_{\kerv}(X) = (-1)^{\dim H^0(X;\Z/2) + \dim H^2(X;\Z/2)}.
\]
    More generally, if $n \equiv 3 \pmod 8$, there exists a non-unitary invertible spin TQFT $Z_{\kerv}$ with partition function
    \[
    Z_{\kerv}(X) = (-1)^{\kerv_{\Z/2}(X)}.
    \]

     This example only works for the field $\Z/2$, because for any other characteristic the Kervaire semi-characteristic is not an SKK invariant of $\Spin$ manifolds in dimension $3$, see \cref{rem:KervCharDifferentCharacteristics} and \cref{rm:differentKervaire}.
    In particular, there is no three-dimensional invertible spin TQFT $Z_{\kerv_\Q}$ with partition function $(-1)^{\kerv_\Q(M)}$.
\end{example}

We are therefore led to define the Kervaire TQFT as a theory of which the partition function arises from the Kervaire semi-characteristic over $\Z/2$.
Given a tangential structure, this TQFT exists in a certain range of spacetime dimensions:

\begin{definition}
\label{def:kervaireTFT}
    Let $n = 2k+1$ be an odd spacetime dimension, and $\xi\colon B \to BO$ a stable tangential structure such that 
 for every $\xi$-manifold $W$ with boundary 
 \[
 \rank_{\Z/2}\left(H_k(W;\Z/2) \xrightarrow{j_*} H_k(W,\partial W; \Z/2)\right)
 \]
 is even. 
 The \emph{$n$-dimensional $\xi$-Kervaire TQFT} is the unique invertible TQFT with domain $\Cob_{n-1,n}^\xi$ and target $\sVect_\C$ which has as its partition function
    \[
    Z_{\kerv}(X^n) = (-1)^{\kerv_{\Z/2}(X)},
    \]
    where $\kerv_{\Z/2}(X)$ is the Kervaire semi-characteristic from Definition \ref{def:kervaire}.
\end{definition}

\begin{remark} 
    By \cref{prop:iffkervaire}, $\kerv_{\Z/2}$ is an $\SKKxi$ invariant under the stated assumptions on $\xi$.
    By Corollary \ref{cor:partitionfunctiondeterminestheory}, the partition function in \cref{def:kervaireTFT} uniquely defines the Kervaire TQFT.
\end{remark}

\begin{remark}
The Kervaire TQFT is not unitary because $Z_{\kerv}(S^n) = -1$.
\end{remark}

\begin{remark}
    For certain odd spacetime dimensions $n > 1$ and stable tangential structures $\xi$, it happens that Kervaire semi-characteristics over different fields yield well-defined but non-isomorphic invertible TQFTs.
In dimension $5$ and $B = BSO$ for example, we can define an invertible TQFT $Z_{\kerv_\Q}$ with partition function $(-1)^{\kerv_\Q(X)}$, which is not isomorphic to the oriented TQFT $Z_{\kerv}$, see \cref{rm:differentKervaire}.
\end{remark}

\begin{remark}
Given an $(n-1)$-dimensional closed manifold $Y$, the Kervaire TQFT assigns to $Y$ the even line if $Y$ has even Euler characteristic and the odd line if $Y$ has odd Euler characteristic.
    Indeed, the super dimension of $Z_{\kerv}(Y)$ is the trace of the identity computed as
    \[
    Z_{\kerv} (Y \times S^1) = (-1)^{\kerv_{\Z/2}(Y \times S^1)} = (-1)^{\chi(Y)},
    \]
    since 
    \begin{align*}
        \kerv_{\Z/2}(Y\times S^1)& = \sum_{i=0}^{n/2}\dim(H_{2i} (Y\times S^1,\Z/2))
        \\
        & = \sum_{i=0}^{n/2}\dim(H_{2i+1}(Y;\Z/2))+\dim(H_{2i}(Y;\Z/2))
        \\
        &\equiv \chi(Y) \pmod 2.
    \end{align*}
\end{remark}

\begin{remark}
    It would be interesting to compare our description of Kervaire TQFTs with the index-theoretic construction of  non-unitary TQFT in certain odd dimensions given in \cite[Appendix E.3]{TachikawaYonekura}.
\end{remark}

\begin{example}
Consider charged fermions (class A), which corresponds to the tangential structure $\xi\colon B\Spin^c \to BO$.
Since $\C \mathbb{P}^2$ is a $\Spin^c$ manifold with odd Euler characteristic and $\Omega^{\Spin^c}_3 =0$, we have $\SKK^{\Spin^c}_3 = 0$.
Therefore there are no non-trivial invertible $\Spin^c$ TQFTs in spacetime dimension $3$. 
In particular, there are no invertible $\Spin^c$ TQFTs of which the partition function is a Kervaire semi-characteristic.
\end{example}

\cref{conjecture} translates in the language of the current chapter to the following.

\begin{expectation}
\label{expectation}
Let $G$ be an internal symmetry group and $n$ an odd spacetime dimension.
The group of invertible TQFTs with structure group $H(G)$ is a direct sum of unitary invertible TQFTs plus potentially one non-unitary $\Z/2$-summand. 
 This extra $\Z/2$ appears if and only if every $(n+1)$-dimensional $H(G)$-manifold has even Euler characteristic.
\end{expectation}

\begin{remark}
A $\Z/2$-subgroup splitting off the non-unitary summand in \cref{expectation} is not always given by the Kervaire TQFT.
This is for example the case for $H(G) = \Pin^+$ in spacetime dimension one, see \cref{Ex:pindim1}. 
In that case, there is a single non-trivial invertible TQFT, which happens to be non-unitary.
Its partition function is $1$ on the periodic circle and $-1$ on the anti-periodic circle
One useful fact to determine the analogue $Z$ of the Kervaire TQFT for general $\xi$ and dimension $n$, is the following anomaly-inflow principle: the partition function on an $n$-dimensional $\xi$-manifold $Y$ that bounds a $(n+1)$-dimensional $\xi$-manifold $X$ should be given by $Z(Y) = (-1)^{\chi(X)}$, see \cref{thm:iffForMfldInvariant}.
\end{remark}

\subsection{Even-dimensional non-unitary invertible TQFTs}

Thus far, we have considered non-unitary invertible TQFTs in odd spacetime dimensions.
    In even spacetime dimensions $n$, our results imply roughly that the only non-unitary invertible TQFTs are `Euler TQFTs'. 
    
    \begin{definition}[\cite{quinn1995lectures},\cite{freed2006setting}]
        Given any stable tangential structure $\xi$, the \emph{Euler TQFT} $Z_\lambda$ corresponding to the non-zero complex number $\lambda \in \C^\times$ is the invertible TQFT with partition function 
    \[
    Z_\lambda (X^n) = \lambda^{\chi(X)}.
    \]
    \end{definition}
    
        Applying $\Hom(-, \C^\times)$ to the \ref{SKKseq} gives a short exact sequence
    \begin{equation}
    \label{nonunitaryses}
    \begin{tikzcd}
          0 \ar[r]& \Hom(\Omega^\xi_n, \C^\times) \ar[r]& 
           \Hom(\SKKxi_n, \C^\times) 
          \ar[r] & \C^\times \ar[r]& 0,
    \end{tikzcd}
    \end{equation}
    where the last map is given by evaluating on the bounding sphere. 
    This follows from the fact that $\C^\times$ is an injective abelian group, so that $\Hom(-, \C^\times)$ is an exact functor.
     This sequence is convenient to relate unitary and non-unitary invertible TQFTs, see \cref{rem:unitaryITFTbordism}.
    Its potential non-splitness is caused by the fact that $Z_\lambda$ for $\lambda = -1$ is a bordism invariant. 
    More precisely, since exact functors preserve finite limits, it follows by \cref{prop:attempt} that
    \[
    \begin{tikzcd}
    \C^\times \ar[d,"Z_\lambda"] & \Z/2 \ar[l,hookrightarrow] \ar[d,"Z_{-1}"]
    \\
        \Hom(\SKKxi_n,\C^\times) & \Hom(\Omega^\xi_n,\C^\times) \ar[l]
    \end{tikzcd}
    \]
    is a pushout square.
    However, note that $Z_{-1}$ is the trivial TQFT if and only if $\xi$-manifolds have even $\chi$. 
    It also follows by \cref{thm:slittingSKKevenEvenChi} that in that case $\lambda \mapsto (X \mapsto \lambda^{\chi(X)/2})$ splits the sequence \cref{nonunitaryses} on the right.

\begin{example}
    Consider a $2$-dimensional system of neutral fermions with a time-reversal symmetry that squares to one.
    In that case, the structure group is known to be $\xi\colon B\Pin^- \to BO$. 
    It follows from \cref{ex:pin-dim2} that two-dimensional $\Pin^-$ invertible field theories fit into the non-split short exact sequence
    \[
    \begin{tikzcd}
        0 \ar[r] & \Hom(\Omega^\xi_n, \C^\times) \ar[r] \ar[d,equals] & \Hom(\SKKxi_n, \C^\times) \ar[r]  \ar[d,equals] & \C^\times \ar[r] & 0.
        \\
        & \Z/8 & \C^\times \times \Z/4 &  & 
    \end{tikzcd}
    \]
    In particular, the unitary invertible TQFTs do not form a direct summand inside the group of all invertible TQFTs in this example.
\end{example}

\subsection{Classification of not necessarily unitary invertible TQFTs}
\label{subsection:classification_physics}

 In this subsection, we compute the group of invertible TQFTs in spacetime dimensions 1-5 for many tangential structures of physical interest, see \cref{tab:physics}.
In particular, we compute the groups for all the tangential structures corresponding to the tenfold way, enriching the computations of \cite{freedhopkins} to the non-unitary setting.
This section consists of three parts
\begin{enumerate}[label=(\roman*)]
    \item Firstly, we briefly explain the tenfold way in the setting of this paper, as the authors learned from Peter Teichner;
    \item We then apply and amend our computations in the previous sections to compute the relevant $\SKK$-groups;
    \item We finally present our results in \cref{tab:physics}.
\end{enumerate}

The tenfold way is an organising principle on topological phases of matter, categorizing symmetries into ten important classes \cite{altlandzirnbauer, kitaev2009periodic}.
Mathematically, these ten classes are related to the classification of super division algebras \cite{mooretenfold, baez2020tenfold}:
recall that a superalgebra is a $\Z/2$-graded algebra $A = A_0 \oplus A_1$ such that the grading is respected by the multiplication.

\begin{definition}
    A \emph{super division algebra} $D$ is a superalgebra such that every homogeneous element is invertible.
\end{definition}

\begin{theorem}[\cite{wallgradedbrauer}]
    There are ten isomorphism classes of real super division algebras.
\end{theorem}

    Let $D$ be a real super division algebra.
    Let $G(D)$ be the quotient of the group $D_{hom}$ of homogeneous elements of $D$ by the subgroup $\R^\times$ of nonzero scalars.
    Note that $G(D)$ admits an extension 
    \[
    \begin{tikzcd}
    1 \ar[r] & \Z/2 \ar[r] & D_{hom}/{\R_{>0}} \ar[r] & G(D) \ar[r] & 1,
    \end{tikzcd}
    \]
    which defines a map $BG(D) \to B^2 \Z/2$.
The supergrading on $D$ induces a homomorphism $G(D) \to \Z/2$ and together these define a map 
\[
BG(D) \to B\Z/2 \times B^2 \Z/2 = \pi_{\leq 2} BO.
\]

\begin{definition}(compare \cite[(10.12)]{freedhopkins})
\label{tenfoldtangential}
    The \emph{tenfold way tangential structure} for the super division algebra $D$ is the stable tangential structure given by the homotopy pullback
    \begin{equation}
        \begin{tikzcd}
            BH(D) \ar[r] \ar[d] & BG(D) \ar[d]
            \\
            BO \ar[r] & \pi_{\leq 2} BO.
        \end{tikzcd}
    \end{equation}
\end{definition}

\newgeometry{bottom=1in}
\begin{landscape}
\begin{table}[h!]
\begin{adjustbox}{scale=0.85}
{\renewcommand{\arraystretch}{1.2}
\begin{tabular}{|l|l||l|l|l||l|l|l|l|l|}\hline
 \textbf{Particle content} & $\xi$-structure & $\langle S_b^1 \rangle$ & $\langle S_b^3 \rangle$ & $\langle S_b^5 \rangle$ 
 & $ITQFT_1$ & $ITQFT_2$ & $ITQFT_3$ & $ITQFT_4$ & $ITQFT_5$\\ \hline\hline
bosons & $BSO$ & $\Z/2$ & $0$ & $\Z/2$ 
& $\Z/2$ & $\C^\times$ & $0$ & $(\C^\times)^2 $ & $\Z/2 \times \Z/2$ \\ \hline
bosons with TRS & $BO$ & $0$ & $0$ & $0$  
& $0$ & $\textcolor{blue}{\C^\times}$ & $0$  & $\textcolor{blue}{\C^\times \times \Z/2}$ & $\Z/2$\\ \hline
charged fermions (class A) & $B\Spin^c$ & $\Z/2$ & $0$ & $\Z/2$  
& $\Z/2$ & $(\C^\times)^2$ & $0$ & $(\C^\times)^3$ & $\Z/2$ \\ \hline
charged fermions with & $B\Pin^c$ & $0$ & $0$ & $0$ 
& $0$ & $\textcolor{blue}{\C^\times \times \Z/2}$ & $0$ & $\textcolor{blue}{\C^\times \times \Z/8}$ & $0$ \\
sublattice symmetry (class AIII)&&&&&&&&&\\ \hline
neutral fermions (class D) & $B\Spin$ & $\Z/2$ & $\Z/2$ & $\Z/2$ 
& $\Z/2 \times \Z/2$ & $\C^\times \times \Z/2$ & $\Z/2$ & $(\C^\times)^2$ & $\Z/2$ \\ \hline
neutral fermions with  & $B\Pin^+$ & $\Z/2$ & $0$ & $\Z/2$  
& $\Z/2$ & $\C^\times \times \Z/2$ & $\Z/2$ & $\textcolor{blue}{\C^\times \times \Z/8}$ & $\Z/2$  \\
 TRS squaring to $(-1)^F$ (class DIII) &&&&&&&&&\\ \hline
neutral fermions with  & $B\Pin^-$ & $0$ & $\Z/2$ & $0$ 
& $\Z/2$ & $\textcolor{blue}{\C^\times \times \Z/4}$ & $\Z/2$ & $\C^\times$ & $0$ \\ 
TRS squaring to $1$ (class BDI)&&&&&&&&& \\ \hline
charged fermions with  & $B\Pin^{\tilde{c}+}$ & $0$ & $0$ & $0$ 
& $0$ & $(\C^\times)^2$ & $\Z/2$ & $\textcolor{blue}{\C^\times \times (\Z/2)^2}$ & $0$ \\ 
TRS squaring to $(-1)^F$ (class AII) &&&&&&&&& \\ \hline
charged fermions with  & $B\Pin^{\tilde{c}-}$ & $0$ & $0$ & $0$
& $0$ & $\textcolor{blue}{\C^\times \times \C^\times}$ & $0$ & $\textcolor{blue}{\C^\times}$ & $0$ \\ 
TRS squaring to $1$ (class AI)&&&&&&&&& \\ \hline
fermions without SOC (class C) & $BG^0 = B\Spin^{h}$ & $\Z/2$ & $0$ & $\Z/2$ 
& $\Z/2$ & $\C^\times$ & $0$ & $(\C^\times)^3$ & $\Z/2 \times (\Z/2)^2$ \\ \hline 
fermions with TRS squaring to $1$, & $BG^+ = B\Pin^{h+}$ & $0$ & $0$ & $0$ 
& $0$ & \textcolor{blue}{$\C^\times$} & $0$ & \textcolor{blue}{$\C^\times \times \Z/4$} & $\Z/2$ \\ 
without SOC (class CI) &&&&&&&&&\\ \hline
fermions with TRS squaring to $(-1)^F$, & $BG^- = B\Pin^{h-}$ & $0$ & $0$ & $0$ 
& $0$ & $\textcolor{blue}{\C^\times}$ & $0$ & $\textcolor{blue}{\C^\times \times \Z/2^2}$ & $(\Z/2)^2$\\
 without SOC (class CII) &&&&&&&&&\\ \hline
\end{tabular}}
\end{adjustbox}
\caption{This table shows the group of all invertible TQFTs for the symmetry classes listed on the left in dimensions 1-5.
    In the first column, $(-1)^F$ refers to the fermion parity operator, and we used the abbreviations TRS for time-reversal symmetry and SOC for spin-orbit coupling. See \cite[Proposition 9.4, Proposition 9.16 and Tables (9.24), (9.25)]{freedhopkins} for the notation of the stable tangential structures in the second column and a translation with the first column.
    In columns 3-5 we display the subgroup of $\SKKxi_n$ generated by the bounding sphere in odd dimensions, which agrees with the quotient of the group of invertible TQFTs by the subgroup of unitary invertible TQFTs.
    In even dimensions, the bounding sphere always generates a $\Z$ and therefore this quotient is always $\C^\times$, although it may or may not split off as a subgroup of $ITQFT_n^\xi$. Columns 6-10 show our computations of $ITQFT_n^\xi$ in dimensions 1-5. We coloured the nonsplit cases \textcolor{blue}{in blue}.
    We refer the reader to \cref{rem:tableCaption} for details on how we arrived at the results displayed here.
}\label{tab:physics}
\end{table}

\end{landscape}
\restoregeometry

\subsubsection{TQFTs with all structure groups of \cref{tab:physics} except $\Pin^{\tilde{c}-}$}\label{rem:tableCaption}

The entries of \cref{tab:physics} for odd dimensions are all either consequences of what is proven in the text about $\Spin,\Pin^{\pm}$, $\Pin^{\tilde{c}+}$ and $\Spin^c$, or in some cases consequences of the fact that 
\begin{itemize}
    \item every $\Pin^{\pm}$ manifold is $\Pin^{c}$;
    \item every $\Spin^c$-manifolds is $\Pin^c$;
    \item every $\Pin^{\tilde{c}\pm}$ manifold is $\Pin^{h\pm}$;
    \item and every $\Spin^c$ manifold is $\Spin^h$.
\end{itemize}
In particular, the entries for class AII are a consequence of \cref{pinc+structure} showing that there is a $\Pin^{\tilde{c}+}$ manifold with odd Euler characteristic in every even dimension.
All $\Z/2$-quotients split by mapping the generator to the Kervaire TQFT over $\Z/2$, as a consequence of inheritance of splittings. 

In even dimensions, we apply our splitting result \cref{prop:attempt} to get an explicit expression for $\SKK^\xi_n$.
For this, we need enough information about the Euler characteristic mod 2 map $\Omega^\xi_n \to \Z/2$ to compute the pullback $\Omega^\xi_n \times_{\Z/2} \Z$.
However, we can use some tricks to obtain the isomorphism type of $\SKK^\xi_n$.
If the Euler characteristic of $\xi$-manifolds in the given dimension is always even, then $\SKK^\xi_n \cong \Z \times \Omega^\xi_n$ and we are done. Hence assume instead that $\Omega^\xi_n \to \Z/2$ is surjective. There are some cases where no further analysis is required:
\begin{enumerate}[label=(\roman*)]
    \item If $\Omega^\xi_n \cong \Z/2^k$, then there is only one surjective homomorphism to $\Z/2$ and we readily compute $\SKK^\xi_n \cong \Z \times \Z/2^{k-1}$;
    \item If $\Omega^\xi_n \cong (\Z/2)^k$, then there are many surjective homomorphisms to $\Z/2$, but they are all related by a self-automorphism of $(\Z/2)^k$. 
    It follows that $\SKK^\xi_n \cong \Z \times (\Z/2)^{k-1}$.
\end{enumerate}

In the case $\Omega^{\Pin^c}_4 \cong \Z/8 \times \Z/2$, there are non-isomorphic possible extensions.
However, \cite[Theorem 0.2(b)]{bahrigilkey} shows that $\R \mathbb{P}^4$ and $\C \mathbb{P}^2$ are generators of $\Z/8$ and $\Z/2$ factor respectively.\footnote{
The work \cite{giambalvo1973pin} was used to obtain the results in \cite{bahrigilkey} and \cite{kirbyTaylorPinPlus}
pointed out some mistakes in \cite{giambalvo1973pin}. However, this has no consequences for the generators of $\Omega^{\Pin^c}_4$ we need. This can independently be checked by an Adams spectral sequence argument \cite{beaudrycampbell}.}
It follows that $\chi \pmod 2$ is the sum modulo two $\Z/8 \times \Z/2 \to \Z/2$ and so $\SKK^{\Pin^c}_4 \cong \Z \times \Z/8$.

To determine $\SKK^{G_+}_4$, we use the fact that the explicit generators of $\Omega^{G_+}_4 \cong \Z/4 \times \Z/2$~\cite{freedhopkins} are known \cite[Claim 3]{guo2018time} (which we learned from \cite[Lemma A.29]{debray2023bosonizationanomalyindicators21d}).
It follows that the Euler characteristic modulo two homomorphism 
$$\Omega^{G^+}_4\cong\Z/4\oplus \Z/2\xrightarrow{+} \Z/2$$
is given by the sum modulo two.
    We obtain $\SKK^{G^+}_4 \cong \Z \times \Z/4$.

The only remaining open case in the tenfold way is $\SKK^{\Pin^{\tilde{c}-}}_2$.
We will provide a spectral sequence argument to determine the Euler characteristic homomorphism $\Omega^{\Pin^{\tilde{c}-}}_2 \to \Z/2$ in the next section.
After this computation, we obtain the classification of non-unitary invertible field theories for all groups in the tenfold way in spacetime dimensions up to $5$, as displayed in \cref{tab:physics}.

\subsubsection{TQFTs with structure group $\Pin^{\tilde{c}-}$}

Recall that the structure $B\Pin^{\tilde{c}-}$ is defined as the homotopy pullback of the following diagram.
\[\begin{tikzcd}
B\Pin^{\tilde{c}-}\ar[r]\ar[d]&BO_2\ar[d]\\
BO\ar[r]&(BO)_{\leq 2}
\end{tikzcd}\]
 where $(BO)_{\leq 2}$ is the second  Postnikov stage.

 We have the following result:

 \begin{proposition}[{\cite[Theorem 9.87]{freedhopkins}}]
     We have the following abstract isomorphism
     \[\Omega_2^{\Pin^{\tilde{c}-}}\cong \Z\oplus\Z/2. \]
 \end{proposition}
 \noindent We now want to determine the Euler characteristic map modulo 2 \[\chi\colon\Omega_2^{\Pin^{\tilde{c}-}}\to \Z/2.\]
 The following theorem is motivated by studying the edge homomorphism in the James spectral sequence~\cite{petethesis} for the fibration
\[
B\Spin \to B\Pin^{\tilde{c}-} \to BO_2.
\]

\begin{theorem}
\label{th:pinctildebordim}
There is a well defined isomorphism $\varphi\colon\Omega^{\Pin^{\tilde{c}-}}_2 \to \Z \times \Z/2$ given as follows.
    Let $(M,E)$ be a $2-$dimensional $\Pin^{\tilde{c}-}$-manifold, i.e. a $2$-dimensional real vector bundle $E \to M$ together with a trivialisation of $w_1(E) + w_1(M)$ and a trivialisation of $w_2(M) + w_1(E)^2 + w_2(E)$.
    Then define $\varphi$ by
    \begin{align*}
        [M,E] \mapsto \left(\frac{1}{2} \int_M e(E),\quad \chi(M) \pmod 2\right)\in \Z \times \Z/2,
    \end{align*}
    where $e(E) \in H^2(M; \Z^{w_1(E)})$ is the twisted Euler class.
\end{theorem}
\begin{proof}
    Note that the classes $w_1(E) = w_1(M)$ give the same twisted coefficient system, and so we can indeed integrate $e(E)$ over $M$.
    For surfaces we have $w_1(M)^2=w_2(M)$ and so a $\Pin^{\tilde{c}-}$ manifold $(M,E)$ has to have $w_2(E)=0$. 
    Since $\int_M e(E)\equiv \int_M w_2(E)\pmod{2}$, we get that $\int_M e(E)$ is even.
    We see that the provided invariants are given by integrating characteristic classes, so they are bordism invariants\footnote{The proof of this fact is analogous to \cite[Theorem 4.9]{milnorstasheff}.}.

    Finally, we prove that $\varphi$ is surjective. If $M$ is a surface, we have 
\[
H^2(M; \Z^{w_1(M)}) \cong \Z \twoheadrightarrow \Z/2 \cong H^2(M; \Z/2).
\]
A two-dimensional real vector bundle $E$ is classified by its Euler class $e(E) \in H^2(M; \Z^{w_1(E)})$.
For this to form a $\Pin^{\tilde{c}-}$-structure, we need 
$w_1(E) = w_1(M)$ and 
\[
e(E) \pmod 2 = w_2(E) \overset{?}{=} w_1(M)^2 + w_2(M) = 0.
\]
Therefore $e(E)$ can be taken to be any even integer $2n$.
The invariant of $(M,E)$ is $(n, \chi(M) \pmod 2) \in \Z \times \Z/2$.
By taking $M$ to have either even or odd Euler characteristic we have realised the whole codomain of $\varphi$.
\end{proof}

\begin{corollary}
    $\SKK^{\Pin^{\tilde{c}-}}_2 \cong \Z \times \Z.$
\end{corollary}
\begin{proof}
    It follows from \cref{th:pinctildebordim} that the Euler characteristic modulo two homomorphism
    \[
    \Omega^{\Pin^{\tilde{c}-}}_2 \cong \Z \times \Z/2 \to \Z/2
    \]
    is given by projection onto the second factor.
    The result follows by \cref{prop:attempt}.
\end{proof}

\subsection{Continuous invertible TQFTs}

Let $\xi\colon B \to BO$ be a stable tangential structure.
    If $\SKKxi_n$ is finitely generated, we can write the group of discrete invertible TQFTs as
    \begin{equation*}
    \Hom(\SKKxi_n,\C^\times) \cong (\C^\times)^k \times  T,
    \end{equation*}
    where $T$ is a finite torsion group abstractly isomorphic to the torsion in $\SKKxi_n$.

In practice, we often want to think of the $\C^\times$ terms as forming a continuous family of invertible TQFTs, hence all sitting in the same deformation class.
To take this Euclidean topology of $\C^\times$ into account, we will generalise the previous considerations from `discrete invertible TQFTs' to `continuous invertible TQFTs', see \cite[Ansatz 5.14 and Ansatz 5.26]{freedhopkins}.

    For this, it is convenient to consider the generalisation of Atiyah's definition of a TQFT to a general target symmetric monoidal $(\infty, 1)$-category $\mathcal{C}$ by requiring a TQFT to be 
a symmetric monoidal functor
\[
Z\colon \Bord_{n-1, n}^\xi \to \mathcal{C}
\]
from the symmetric monoidal $(\infty, 1)$-category of cobordisms $\Bord_{n-1, n}^\xi$ to $\mathcal{C}$. 

Generalising \cref{def:invertible}, a TQFT is \emph{invertible} if it lands in $\mathcal{C}^{\times} \subseteq \mathcal{C}$, the maximal Picard sub-$\infty$-groupoid.
We specialise to the case where the target is the Picard $(\infty,1)$-category of super lines:

\begin{definition}
\label{def:continuousITQFT}
Let $\sline_{\C}^{cts}$ be the $(\infty,1)$-category in which objects are complex one-dimensional super vector spaces and morphisms are invertible linear maps with the Euclidean topology.
A \emph{continuous invertible TQFT} is a TQFT with target $\sline_{\C}^{cts}$.
\end{definition}

By the universal property of $\infty$-groupoidification\footnote{The $\infty$-groupoidification of the bordism category is automatically a Picard $\infty$-groupoid because the bordism category admits duals.} $\|.\|$, a continuous invertible TQFT is equivalent to a map of Picard $\infty$-groupoids
\[
\|\Bord^\xi_{n-1,n}\| \to \sline_{\C}^{cts}.
\]
We can then identify a Picard $\infty$-groupoid with its corresponding infinite loop space (or equivalently the corresponding connective spectrum) to translate the problem of classifying invertible field theories into a problem in stable homotopy theory.
By the appropriate generalisations of the Galatius-Madsen-Tillmann-Weiss theorem \cite{GMTW, nguyen2017infinite, schommerpriesinvertible}, it is known that $\|\Bord_{n-1,n}^\xi\|$ corresponds to the connective cover of the Madsen-Tillmann spectrum $\Sigma MT\xi$.
Therefore, an invertible TQFT is equivalent to a map of connective spectra $\Sigma MT\xi \to \sline_{\C}^{cts}$, see \cite[Section 2.5]{luriecobordismhyp} for more on this perspective.
It is hard to make general statements about maps of spectra,
but in cases where the homotopy groups of $MT\xi$ are known, it is possible to compute simple examples.
This may involve understanding unstable homotopy groups of $MT\xi$ higher than $\SKKxi_n$, 
which can be understood as vector field bordism groups with multiple vector fields \cite{bokstedtsvane},
(see also \cite[Lemma 3.13]{reutterschommerpries} and the discussion above that):

\begin{theorem}
\label{th:ctsITFT}
    Equivalence classes of $n$-dimensional continuous invertible (not necessarily unitary) TQFTs are non-canonically isomorphic to the sum of the torsion subgroup of $\SKKxi_n$ and the free part of $\pi_{1}  MT \xi_n$, the $(n+1)$-dimensional $\xi_n$-bordism group with two linearly independent vector fields.
\end{theorem}
\begin{proof}
    It follows by a $k$-invariant computation that $\sline_{\C}^{cts}$ is the connective cover $\pi_{\geq 0} \Sigma^{2} I\Z$ of the Anderson dual of the sphere. 
    Consider the spectrum of maps from $\Sigma MT \xi$ to $\Sigma^{2} I\Z$.\footnote{This is the same as spectrum maps from $\Sigma^n MT \xi$ to $\Sigma^{n+1} I\Z$, the space of continuous invertible field theories in \cite[Ansatz 5.26]{freedhopkins}.}
    By the universal property of the Anderson dual (see \cite[Equation (5.17)]{freedhopkins}), $\pi_0$ of this spectrum is non-canonically isomorphic to the direct sum of 
    the torsion subgroup of $\pi_0 \Sigma MT \xi $ and the free part of $\pi_{1} \Sigma MT \xi$.
    Since the resulting group 
    only depends on $\pi_0$ and $\pi_{1}$ of $\Sigma MT \xi$, we have that
    \begin{align*}
    \pi_0 \Map (\Sigma MT \xi, \Sigma^{2} I\Z) 
    &= \pi_0 \Map (\pi_{\geq 0} \Sigma MT \xi, \Sigma^{2} I\Z)
    \\
&=  \pi_0 \Map (\pi_{\geq 0} \Sigma MT \xi, \pi_{\geq 0} \Sigma^{2} I\Z)
    \\
     &= \pi_0 \Map(\| \Bord^\xi_{n-1,n}\|, \sline_{\C}^{cts}).
    \end{align*}
    This finishes the proof.
\end{proof}

\begin{remark}
    It would be interesting to compute the free part of the group of $n$-dimensional continuous invertible TQFTs in examples and realise non-trivial group elements as anomalies of non-unitary quantum field theories.
\end{remark}

\begin{remark}
    Our identification of $\Map (\Sigma MT \xi, \Sigma^{2} I\Z)$ with continuous invertible TQFTs 
as in \cref{def:continuousITQFT} only works on the level of $\pi_0$.
    The underlying reason is that $\Map (\Sigma MT \xi, \Sigma^{2} I\Z)$ is expected to be given by \emph{fully extended} continuous invertible field theories \cite[Ansatz 5.26]{freedhopkins}, while our definition of a TQFT is non-extended.
\end{remark}

\appendix

\section{\texorpdfstring{$\xi$}{xi}-structures on vector bundles}
\label{app:defOfXiStructuresOnVectBundles}

\subsection{Manifolds with \texorpdfstring{$\xi$}{xi}-structures and \texorpdfstring{$\xi$}{xi}-diffeomorphism}

This Appendix continues to develop the theory of $\xi$-manifolds from \cref{subsec:TangStruc} on pages \pageref{subsec:TangStruc}-\pageref{section:k-orientability}.

\begin{lemma}
\label{lm:stableEquiv}
Let $E \to X$ be a $k$-dimensional real vector bundle.
    There is a homotopy equivalence between $\xi_n$-structures on $E \oplus \underline{\R}^{n-k}$ and $\xi_{k}$-structures on $E$.
\end{lemma}

\begin{proof}
    This follows from the universal property of the homotopy pullback:
    \[
    \begin{tikzcd}
        X \ar[dr, dashed,"\ol{c}"description] \ar[drr,dashed,"c"description, bend left=15] \ar[ddr,"E", bend right=15,swap] & & 
        \\
        & B_k \ar[d,"\xi_k"] \ar[r] & B_n \ar[d,"\xi_n"]
        \\
        & BO_k \ar[r,"\oplus \underline{\R}^{n-k}"] & BO_n
    \end{tikzcd}
    \]
    i.e.\! the maps $\bar{c}$ fitting in the diagram are in one-to-one correspondence with maps $c$ fitting in the diagram.
\end{proof}

\begin{remark}
In principle, given a structure $\xi_n\colon B_n \to BO_n$ that is only defined up to dimension $n$, one could consider a stabilisation $\xi_{n+1}'$ by composing with the canonical map $BO_n \to BO_{n+1}$. Note that for an $(n+1)$-manifold to have a $\xi_{n+1}' = (\xi_n \oplus \underline{\R})$-structure, we need the tangent bundle to be isomorphic to the direct sum of $\underline{\R}$ and an $n$-dimensional bundle with $\xi_n$-structure. 
Moreover, if we take the pullback of the diagram 
\[
\begin{tikzcd}
B'_n \arrow[d,"\xi_n'"] \arrow[r] & B_{n} \arrow[d, "\xi_{n+1}'"] \\
BO_n \arrow[r]                   & BO_{n+1}      
\end{tikzcd}
\]
then $B'_n $ is typically not homotopy equivalent to $B_n$, because $BO_n \to BO_{n+1}$ is not a homotopy equivalence.
In particular, $\xi_n$-structures on $n$-manifolds $M$ will not correspond to $\xi_{n+1}$-structures on $TM \oplus \underline{\R}$.
Concretely, a $\xi_n'$-structure on an $n$-dimensional vector bundle $E$ consists of a isomorphism of vector bundles $E \oplus \underline{\R} \cong E' \oplus \underline{\R}$ with an $n$-dimensional vector bundle $E'$ together with a $\xi_n$-structure on $E'$.
This construction can be a useful tool for understanding the cut-and-paste groups we consider (see \cite[section 3.2]{reutterschommerpries}), but will not be studied further in this paper.
    \end{remark}

\begin{remark}[details on the definition of orientation reversal]
\label{rem:orientationreversaldetails}
    Let $\xi_{n+1}\colon B_{n+1}\to BO_{n+1}$ be a tangential structure. Let $M$ be a closed $k$-dimensional $\xi$-manifold, $k\leq n$. 
    The vector bundle $TM \oplus \underline{\R}$ corresponds to the composition $BO_n \to BO_n \times BO_1 \to BO_{n+1}$.
    This composition has a self-homotopy given as follows.
    Consider the self-homotopy of the inclusion of the basepoint of $BO_1$ inducing the generator of $\pi_1 BO_1$.
    Note that this homotopy is induced by the automorphism $-\id_\R$ of the trivial one-dimensional vector bundle over the point.
    This induces a self-homotopy of the map $BO_n \to BO_n \times BO_1$, which is given by the inclusion of the basepoint in the second factor.
We can change our given $\xi_{k+1}$-structure by picking the same map to $B$, but changing the homotopy filling the triangle by the induced self-homotopy of $M \to BO_{n}$.
Since up to homotopy, $\xi_{k+1}$-structures on $TM \oplus \underline{\R}$ correspond to $\xi_{k}$-structures on $TM$ (see Lemma \ref{lm:stableEquiv}), we obtain an operation on manifolds $M$ with $\xi_n$-structures $M \mapsto \overline{M}$ that we will call \emph{orientation reversal}, see \cref{def:orientationreversal}.
Note that the orientation-reversal for a $\xi$-manifold $M$ is only defined if our tangential structure is at least once stabilised with respect to the dimension of the manifold.
\end{remark}

Note that $n$-dimensional manifolds $M$ not only come equipped with a natural map to $BO_n$, but these maps are natural in diffeomorphisms of manifolds in the sense that the isomorphism of vector bundles $df$ induces a homotopy of the following diagram
\[\begin{tikzcd}
        M_1 \ar[rr, "f"] \ar[dr] & & M_2 \ar[dl]
        \\
        & BO_n.  &
    \end{tikzcd}\]    Continuing to follow the logic of including higher coherent homotopies, we thus arrive at the following definition for equivalences between manifolds with $\xi$-structures.

\begin{definition}\label{def:xidiffeoTetrahedron}
Let $M_1,M_2$ be two $n$-dimensional manifolds with $\xi_{n+k}\colon B_{n+k}\to BO_{n+k}$-structure for $k\geq 0$. 
Then a $\xi_{n+k}$-diffeomorphism consists of a diffeomorphism $f\colon M_1 \to M_2$, a homotopy filling the triangle
\[
    \begin{tikzcd}
        M_1 \ar[rr, "f"] \ar[dr] & & M_2 \ar[dl]
        \\
        & B_{n}  &
    \end{tikzcd}
    \]    
and a homotopy between the homotopies filling the tetrahedron
\begin{equation}
    \begin{tikzcd}[row sep={7.2em,between origins}, column sep={9.0em,between origins}, ampersand replacement=\&]
    \&
    B_{n}
    \arrow[rdd, "\xi_{n}", bend left=20]
    \&\\\&
    M_2
    \arrow[u, "\tau_2", dash pattern=on 4.0pt off 4.0pt]
    \arrow[rd, "TM_2"'{name=f13}, dash pattern=on 4.0pt off 4.0pt, swap, pos=0.4]
    \&\\
    M_1
    \arrow[ruu, "\tau_1"{name=f02}, bend left=20]
    \arrow[rr, "TM_1"{name=f03},bend right=20]
    \arrow[ru, "f", dash pattern=on 4.0pt off 4.0pt]
    \&\&
    BO_{n}
    \arrow[from=2-2, to=f02, Rightarrow, shorten=1.0em, pos=0.475]
    \arrow[from=1-2, to=f13, Rightarrow, yshift=-0.25em, xshift=0.75em, shorten=3.5em, pos=0.475]
    \arrow[from=2-2, to=f03, Rightarrow, shorten=2.5em, "df", pos=0.475]
    \arrow[from=1-2, to=f03, bend left=20, yshift=+0.0em, xshift=0.0em, crossing over, Rightarrow, shorten=2.0em, pos=0.425, crossing over clearance=1.5ex]
\end{tikzcd}
\end{equation}
\end{definition}

\begin{remark}
    We obtain the homotopy $df$ above from a diffeomorphism $M_1\xrightarrow{f} M_2$ as follows: A model for $BO_{n}$ is the space of $n$-dimensional subspaces in $\R^{\infty}$. Then any embedding $M_i\hookrightarrow \R^{\infty}$ gives us a map $M_i\to BO_{n}$ by considering the tangent planes as affine planes in $\R^{\infty}$. The space of embeddings $\Emb(M_i,\R^{\infty})$ is contractible.
So for any pair of embeddings $\iota_i\colon M_i\hookrightarrow BO_{n}$, the embeddings $\iota_2f$ and $\iota_1$ are regularly isotopic through a homotopy $df$, which is unique up to a contractible choice. 
\end{remark}

\begin{example}[An orientation reversed manifold is $\xi$-diffeomorphic to the original manifold for $BO$ ]

Consider the tangential structure $B_{n+1} = BO_{n+1}$ with $\xi_{n+1} = \id$.
    Let $M \to BO_n$ be a manifold with its canonical $\xi$-structure.
    By definition, its orientation reversal $\ol{M}$ is the $\xi$-structure which destabilises the self homotopy $H$ of $M \to BO_{n+1}$ given by the reflection in the $(n+1)^{st}$ coordinate.
    Then the identity on $M$ can be made into a $\xi$-diffeomorphism $M \to \ol{M}$.
    Indeed, in the tetrahedron at the $(n+1)^{st}$ level there are two triangles of the form
    \[
    \begin{tikzcd}
        M \ar[d,"\id_M", swap] \ar[r,"TM"] & BO_{n+1}
        \\
        \ol{M} \ar[ru,"TM", swap] & 
    \end{tikzcd}.
    \]
    One of these is filled with the homotopy $H$ by definition of the orientation-reversal, the other one is part of the data of a $\xi$-diffeomorphism.
    Therefore, we are free to choose that triangle to also get filled by $H$.
    Destabilising the resulting strict filling of this tetrahedron shows that $M \cong \ol{M}$ as $\xi$-manifolds.
\end{example}

\begin{example}[The two orientations on a point are not $\xi$-diffeomorphic for $BSO$]\label{ex:twoorientationspoint}
Let $\xi$ be the map $BSO\to BO$. For zero-dimensional manifolds, we have to consider the map $\xi_0\colon BSO_0\rightarrow BO_0$. We have $BSO_0\simeq \{*,*\}$ by \cref{def:homotopyPullbackB_k} and so $\xi_0$ can be taken to be the pointed map from two points into one point. The two orientations on a point are given by the two different lifts.  Let $f\colon *_+\to *_-$ be the constant map between points with different orientations. Then $f$ is not a $\xi_0$-diffeomorphism because the diagram 
\[
\begin{tikzcd}
    *_+\ar[r]\ar[d]&BSO_0\\
    *_-\ar[ru]&
\end{tikzcd}\]
does not commute up to homotopy.
\end{example}

\subsection{Cobordisms with \texorpdfstring{$\xi$}{xi}-structure}

The following result shows that the bordism category given in \cref{def:cob} for a once stabilised structure is reversible in the sense of \cref{def:reversible}.

\begin{proposition}\label{prop:turningBordismsAround}
  Let $\xi_{n+1}\colon B_{n+1} \to BO_{n+1}$ be a tangential structure.
For every $n$-dimensional $\xi$-bordism $M$ from $Y_0$ to $Y_1$, there exists some $\xi$-bordism $M'$ from $Y_1$ to $Y_0$, where $Y_0$ and $Y_1$ are some $(n-1)$-dimensional $\xi$-manifolds.
\end{proposition}
\begin{proof}
    Consider the $(B_n,\xi)$-manifold $M' \coloneqq \overline{M}$, which exists because $B_n$ admits the stabilisation $B_{n+1}$ by assumption.
    Define the decomposition $\partial_{out} M' \coloneqq \partial_{in} M$ and $\partial_{in} M' \coloneqq \partial_{out} M$ of $\partial M'$, which is equal to $\partial M$ as a smooth manifold.
    Let $\phi_1\colon \partial_{out} M \to Y_1$ denote the $\xi_{n-1}$-diffeomorphism, which is part of the data of being a bordism.
    We will now show that this induces a $\xi_{n-1}$-diffeomorphism $\partial_{in} \overline{M} \to \overline{Y}_1$.
    
    Indeed, consider the situation after stabilising twice.
    First of all note that the $(B_{n+1},\xi)$-structure on $T \partial_{in} M \oplus \underline{\R}^2$ would be defined by taking the $(B_n,\xi)$-structure on $TM$, stabilising it once and restricting to $\partial_{in} M$.
    Similarly, the $(B_{n+1},\xi)$-structure on $T \partial_{in} \overline{M} \oplus \R^2$ is defined in the same way, except that we compose the $(B_{n+1},\xi)$-structure on $TM \oplus \underline{\R}$ with $\id_{TM} \oplus -\id_{\underline{\R}}$ to reverse the orientation.
    Therefore, comparing the $(B_{n+1},\xi)$-structure on $T \partial_{in} \overline{M} \oplus \underline{\R}^2$ with the $(B_{n+1},\xi)$-structure on $T \partial_{in} M \oplus \underline{\R}^2$, the only thing changed is that we composed with $\id_{T\partial_{in} M} \oplus \id_{\underline{\R}} \oplus - \id_{\underline{\R}}$.
    Similarly, the $(B_{n+1},\xi)$-structure on the twice stabilised tangent bundle of $\overline{Y}_1$ is the twice stabilised $(B_{n+1},\xi)$-structure on $Y_1$ composed with $\id_{T Y_1} \oplus -\id_{\underline{\R}} \oplus \id_{\underline{\R}}$.
    Since the two vector bundle automorphisms $\id_{\underline{\R}} \oplus - \id_{\underline{\R}}$ and $-\id_{\underline{\R}} \oplus \id_{\underline{\R}}$ are homotopic, composing the $(B_{n+1},\xi)$-structure on $TY_1 \oplus \underline{\R}^2$ with $\id_{T Y_1} \oplus -\id_{\underline{\R}} \oplus \id_{\underline{\R}}$ and $\id_{T Y_1} \oplus \id_{\underline{\R}} \oplus - \id_{\underline{\R}}$ yield equivalent $(B_{n+1},\xi)$-structures.
    Therefore, the vector bundle isomorphism $\partial_{in} \overline{M} \oplus \underline{\R}^2 \to \overline{Y}_1 \oplus \underline{\R}^2$ induced by $\phi_1$ is still compatible with the $(B_{n+1},\xi)$-structures.
    This shows that it defines a $\xi$-diffeomorphism.
    
    Showing that the $\xi_{n-1}$- diffeomorphism $\phi_2\colon \partial_{out} M \to \overline{Y_2}$ induces a $\xi_{n-1}$- diffeomorphism $\partial_{out} \overline{M} \to Y_2$ is analogous.
    This shows that $\overline{M}$ defines a $(B,\xi)$-bordism from $Y_1$ to $Y_0$.
\end{proof}

\begin{corollary}\label{cor:bordreversible}
    For $\xi\colon B_{n+1} \to BO_{n+1}$ a once-stabilised tangential structure, the category $\Cob^\xi_{n-1,n}$ is reversible, in the sense of \cref{def:reversible}.
\end{corollary}

\section{\texorpdfstring{$\SKK$}{SKK} of a category}
\label{ap:ProofPi1ofCategories}

In this Appendix, we will discuss a more abstract perspective on $\SKK$ groups that the authors learned independently of Stephan Stolz and Achim Krause.
This generalises the relation between $\SKK$ and the fundamental group of the cobordism category proved in \cite{bokstedtsvane}.

\subsection{Reversibility and \texorpdfstring{$\SKK$}{SKK} of a category}
If $\mathcal{C}$ is a category, there exists a smallest groupoid that contains $\mathcal{C}$ in which all morphisms are invertible.
More precisely, the \emph{groupoidification} $\hat{\mathcal{C}}$ is the image of $\mathcal{C}$ under the left adjoint to the inclusion $\mathrm{Gpd} \to \mathrm{Cat}$ of the category of groupoids into the category of categories, and it is the universal groupoid receiving a map from $\mathcal{C}$.
Concretely, $\hat{\mathcal{C}}$ is defined to be the category with objects the objects of $\mathcal{C}$ and morphisms 
given by equivalence classes of zigzags of morphisms in $\mathcal{C}$.
Here a zigzag from $Y$ to $Y'$ is a sequence of morphisms of the form
\[\begin{tikzcd}
    Y' = Y_0\ar[r,leftrightarrow,"X_0"]& Y_1\ar[r,leftrightarrow,"X_1"]&\dots\ar[r,leftrightarrow,"X_{n-1}"] &Y_n\ar[r,leftrightarrow,"X_n"]&Y_{n+1}=Y
\end{tikzcd}
\]
where each $X_i$ is either a morphism from $Y_i$ to $Y_{i+1}$ or from $Y_{i+1}$ to $Y_i$.
We quotient by the equivalence relation given by replacing two composable morphisms pointing in the same direction (either left or right) by their composition, and defining $Y \xrightarrow{X} Y'$ to be inverse to $Y' \xleftarrow{X} Y$.
With composition given by concatenation of zigzags, $\hat{\mathcal{C}}$ becomes a groupoid.
The relations imply that if a morphism is invertible in $\mathcal{C}$, then its formal inverse in $\hat{\mathcal{C}}$ is equal to its inverse. Therefore, we can abuse notation and write zigzags as 
\[
X_0^{\epsilon_0} X_1^{\epsilon_1}  \dots X_n^{\epsilon_n}, 
\]
where $\epsilon_i \in \{\pm 1\}$ and $X_i$ is a morphism in $\mathcal{C}$.
Note that the domain of the above morphism is the domain of $X_n$ if $\epsilon_n = 1$ and the codomain of $X_n$ if $\epsilon_n = -1$.

Given a groupoid $\mathcal{G}$, a classifying space construction gives a space $B \mathcal{G}$. 
An object $x \in \mathcal{G}$ gives a point in $B\mathcal{G}$ and $\pi_1(B \mathcal{G},x) = \End_{\mathcal{G}}(x)$.
From now on we will consider \emph{pointed categories}, i.e. we fix an object $1 \in \mathcal{C}$ that we consider as a basepoint.
If $\mathcal{C}$ is monoidal (such as the bordism category) we take $1$ to be the monoidal unit.
Our goal is to give a concrete description of the group
$\pi_1 (B\hat{\CC},1) \cong \End_{\hat{\mathcal{C}}} (1) $
under some mild assumptions.

Note that there is a monoid homomorphism
\[
\End_\mathcal{C}(1) \to \End_{\hat{\mathcal{C}}}(1).
\]
Since the latter is a group, this induces a group homomorphism 
\[
\phi\colon \Gr (\End_\mathcal{C}(1)) \to \End_{\hat{\mathcal{C}}}(1)
\]
from the nonabelian Grothendieck group of the nonabelian monoid $\End_\mathcal{C}(1)$.
Concretely, this Grothendieck group has elements of the form $X_0^{\epsilon_0} X_1^{\epsilon_1}  \dots X_n^{\epsilon_n}$, where $X_i \in \End_\mathcal{C}(1), \epsilon_i \in \{\pm 1\}$ and $X_i^{-1}$ denotes the formal inverse.
Observe that a general element $X_0^{\epsilon_0} X_1^{\epsilon_1} \dots X_n^{\epsilon_n}$ of $\End_{\hat{\mathcal{C}}}(1)$ need not be in $\Gr (\End_\mathcal{C}(1))$ since the $X_i$ need not have source and target equal to $1$.

\begin{example}
\label{ex:groupoidifycounterex}
    Let $\mathcal{C}$ consist of two objects $1$ and $Y$ and a single non-trivial morphism $X\colon 1 \to Y$.
    Then $\End_{\mathcal{C}} (1)$ and $\End_{\hat{\mathcal{C}}}(1)$ are trivial, and so the induced group homomorphism $\phi$ is a map between trivial groups.
\end{example}

\begin{example}
    Let $\mathcal{C}$ consist of two objects $1$ and $Y$ and two parallel morphisms $X_1,X_2\colon 1 \to Y$.
    Then $\End_{\mathcal{C}} (1) = 1$, but $\End_{\hat{\mathcal{C}}}(1)$ is a free group generated by $X_1^{-1} X_2$. 
    So the induced group homomorphism $\phi$ is not surjective.
\end{example}

The condition we will require on $\mathcal{C}$ in order to compare $\End_{\mathcal{C}} (1)$ and $\End_{\hat{\mathcal{C}}}(1)$ is reversibility of morphisms:

\begin{definition}\label{def:reversible}
    Let $(\mathcal{C}, 1)$ be a pointed category.
        We say that $(\mathcal{C}, 1)$ is \emph{reversible} with respect to $1$ if for any morphism $X\colon 1 \to Y$ there exists a morphism $X'\colon Y \to 1$.
\end{definition}

The following is an example of a non-reversible category.

\begin{example} 
\label{ex:UnstableFramingNotReversible}
    Let $\mathcal{C}$ be the framed bordism category $\Cob_{2,1}^{B_2 = *}$ in which morphisms are unstably framed surfaces as bordisms between $1$-dimensional manifolds with a framing of their once stabilised tangent bundle, compare \cref{rm:BordantIsNotSymmetric}.

    Then this category is not reversible.  Indeed, consider for $g>1$ a  genus $g$  surface with one boundary component $\Sigma^1_g$. Because it is homotopy equivalent to a one-dimensional CW-complex, this surface has an unstable framing. 
    Consider $\Sigma^1_g$ as a bordism $\varnothing\to (S^1,f)$, where $f$ is the once stabilised framing of $S^1$ induced by restricting the framing of $\Sigma^1_g$. Then there is no framed bordism $(S^1,f) \to\varnothing$. For assume there was such a bordism $\Sigma^1_{g'}$. We could then form the composition $\Sigma^1_g \cup_{S^1} \Sigma^{1}_{g'}$, a framed surface of genus $g+g'>1$, which is not possible. 

    This gives another proof that $B_{2} = *$ cannot be stabilised, see \cref{lem:UnstableFramingStructureCannotBeStabilised}.
\end{example}

\begin{remark}
    In \cite{KST}, it is shown that for every $n > 2$ and every $\xi\colon B \to BO_n$ the bordism category $\Cob^\xi_{n-1,n}$ is reversible at $\varnothing$.
\end{remark}

Note that if $(\mathcal{C}, 1)$ is reversible and $Y_1$ and $Y_2$ are both connected to $1$ by some zigzag, then there exists a morphism $Y_1 \to Y_2$.
The proof of the following Lemma was communicated to the second author by Stephan Stolz, also see \cite[Proposition 3.2]{juer2013localisations}.

\begin{lemma}
\label{lm:phisurj}
    If $(\mathcal{C}, 1)$ is reversible, then $\phi$ is surjective.
\end{lemma}
\begin{proof}
    Let $X_0^{\epsilon_0} X_1^{\epsilon_1} \dots X_n^{\epsilon_n} \in \End_{\hat{\mathcal{C}}}(1)$, where $\epsilon_i \in \{\pm 1\}$ and $X_i$ is a morphism in $\mathcal{C}$.
    By composing morphisms that are composable in $\mathcal{C}$ we can assume without loss of generality that $\epsilon_i \neq \epsilon_{i+1}$ for all $i$.
    We will perform an induction on the number of morphisms that do not have domain and codomain equal to $1$.
    Suppose $X_0, \dots, X_{i-1} \in \End_{\mathcal{C}}(1)$ for some $i \geq 0$.
    Assume first that $\epsilon_i = 1$ so that $\epsilon_{i+1} = -1$ and $X_i$ is a morphism from some $Y$ to $1$.
    Let $X'_i$ be a morphism from $1$ and $Y$ so that $X_i X_i' \in \End_\mathcal{C} (1)$.
Then
\begin{align*}
X_0^{\epsilon_0} X_2^{\epsilon_2} \dots X_n^{\epsilon_n} &= X_0^{\epsilon_0} X_2^{\epsilon_2} \dots X_{i-1}^{\epsilon_{i-1}} (X_i X_i') (X_{i+1} X_{i}')^{-1}X_{i+2}^{\epsilon_{i+2}} \dots  X_n^{\epsilon_n}
\end{align*}

has one less occurrence of an object different from $1$.
    We can do a similar computation if $\epsilon_i = -1$ when $X_i$ is a morphism from $1$ to $Y$ by taking $X_i'$ to go from $Y$ to $1$.
\end{proof}

Under the above assumption, we can ask what the kernel of $\phi$ is to get an explicit description of $\End_{\hat{\mathcal{C}}} (1)$ as a quotient group of the Grothendieck group of $\End_\mathcal{C} (1)$.
It turns out the kernel is generated by a kind of $\SKK$ relation:

\begin{definition}\label{def:chimaeraRelations}
    Given a specified basepoint $1 \in \operatorname{ob} \mathcal{C}$ the \emph{SKK group of $\mathcal{C}$}, $\SKK(\mathcal{C}, 1)$,
    is the Grothendieck group of the monoid 
    $\End_{\mathcal{C}} (1)$ modulo the so-called chimaera relations saying that
    \[
    (X_1' \circ X_2)^{-1} \circ (X_1' \circ X_1) \sim (X_2' \circ X_2)^{-1} \circ (X_2' \circ X_1)
    \]
    for all $X_1,X_2 \in \Hom_\mathcal{C}(1,Y)$ and $X_1',X_2' \in \Hom_\mathcal{C}(Y,1)$.

\end{definition}

The alternative chimaera relation
\[(X_1' \circ X_1) \circ (X_2' \circ X_1)^{-1} \sim  (X_1' \circ X_2) \circ (X_2' \circ X_2)^{-1}\]
is sometimes added in the literature, which is an easy consequence of the above one.

In the specific situation where $\mathcal{C} = \Cob^\xi_{n-1,n}$ is the bordism category with $\xi_n$ tangential structure and $1 = \varnothing$ is the monoidal unit, we have that $\SKK(\Cob^\xi_n, 1) = \SKKxi_n$.
Indeed, note that the chimaera relation exactly corresponds to the alternative SKK relation in \cref{prop:SecondDefOfTheSKKRelation}.

\begin{lemma}\label{lm:induced}
The map $\phi$ induces a map
\[
\SKK(\mathcal{C}, 1) \to \End_{\hat{\mathcal{C}}}(1).
\]
\end{lemma}
\begin{proof}
We have to show that the chimaera relation is in the kernel of $\phi$.
This follows by the computation
\begin{align*}
    (X_1' \circ X_2)^{-1} \circ X_1' \circ X_1 &= X_2^{-1} \circ (X_1')^{-1} \circ X_1' \circ X_1 
    \\
    &= X_2^{-1} \circ X_1 = X_2^{-1} \circ (X_2')^{-1} \circ X_2' \circ X_1 
    \\
    &= (X_2' \circ X_1)^{-1} \circ (X_2' \circ X_2)
\end{align*}
in $\hat{\mathcal{C}}$. 
\end{proof}

\begin{theorem}
\label{th:SKKispi1ofcob}
Suppose $(\mathcal{C},1)$ is reversible.
    Then $\phi$ induces an isomorphism
    \[
    \SKK(\mathcal{C}, 1) \cong \End_{\hat{\mathcal{C}}}(1) \cong \pi_1 (\|\mathcal{C}\|).
    \]
\end{theorem}

We omit the proof of \cref{th:SKKispi1ofcob}, which can be shown via adaptations of the techniques in \cite{bokstedtsvane}.

\begin{remark}
    Example \ref{ex:groupoidifycounterex} shows that there are cases in which the conclusion of the above theorem holds, but the assumption of reversibility of arrows from the basepoint does not.
    We do not know the weakest possible assumption for which the $\SKK$ group of $\mathcal{C}$ is isomorphic to $\End_{\hat{\mathcal{C}}}(1).$
    However, note that generalizing Example \ref{ex:groupoidifycounterex}, we could allow morphisms $X\colon 1 \to Y$ for which there is no morphism $Y \to 1$ as long as every zigzag from $1$ to $Y$ in $\hat{\mathcal{C}}$ is equal to $X$.

    For another example of a non-reversible category for which $\phi$ is an isomorphism, consider $\mathcal{C} = \Cob_{1,2}^{B_2 = *}$ of \cref{ex:UnstableFramingNotReversible}.
    By \cite{GMTW}, $\pi_1$ of $\Cob_{1,2}^{B_2 = *}$ is the stably framed bordism group in dimension two.
 This is the second stable stem, which is $\Z/2$ generated by the torus with the Lie group framing.
 Independently, it was shown in \cite{lorant} that for $B_2=*$, which corresponds to $\Spin^r$ for $r=0$, it holds that
 $\SKK^{B_2 = *}_2 \cong \Z/2$ is generated by the same element, from which it also follows that $\phi$ is an isomorphism\footnote{In \cite{lorant}, Szegedy falsely claims that any rigid symmetric monoidal category is reversible in order to deduce an isomorphism $\pi_1(\| \Cob^{\Spin^r}_2 \|) \cong \SKK^{\Spin^r}_2$. However, his computation of $\SKK^{\Spin^r}_2$ does not use this isomorphism. This also applies to \cref{rem:spinr}.}.
\end{remark}

\subsection{Geometric realisations of \texorpdfstring{$\infty$}{infinity}-categories}
\label{sec:highercats}

We provide a further abstract setting that will be useful to compare with the analogous $\infty$-categorical setting.
Consider the diagram of $(\infty,1)$-categories
\begin{equation}
\label{eq:infinitysquare}
\begin{tikzcd}[column sep =30,row sep=30]
\Gpd_1 \ar[r, hookrightarrow,"N"] \ar[d, hookrightarrow] & \Gpd_\infty \ar[d, hookrightarrow] \ar[l, bend right, "\pi_{\leq 1}", swap]
\\
    \Cat_1 \ar[r, hookrightarrow, "N"] \ar[u, bend right,"\hat{(.)}",pos=0.55,swap] & \Cat_\infty \ar[u, bend right,"\| .\|", swap] \ar[l, bend right, "\operatorname{ho}",swap]
\end{tikzcd},
\end{equation}
where $\Gpd_1$ denotes the $(2,1)$-category of groupoids, $\Gpd_\infty$ the $(\infty,1)$-category of spaces (also known as $\infty$-groupoids), $\Cat_1$ the $(2,1)$-category of categories and $\Cat_\infty$ the $(\infty,1)$-category of $(\infty,1)$-categories.
We have written down the obvious fully faithful inclusions between them making the square commute and all inclusions are reflective.
Their left adjoints, given by the $1$-categorical and the $\infty$-categorical version of groupoidification and the homotopy category.
If we are working with the model in which $(\infty,1)$-categories are quasi-categories and $\infty$-groupoids Kan complexes, then the inclusion $\Cat_1 \hookrightarrow \Cat_\infty$ is given by the nerve and the $\infty$-groupoidification $\| .\|\colon\Cat_\infty \to \Gpd_\infty$ by the geometric realisation.
The homotopy category of an $\infty$-groupoid represented by a Kan complex is given by its fundamental groupoid. 

We can realise the $\SKK$ group of a reversible category as the fundamental group of the geometric realisation of its nerve using the following lemma:

\begin{lemma}
\label{lem:groupoidfundamantalgroup}
Let $\mathcal{C}$ be a category with basepoint $1$.
Then $\End_{\hat{\mathcal{C}}}(1)$ is the fundamental group of $\| N \mathcal{C} \|$ at the basepoint $1 \in \mathcal{C}$. 
\end{lemma}
\begin{proof}
First, note that $N(\hat{\mathcal{C}}) = \| N(\mathcal{C})\|$.
Indeed, going through square \ref{eq:infinitysquare} from the southwest to the northeast corner is independent of which of the two paths one takes by uniqueness of adjoints.   
    If $\mathcal{G} \in \Gpd_1$ is a groupoid with basepoint $1 \in \obj \mathcal{G}$, then $\End_{\mathcal{G}}(1)$ agrees with $\pi_1$ based at $1$ of $N\mathcal{G} \in \Gpd_\infty$.
It then also follows by the commutativity of the square \ref{eq:infinitysquare}.
\end{proof}

\begin{definition}
    If $\mathcal{C}$ is an $\infty$-category, we define 
    \[
    \SKK(\mathcal{C}) \coloneqq \SKK(\operatorname{ho} \mathcal{C}).
    \]
\end{definition}

\begin{lemma}
    Let $\mathcal{C}$ be an $\infty$-category such that $\operatorname{ho} \mathcal{C}$ is reversible at $1 \in \mathcal{C}$.
    Then 
    \[
    \SKK(\mathcal{C}) \cong \pi_1 (\|\mathcal{C}\|, 1).
    \]
\end{lemma}
\begin{proof}
    By \cref{th:SKKispi1ofcob}, it suffices to show that
    \[
    \pi_1 (\|\mathcal{C}\|, 1) = \End_{\widehat{\operatorname{ho} \mathcal{C}}}(1).
    \]
    It follows by \cref{lem:groupoidfundamantalgroup} that 
    \[
    \End_{\widehat{\operatorname{ho} \mathcal{C}}}(1) = \pi_1( \|N(\widehat{\operatorname{ho} \mathcal{C}})\|, 1).
    \]
    Note that going through square \ref{eq:infinitysquare} from the southeast to the northwest corner is independent of which of the two paths one takes, because left adjoints are unique.
    It follows that $\pi_{\leq 1}( \|\mathcal{C}\|) = \pi_{\leq 1}(\|N( \widehat{\operatorname{ho} \mathcal{C}})\|)$.
\end{proof}

We have shown in \cref{cor:bordreversible} that if $\xi_n\colon B_{n} \to BO_n$ admits a single stabilisation, then $\Cob^{\xi_n}_{n-1,n}$ is a reversible category at every basepoint.
In particular, we recover the original result of \cite{bokstedtcomputations}:

\begin{corollary}
    Let $\xi\colon B_{n+1} \to BO_{n+1}$ be a tangential structure.
    Then $$\pi_1 (\|\Bord^\xi_{n-1,n}\|) \cong \pi_1 (\|N \Cob^\xi_{n-1,n}\|) 
    \cong \SKK(\Cob^\xi_{n-1,n}) = \SKK^\xi_n.$$
\end{corollary}

\begin{remark}
    The higher homotopy groups of $\|\Bord^\xi_{n-1,n}\|$ and  $\|N \Cob^\xi_{n-1,n}\|$ are in general different.
\end{remark}

\bibliography{biblio.bib}{}
\bibliographystyle{amsalpha}

\end{document}